\documentclass[aap,preprint]{imsart}

\RequirePackage{amsthm,amsmath,amsfonts,amssymb,mathrsfs,stmaryrd,stackrel}
\RequirePackage[numbers]{natbib}
\RequirePackage[colorlinks,citecolor=blue,urlcolor=blue]{hyperref}
\RequirePackage{graphicx,color,epsfig,subfigure,cancel}
\RequirePackage{bm,bbm}

\startlocaldefs

\newtheorem{theorem}{Theorem}[section]
\newtheorem{proposition}{Proposition}[section]
\newtheorem{assumption}{Assumption}
\newtheorem{lemma}[theorem]{Lemma}

\theoremstyle{remark}

\newtheorem{definition}[theorem]{Definition}

\newtheorem*{property1}{Property 1}
\newtheorem*{property2}{Property 2}

\newcommand{\Prob}[1]{\mathbb{P}\left[#1\right]}

\newcommand{\Exp}[1]{\mathbb{E}\left[#1\right]}
\newcommand{\Var}[1]{\mathbb{V}\left[#1\right]}

\newcommand{\dd}{\mathrm{d}}

\endlocaldefs

\begin{document}

\begin{frontmatter}
\title{Global solutions with infinitely many blowups in a mean-field neural network}
\runtitle{Infinite number of blowups in mean-field}

\begin{aug}

\author[A]{\fnms{Lorenzo} \snm{Sadun}\ead[label=e2]{sadun@math.utexas.edu}}
\and
\author[A,B]{\fnms{Thibaud} \snm{Taillefumier}\ead[label=e1]{ttaillef@austin.utexas.edu}}
\address[A]{Department of Mathematics,
University of Texas, Austin,
}

\address[B]{Department of Neuroscience,
University of Texas, Austin,
}
\end{aug}

\begin{abstract}
We recently introduced idealized mean-field models for networks of integrate-and-fire neurons with impulse-like interactions---the so-called delayed Poissonian mean-field models.
Such models are prone to blowups: for a strong enough interaction coupling, the mean-field rate of interaction diverges in finite time with a finite fraction of neurons spiking simultaneously.
Due to the reset mechanism of integrate-and-fire neurons, these blowups can happen repeatedly, at least in principle.
A benefit of considering Poissonian mean-field models is that one can resolve blowups analytically by mapping the original singular dynamics onto uniformly regular dynamics via a time change.
Resolving a blowup then amounts to solving the fixed-point problem that implicitly defines the time change, which can be done consistently for a single blowup and for nonzero delays. 
Here we extend this time-change analysis in two ways: 
First, we exhibit the existence and uniqueness of explosive solutions with a countable infinity of blowups in the large interaction regime.
Second, we show that these delayed solutions specify ``physical'' explosive solutions in the limit of vanishing delays, which in turn can be explicitly constructed.
The first result relies on the fact that blowups are self-sustaining but nonoverlapping in the time-changed picture.
The second result follows from the continuity of blowups in the time-changed picture and incidentally implies the existence of periodic solutions.
These results are useful to study the emergence of synchrony in neural network models.
\end{abstract}



\end{frontmatter}


\section{Introduction}


\subsection{Background} 

In this work, we consider idealized neural-network models called delayed Poissonian mean-field (dPMF) models introduced in \cite{TTPW}.
These dPMF models are variations of classical mean-field models \cite{Brunel:1999aa,Brunel:2000aa,Caceres:2011}, whose dynamics are also prone to blowups.
From a modeling perspective, blowups correspond to the occurrence of synchronous, macroscopic spiking events in a neural network.
Understanding the emergence of these synchronous events is of interest for studying the maintenance of precise temporal information in neural networks \cite{Panzeri:2010,Kasabov:2010,Brette:2015}.
However, blowups resist direct analytical treatment in classical mean-field models \cite{Delarue:2015,Delarue:2015b,Hambly:2019,Nadtochiy:2019,Nadtochiy:2020}.
This observation is the primary motivation justifying the introduction of dPMF dynamics, whose blowups prove analytically tractable.

In \cite{TTPW}, we conjectured dPMF dynamics as the mean-field limit of finite-size particle systems with singular interactions (see Fig. \ref{fig:FiniteNetwork}).
In a finite-size system of $N$ particles, each particle $i$, $1\leq i \leq N$, represents a neuron with time-dependent state variable  $X_{N,i,t}$.
These  state variables evolve jointly according to a $\mathbbm{R}^N$-valued continuous-time process $t \mapsto \lbrace X_{N,i,t} \rbrace_{1 \leq i \leq N}$.
Specifically, the network dynamics is parametrized by the drift value $\nu>0$, the refractory period $\epsilon>0$, and the interaction parameter $\lambda>0$ as follows:
$(i)$ Whenever a process $X_{N,i,t}$ hits the spiking boundary at zero, it instantaneously enters its inactive refractory state.
$(ii)$ At the same time, all the other active processes $X_{N,j,t}$ (which are not in the inactive refractory state) are respectively updated by amounts $-w_{N,ij,t}$, where $w_{N,ij,t}$ is independently drawn from a normal law with mean and variance equal to $\lambda/N$.
$(iii)$ After an inactive (refractory) period of duration $\epsilon>0$, the process $X_{N,i,t}$ restarts its autonomous stochastic dynamics from the reset state $\Lambda>0$.
$(iv)$ In between spiking/interaction times, the autonomous dynamics of active processes follow independent drifted Wiener processes with negative drift $-\nu$.
Correspondingly, an initial condition for the network is specified by the starting values of the active processes, i.e., $X_{N,i,0}>0$ if $i$ is active, and the last inactivation time of the inactive processes,  i.e., $-\epsilon  \leq \rho_{i,0} \leq 0$, if $i$ is inactive.
Following the above definition, the finite-size versions of dPMF dynamics exhibit two key features (see Fig. \ref{fig:FiniteNetwork}):
$(1)$ recurrent interactions implement self-excitation whose strength is quantified by the interaction parameter $\lambda$; $(2)$ owing to the post-spiking refractory period $\epsilon>0$, individual neuronal processes follow delayed dynamics.
Parenthetically, the inclusion of a negative drift $\nu$ is necessary to ensure that neurons will spike in finite time, independent of the initial conditions and of the coupling strength.

By analogy with  \cite{Caceres:2011,Carrillo:2013}, dPMF dynamics are deduced from finite-size ones by conjecturing propagation of chaos in the infinite-size limit $N \to \infty$, which is supported by numerical simulations (see Fig. \ref{fig:sim}).
The propagation of chaos states that for exchangeable initial conditions, the processes $X_{N,i,t}$, $1 \leq i \leq N$, become i.i.d. in the limit of infinite-size networks $N \to \infty$, so that each individual process follows a mean-field dynamics \cite{Sznitman:1989}.
For exchangeable initial conditions, a representative dPMF process $X_t=\lim_{N \to \infty} X_{N,i,t}$ is governed by the deterministic cumulative drift $\Phi$ that amalgamates the contribution of the autonomous drift $-\nu$ and the mean-field contribution of neuronal interactions via  $\lambda$.
Specifically, we have $\Phi(t) = \nu t + \lambda \Exp{ M_t}$, where the process $M_t$ counts the successive first-passage times of the representative process $X_t$ to the zero spiking threshold.
Such a process is defined as
\begin{eqnarray}\label{eq:MFM}
M_t = \sum_{n>0} \mathbbm{1}_{[0,t]}(\rho_n) \, ,\quad \mathrm{with} \quad \rho_{n+1} = \inf \left\{  t>r_n=\rho_{n-1}+\epsilon \, \big  \vert \,  X_{t}  \leq 0 \right\}  \, ,
\end{eqnarray}
with initial conditions bearing on $X_0$ or $\rho_0$ depending on the representative process being initially active or inactive.
We will later specify such initial conditions in detail.
Note that in the above definition,  the first-passage times $\rho_n$, $n \geq 1$,  denote successive inactivation times, whereas $r_n=\rho_n+\epsilon$, $n >0$, denotes the corresponding sequence of reset times.
For mediating all interactions, the function $t \mapsto \Exp{M_t}$ plays a central role in our analysis. 
We denote this function by $F$ and observe that $F$ is defined as an increasing \emph{c\`adl\`ag} function.
Remember that a function is said to be \emph{c\`adl\`ag} if it is right-continuous with left limits.
Then, the hallmark of dPMF dynamics is that the drift $\Phi$ impacts the representative neuron stochastically via a process $Z_t = \Phi(t)+W_{\Phi(t)}$, where $W$ is a canonical, driving Wiener process.
Accordingly, dPMF dynamics are referred to as Poissonian because the process $Z_t$ represents the diffusive limit of a Poisson counting process with $\Exp{Z_t}=\Var{Z_t}$ for all $t \geq 0$.
With these conventions, the stochastic dPMF dynamics of a representative process is given by
\begin{eqnarray}\label{eq:stochEqMF}
X_t = X_0 - \int_0^t \mathbbm{1}_{ \{ X_{s^-} >0 \} } \dd Z_s + \Lambda M_{t-\epsilon} \, .
\end{eqnarray}
The above equation fully defines dPMF dynamics:
The integral term indicates that when the neuron is active, for $X_t>0$, its state evolves according to the Poissonian diffusion process $Z_t$.
The delayed term indicates that after hitting the spiking boundary at zero at time $\tau$, the process remains inactive until it resets at $\Lambda$ after a duration $\epsilon$: $X_t=0$ for $\tau \leq t < \tau+\epsilon$.

Because of their self-exciting nature, dPMF dynamics are prone to blowups/synchronous events for large enough interaction parameter and/or for initial conditions that are concentrated near the zero spiking boundary.
Blowups occur at those times when $f$, the instantaneous spiking rate of a representative neuron, diverges.
Formally, the firing rate $f$ is defined in the distribution sense as the Radon-Nikodym derivative of $F$ with respect to the Lebesgue measure: $f = \dd F /\dd t$.
Then, assuming $f$ to be finite in the left vicinity of $T_1$,  $T_1$ is a blowup time if $\lim_{t \to T_1^-} f(t)=\infty$.
In turn, synchronous events occur at those times $T_1$ for which a finite fraction of processes spikes simultaneously or equivalently, for which a representative neuron spikes with nonzero probability: $\pi_1= \Prob{X_{T_1}=0}>0$.
Formally, this corresponds to $F$ admitting a jump discontinuity at $T_1$ so that $F(T_1)-F(T_1^-)=\lambda \pi_1$.

\begin{figure}[htbp]
\begin{center}
\includegraphics[width=\textwidth]{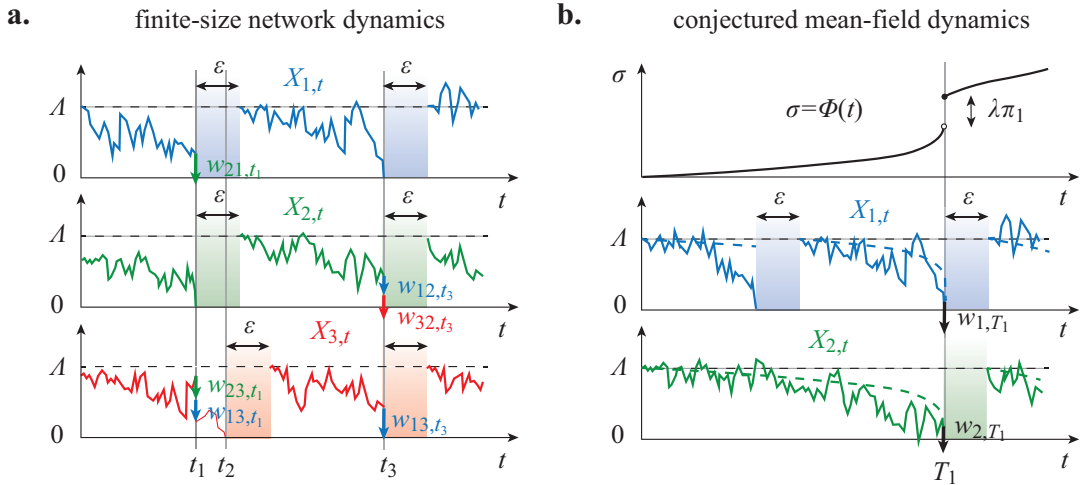}
\caption{
{\bf Delayed Poissonian network dynamics.}
{\bf a.} Schematic representation of the delayed Poissonian network dynamics in a finite-size particel system with $N=3$ interacting processes.
In between spiking events, neuronal dynamics are that of independent Wiener processes with negative drift $-\nu<0$.
Neuron $2$ spikes and inactivates at time $t_1$, which leads to updating neuron $1$ and $3$ via weights $w_{21,t_1}$ and $w_{23,t_1}$ i.i.d. with normal law $\mathcal{N}(\lambda/N,\lambda/N)$.
This causes neuron $1$ to spike and inactivate, leading to another instantaneous update of neuron $3$ via a random weight $w_{13,t_1}$.
Then, neuron $3$ spikes and inactivates at a later time $t_2$, but this has no impact on neurons $1$ and $2$ as they remain inactive during a refractory period of duration $\epsilon>t_2-t_1$.
At the end of their refractory period, neurons reset at $\Lambda>0$.
Finally, the spontaneous spiking of neuron $1$ in $t_3$ cause both neurons $2$ and $3$ to spike at the same time.
{\bf b.} Schematic representation of the conjectured delayed Poissonian mean-field (dPMF) dynamics in the limit of infinite-size networks $N \to \infty$.
When $N \to \infty$, it is conjectured that the neuronal dynamics become independent and that interactions are mediated by the mean-field deterministic cumulative drift $\Phi(t)=\nu t + \lambda F(t)$, where the increasing function $F$ denotes the cumulative spiking or inactivation rate.
Specifically, it is conjectured that neuronal dynamics follow independent time-changed Wiener processes with Poissonian attributes parametrized by $\Phi$.
For time $t<T_1$, the (Radon-Nikodym) derivative of $\Phi$ remains finite and the spiking of representative processes are determined via first-passage time problems with regular boundary.
Accordingly, two representative processes have zero probability to spike synchronously.
However, for large enough interaction parameter $\lambda$, it is possible that the (Radon-Nikodym) derivative of $\Phi$ locally diverges and that a finite fraction $\pi_1$, $0<\pi_1<1$, of the processes synchronously spike as depicted in $T_1$.
We refer to such a possibility as a full-blowup event.
This corresponds to a jump discontinuity of size $\lambda \pi_1$ in $\Phi$ and to singular stochastic updates of the representative processes with weight $w_{1,T_1}$ and $w_{2,T_1}$ i.i.d. normal law $\mathcal{N}(\lambda \pi_1,\lambda \pi_1)$.
}
\label{fig:FiniteNetwork}
\end{center}
\end{figure}

The main interest of considering dPMF dynamics is that blowups and synchronous events are analytically tractable under some reasonable restrictions about the initial conditions.
Specifically, denoting by $\mathcal{M}(I)$ the set of positive measure on an interval $I \subset \mathbbm{R}$, we assume as in \cite{TTPW} that:

\begin{assumption}\label{assump1}
The initial conditions for dPMF dynamics is specified by 
\begin{eqnarray}
\Prob{X_0 \in \dd x \vert X_0 >0}= p_0(\dd x) \quad \mathrm{and} \quad \Prob{ \rho_0 \in \dd t \vert X_0 = 0} = f_0(\dd t) \, ,  \nonumber
\end{eqnarray}
where  $(p_0,f_0)$ in $ \mathcal{M}((0,\infty))\times \mathcal{M}([-\epsilon,0)) $ is normalized, i.e., $\Vert p_0 \Vert_1 +\Vert f_0 \Vert_1=1$ and such that $p_0$ is locally differentiable in zero with $\lim_{x \to 0^+} p_0(x)=0$ and $\lim_{x \to 0^+} \partial_x p_0(x)/2 < 1/\lambda$.
\end{assumption}

The above initial conditions guarantee that there exists a dPMF dynamics locally solving \eqref{eq:stochEqMF} and that this dynamics is initially smooth in the sense that $F$ is an infinitely differentiable function in the right vicinity of zero.
Such a smooth solution can be maximally continued on a possibly infinite interval $[0,T_1)$, where $T_1$ marks the occurrence of the first blowup.
In \cite{TTPW}, we show that for all $0<t<T_1$, the density $x \mapsto p(x,t)= \Prob{X_t \in \dd x \, \vert X_t > 0}/\dd x$ is smooth in the right vicinity of zero with
$\lim_{x \to 0^+} p(x)=0$ and $\lim_{x \to 0^+} \partial_x p(t,x)/2 < 1/\lambda$.
As the later limits are always well defined and finite, we will simply refer to their values as $p(0)$ and $\partial_x p(t,0)/2$ for conciseness.
With this in mind, we show in \cite{TTPW} that for all $0<t<T_1$, the instantaneous spiking rate is given by
\begin{eqnarray}
f(t) = \frac{ \nu \partial_x p(0,t)}{2  - \lambda \partial_x p(0,t)}\,.  \nonumber
\end{eqnarray}
This leads to introducing the following criterion for first-blowup times:

\begin{definition}\label{def:blowup_int}
Under Assumption \ref{assump1}, the first blowup time is defined as
\begin{eqnarray}  \nonumber
T_1 = \sup \{  t> 0 \, \vert \, \partial_x p(t,0)  < 2/\lambda \} \, .
\end{eqnarray}
\end{definition}

The above criterion only bears on the local divergence of  the firing rate $f$  and is silent about the possible occurrence of a synchronous event.
To further characterize blowups, we introduce in  \cite{TTPW} the so-called full-blowup criterion:

\begin{assumption}\label{assump2}
At the blowup time $T_1$,  the density function $x \mapsto p(x,t)$ satisfies
\begin{eqnarray}  \nonumber
\lim_{t \to T_1^-} \partial^3_x p (t,0) > 8/\lambda \, .
\end{eqnarray}
\end{assumption}

The full-blowup criterion, which generically holds with respect to the choice of initial conditions, allowed us to characterize blowups in dPMF dynamics as follows:

\begin{proposition}\label{def:blowup_int}
Under Assumptions \ref{assump1} and \ref{assump2}, the instantaneous rate  blows up as $f(t) \sim 1/\sqrt{T_1-t}$ when $t \to T_1^-$ and triggers a synchronous event in $T_1$ of size 
\begin{eqnarray}\label{eq:selfProb}
\pi_1 = \inf \left\{ p>0 \,  \vert \, p>\Prob{\tau_{X_{T_1}} < \lambda p } \right\} > 0 \, ,
\end{eqnarray}
where $\tau_{X_{T_1}}$ denotes the first-passage time in zero of a Wiener process with unit negative drift and started with random initial condition $X_{T_1^-}$.
\end{proposition}

Thus, for generic initial conditions, blowups correspond to left H\"older singularity of exponent $1/2$ in the cumulative function $F$, followed by a jump discontinuity whose size $\lambda \pi_1$ can be determined as the solution of the self-consistent problem \eqref{eq:selfProb}. 
In \cite{TTPW}, we show that the latter problem directly follows from the requirement of conservation of probability during blowups.

\begin{figure}[htbp]
\begin{center}
\includegraphics[width=0.8\textwidth]{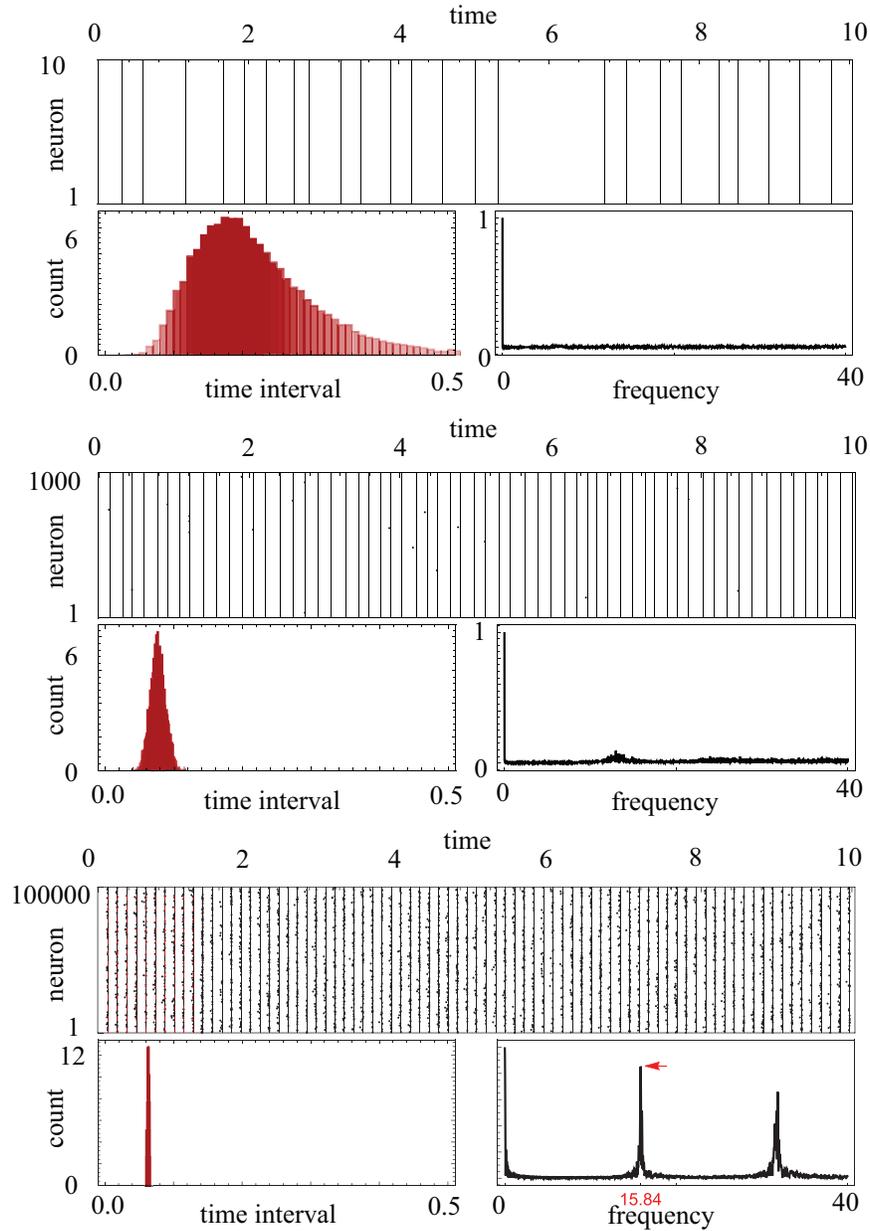}
\caption{{\bf Simulations of finite-size interacting particle systems.}
We simulate a finite network of neurons which follows the interacting dynamics depicted in Fig. \ref{fig:FiniteNetwork} and for varying size $N=10, 100, 10000$.
Parameters: $\Lambda=1$, $\nu=1$, $\lambda=20$ and $\epsilon=0$. Initial conditions: all neurons start at reset value.
In each panel: $(1)$ the top graphics represents a raster plot of the neuronal activity, where the spiking of neuron $i$ at time $t$ is marked by a point at coordinate $(i,t)$. 
Synchronous event corresponds to vertically aligned points which are apparent for all conditions.
$(2)$ the bottom left graphics represents the histogram of time interval between synchronous events.
$(3)$ the bottom right graphics represents the spectral density of the overall network activity.
Making the approximation that the full blowup has size $\pi_1=1$ and neglecting reset besides at $S_1$ yields the mean-field predictions: $1/T_1\simeq 16.18$ with $T_1=(S_1-\lambda H(S_1,\Lambda))/\nu \simeq 0.07143$, where $S_1$ is the smallest solution of $h(S_1,\Lambda)=1/\lambda$. 
The level of synchrony increases with the size of the system $N$, supporting that the particle system admits a (periodic) mean-field dynamics in the limit $N \to \infty$.
}
\label{fig:sim}
\end{center}
\end{figure}


\subsection{Motivation}

Unfortunately, the blowup analysis of \cite{TTPW} is only local in the sense that it allows one to resolve a single blowup for generic initial conditions.
However, this local analysis provides one with natural blowup exit conditions.
In principle, these blowup exit conditions can also serve as initial conditions to analyze the next blowup, if any.
Numerical simulations suggest that for large enough interaction parameters $\lambda > \Lambda$, explosive solutions exhibit repeated blowup episodes at regular time intervals (see Fig. \ref{fig:sim}).
Following on this observation, the first goal of this work is to extend the analysis of \cite{TTPW} to show the existence of global dPMF solutions that are defined over the entire half-line $\mathbbm{R}^+$ and that exhibit a countable infinity of blowups.
This program involves proving that for large enough interaction parameters, iteratively applying the blowup analysis of \cite{TTPW} produces synchronous events with sizes that remain bounded away from zero, at time intervals that also remain bounded away from zero.
Moreover, the blowup analysis of \cite{TTPW} is only concerned with dPMF dynamics for positive refractory period $\epsilon>0$. 
With positive refractory period $\epsilon>0$, explosive dPMF dynamics are always well posed but at the analytical cost of being determined as delayed dynamics.
For small enough $\epsilon>0$, we do not expect the refractory period to impact dPMF dynamics, except for ensuring their well-posedness in the presence of blowups.
This suggests defining explosive Poissonian mean-field (PMF) dynamics in the absence of a refractory period as the limit dPMF dynamics obtained when $\epsilon \to 0^+$.
Accordingly, the second goal of this work is to prove the existence and uniqueness of such limit dPMF dynamics.
This will only be possible for large enough interaction parameters, when global dPMF dynamics sustain isolated blowups for small enough $\epsilon>0$. 
PMF dynamics with zero refractory period $\epsilon=0$ are of interest for analysis as their blowups can be resolved just as for $\epsilon>0$, whereas their inter-blowup dynamics are nondelayed.


\subsection{Time-change approach}

\begin{figure}[htbp]
\begin{center}
\includegraphics[width=\textwidth]{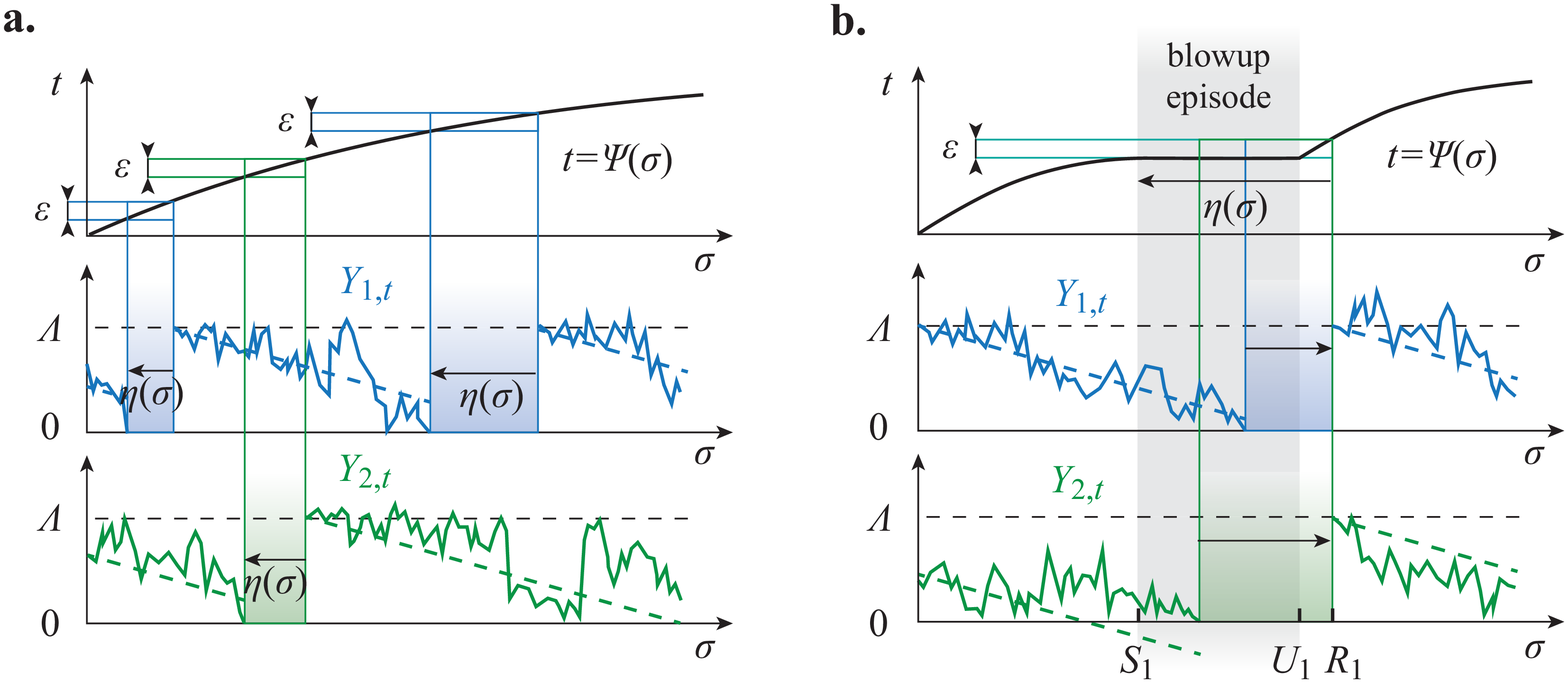}
\caption{
{\bf Time-changed picture.}
{\bf a.}
In the absence of blowups, the cumulative drift $\Phi$ specifies an increasing function which can be interpreted as an invertible time change. 
Denoting the inverse time change by $\Psi=\Phi^{-1}$ and given a representative process $X_t$ of the dPMF dynamics, the  process $Y_\sigma = X_{\Psi(\sigma)}$ follows a noninteracting drifted Wiener dynamics with negative unit drift and nonhomogeneous backward-delay function $\eta: \sigma \mapsto \sigma - \Phi(\Psi(\sigma)-\epsilon)$. 
Thus, in the time-change picture, all the interactions at play in dPMF dynamics are encoded by the time-dependence of delays. 
{\bf b.}
In the absence of blowups, full blowups are marked by a flat section $[U_1,S_1)$ of the inverse time change $\Psi$, where $U_1$ is the blowup trigger time and $U_1$ is the blowup exit time.
This corresponds to the backward-delay function $\eta$ having a jump discontinuity of size $S_1-U_1$ in $U_1$.
This also means that the forward-delay function $\tau$ specifying reset times is constant over the interval $[U_1,S_1)$, so that all processes involved in the blowup episode can reset at the same time $R_1$.
}
\label{fig:timeChange}
\end{center}
\end{figure}

In \cite{TTPW}, we analytically characterized blowup generation in dPMF dynamics by mapping the original time-homogenous, nonlinear dynamics onto a time-inhomogeneous, linear dynamics (see Fig. \ref{fig:timeChange}).
Such a mapping is operated by considering the cumulative drift $\Phi$ as an implicitly defined time-change function
\begin{eqnarray}\label{eq:sigmaTime_int}
\sigma = \Phi(t) = \nu t + \lambda F(t) \, , \quad \mathrm{with} \quad  F(t)=\int_0^t f(s) \, \dd s \, ,
\end{eqnarray}
which is only assumed to be a \emph{c\`adl\`ag} increasing function.
Due to the Poisson-like attributes of the neuronal drives, the time change $\Phi$ parametrizes the dPMF dynamics of a representative process as $X_t=Y_{\Phi(t)}$, where the time-changed dynamics $Y_\sigma$ obeys a linear, noninteracting dynamics.
The latter dynamics  is that of a Wiener process absorbed at zero, with constant negative unit drift and with reset at $\Lambda$, but with time-inhomogeneous refractory period specified via a $\Phi$-dependent, backward-delay function $\sigma \mapsto \eta[\Phi](\sigma)$ (see Fig. \ref{fig:timeChange}).
Independent of the presence of blowups, the functional dependence of $\eta$ on the time change $\Phi$ is given by 
\begin{eqnarray}\label{eq:eta_int}
\eta(\sigma) = \sigma - \Phi(\Psi(\sigma) -\epsilon) \, ,
\end{eqnarray}
where $\Psi=\Phi^{-1}$ refers to the inverse of the strictly increasing time change  $\Phi$.
Given a backward-delay function $\eta$, the transition kernel of the process $Y_\sigma$ denoted by  $(\sigma,x) \mapsto q(\sigma,x)= \dd \Prob{Y_\sigma \in \dd x \, \vert \, Y_\sigma>0} / \dd x$ satisfies the time-changed PDE problem 
\begin{eqnarray}\label{eq:qPDE_int}
\partial_\sigma q &=& \partial_x q +\frac{1}{2} \partial^2_{x} q  + \frac{\dd}{\dd \sigma} [G(\sigma-\eta(\sigma))] \delta_{\Lambda}  \, , 
\end{eqnarray}
with absorbing and conservation conditions respectively given by 
\begin{eqnarray}\label{eq:abscons_int}
q(\sigma,0)=0 \quad \mathrm{and} \quad \partial_\sigma G(\sigma)=\partial_x q(\sigma ,0) /2 \, .
\end{eqnarray}
In equations \eqref{eq:qPDE_int} and \eqref {eq:abscons_int}, $G$ denotes the $\eta$-dependent cumulative flux of $Y_\sigma$ through the zero threshold.
By definition of the time change $\Phi$, which is such that $X_t=Y_{\Phi(t)}$, $G$ is related to the cumulative flux $F$ via $F=G \circ \Phi$.
A key result of \cite{TTPW} is that as long as the delay function $\eta$ remains bounded, the PDE problem defined by \eqref{eq:qPDE_int} and \eqref{eq:abscons_int} admits a unique solution parametrized by an unconditionally smooth cumulative function $G=G[\eta]$.
This result holds even in the presence of blowups for the time-changed versions of the initial conditions given in Assumption \ref{assump1}.
These time-changed  initial conditions are specified as follows:

\begin{definition}\label{def:initCond_int}
Given normalized initial conditions $(p_0,f_0)$ in $\mathcal{M}(\mathbbm{R}^+) \times \mathcal{M}([-\epsilon,0))$, the initial conditions for the time-changed problem are defined by $(q_0,g_0)$ in $\mathcal{M}(\mathbbm{R}^+) \times \mathcal{M}([\xi_0,0))$ such that
\begin{eqnarray}
q_0=p_0 \quad \mathrm{and} \quad g_0 = \frac{\dd G_0}{\dd \sigma} \quad \mathrm{with} \quad G_0=(\mathrm{id} - \nu \Psi_0)/\lambda\, ,  \nonumber
\end{eqnarray}
where the function $\Psi_0$ and the number $\xi_0$ are given by:
\begin{eqnarray}
\Psi_0(\sigma) 
&=& 
\inf  \left\{  t \geq 0 \, \bigg \vert \, \nu t + \lambda \int_0^t f_0(s) \, \mathrm{d}s > \sigma \right\} \, ,  \nonumber\\
\xi_0&=& -\nu \epsilon -\lambda \int_{-\epsilon}^0 f_0(t) \, \dd t < 0 \, .  \nonumber
\end{eqnarray}
\end{definition}

In light of  \eqref{eq:eta_int},  the cumulative function $G=G[\Phi]$ actually depends on $\Phi$ via $\eta=\eta[\Phi]$.
This realization allows one to interpret the implicit definition of $\Phi$ given in \eqref{eq:sigmaTime_int} as a self-consistent equation for admissible time changes:
\begin{eqnarray}\label{eq:selfcons_int}
\Phi(t) = \nu t + \lambda G[\Phi](\Phi(t)) \, .
\end{eqnarray}
In \cite{TTPW}, we show that  such an interpretation specifies $\Phi$ as the solution of a certain fixed-point problem.
This  fixed-point problem turns out to be most conveniently formulated in term of the inverse time change $\Psi=\Phi^{-1}$, assumed to be a continuous, nondecreasing function.
Given an inverse time change $\Psi$, the time change $\Phi$ can be recovered as the right-continuous inverse of $\Psi$.
In the time-changed picture, blowups happen if the inverse time change $\Psi$ becomes locally flat and a synchronous event happens if $\Psi$ remains flat for a finite amount of time.
Informally, flat sections of $\Psi$ unfold blowups by freezing time in the original coordinate $t$, while allowing time to pass in the time-changed coordinate $\sigma$.
This unfolding of blowups in the time-changed picture is the key to analytically resolve blowups in dPMF dynamics.
Concretely, this amounts to showing the existence and uniqueness of solutions to the following fixed-point problem:

\begin{definition}\label{def:fixedpoint_int}
Given time-changed initial conditions $(q_0, g_0)$ in $\mathcal{M}(\mathbbm{R}^+) \times \mathcal{M}([\xi_0,0))$, an admissible inverse time change $\Psi$ satisfies the fixed-point problem
\begin{eqnarray}\label{eq:fixedPoint_int}
\forall \; \sigma \geq 0 \, , \quad \Psi(\sigma) = 
\left\{
\begin{array}{ccc}
\left( \sigma - \lambda \int_0^\sigma g_0(\xi) \dd \xi \right) / \nu   							& \quad \mathrm{if} &  -\xi_0 \leq \sigma < 0 \, , \vspace{5pt}\\
\sup_{0 \leq \xi \leq \sigma} \big( \xi - \lambda G[\eta](\xi)\big) / \nu   & \quad  \mathrm{if}  &  \sigma \geq 0  \, .
\end{array}
\right.
\end{eqnarray}
where $G[\eta]$ is the smooth cumulative flux uniquely specified by the time-inhomogeneous backward-delay function $\eta: \mathbbm{R}^+ \to \mathbbm{R}^+$ (see Definition \ref{def:quasirenew_int}).
The fixed-point nature of the problem follows from the definition of the backward-delay function $\eta$ as the $\Psi$-dependent time-wrapped version of the constant delay $\epsilon$:
\begin{eqnarray}\label{eq:eta_int}
\eta(\sigma) = \sigma - \Phi(\Psi(\sigma) -\epsilon)   \quad \mathrm{with} \quad \Phi(t)= \inf  \left\{  \sigma \geq \xi_0 \, \big \vert \,\Psi(\sigma) > t \right\}  \, ,
\end{eqnarray}
for which we consistently have $\eta(0)=- \Phi(-\epsilon)=\xi_0$.
\end{definition}

\begin{figure}[htbp]
\begin{center}
\includegraphics[width=0.8\textwidth]{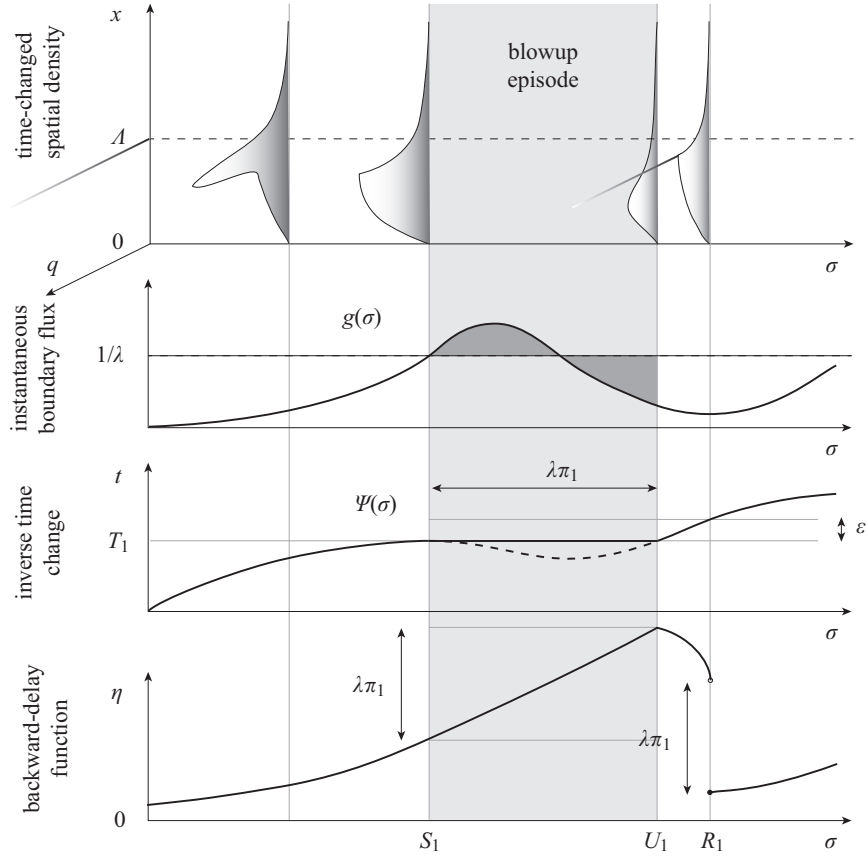}
\caption{{\bf Blowup resolution in the time-change picture.}
Schematic explosive dPMF dynamics considering a purely spatial initial condition with unit mass concentrated at $\Lambda$: $q_0=\delta_\Lambda$.
Initially, the dPMF dynamics solving the time-changed PDE is that of a noninteracting Wiener process with negative unit drift and delayed reset at $\Lambda$.
For $\sigma>0$, such dynamics admits a time-dependent density $x \mapsto q(\sigma, x)$ that is smooth on $(0,\infty)$, except for a slope discontinuity at the reset site $\Lambda$. 
Moreover, the density $x \mapsto q(\sigma, x)$ is locally differentiable in $0^+$ and the time-changed rate of inactivation $g(\sigma)$ is determined as the instantaneous flux: $g(\sigma) = \partial_x q(\sigma,0)/2$.
For large enough $\lambda$, the flux $g$ crosses the level $1/\lambda$ as a locally strictly convex increasing function at time $S_1$.
In principle, this would correspond to the inverse time change $\Phi: \sigma \mapsto (\sigma-\lambda G(\sigma) )/\nu$ admitting a strict local maximum at $S_1$.
However, such a behavior is not allowed as it would indicate that $T_1=\Phi(S_1)$ is a time-reversal point for the original dPMF dynamics.
This reveals $T_1$ as a blowup time for the original dynamics and $S_1$ as a blowup trigger time for the time-changed dynamics. 
During a blowup episode, the original time freezes, while the time changed dynamics is allowed to proceed unimpeded.
This corresponds to imposing that the inverse time change be flat via the definition $\Phi(\sigma)=\sup_{0 \leq \xi \leq \sigma} \big( \xi - \lambda G[\eta](\xi)\big) / \nu$. 
Accordingly, during a blowup episode, the backward-delay function $\eta$ increases with slope one and no reset can occur.
In other words, during a blowup episode, the time-changed dynamics loses its reset character and probability mass is gradually lost by inactivation in zero.
Eventually, this entails a decrease in the inactivation rate $g(\sigma)$ and an increase in $\sigma - \lambda G[\eta](\sigma)\big) / \nu$, up to a time $U_1$ when it reaches the value $T_1=\Phi(S_1)$ anew.
The time $U_1$ marks the blowup exit time and necessarily satisfies $U_1-S_1=\lambda \pi_1$ where $\pi_1$ is the fraction of processes that inactivate during the blowup episode.
Moreover, it is also such that the mean value of $g$ over the blowup episode is exactly $1/\lambda$.
Finally, the inactivated fraction $\pi_1$ is reset all at once at a later time $R_1$, referred to as the blowup reset time for which $\Phi(R_1) = T_1 + \epsilon$.
This instantaneous reset corresponds to a jump discontinuity in the backward-delay function of size $-\lambda \pi_1$.
}
\label{fig:PDE}
\end{center}
\end{figure}

%

In \cite{TTPW}, we further showed that the cumulative flux $G[\eta]$, introduced in the PDE problem defined by equations \eqref{eq:qPDE_int} and \eqref {eq:abscons_int}, is most conveniently characterized in terms of an integral equation obtained via renewal analysis:

\begin{definition}\label{def:quasirenew_int}
Given a time-inhomogeneous backward-delay function $\eta: \mathbbm{R}^+ \to \mathbbm{R}^+$, the cumulative flux $G$ is given as the unique solution to the quasi-renewal equation
\begin{eqnarray}\label{eq:DuHamel_int}
G(\sigma) = \int_0^\infty H(\sigma,x) q_0(x) \, \dd x  + \int_0^\sigma H(\sigma-\tau,\Lambda) \, \dd G(\tau-\eta(\tau))   \, ,
\end{eqnarray}
where by convention we set $G(\sigma)=G_0(\sigma)=\int_0^\sigma g_0(\xi) \dd \xi$ if  $-\xi_0 \leq \sigma < 0$.
The integration kernel featured in \eqref{eq:DuHamel_int} is specified as $\sigma \mapsto H(\sigma,x) = \Prob{\tau_x \leq \sigma} $, where $\tau_x$ is the first-passage time to zero of a Wiener process started in $x>0$ and with negative unit drift.  
\end{definition}

Definitions \ref{def:fixedpoint_int}  and \ref{def:quasirenew_int} fully specify the dPMF fixed-point problem in the time-change picture.
In this time-changed picture, the occurrence of a (first) blowup involves two distinct times: the blowup trigger time $S_1$ and the blowup exit time $U_1$, so that a synchronous event of size $\pi_1$ corresponds to $\Psi$ remaining flat on $[S_1,U_1)$ with $S_1=U_1+\lambda \pi_1$ (see Fig. \ref{fig:PDE}).
The definition of the blowup trigger time $S_1$ directly follows from Proposition \ref{def:blowup_int}, which can be recast in the time-changed picture as:

\begin{definition}
Under Assumption \ref{assump1}, the first blowup trigger time is defined as
\begin{eqnarray}\label{eq:blowup_int}
S_1 = \inf \{  \sigma > 0 \, \vert \, g(\sigma)  > 1/\lambda \} \, .
\end{eqnarray}
where $g(\sigma)= \partial_\sigma G(\sigma) = \partial_x q(\sigma, 0)/2$ is the smooth instantaneous flux in zero. 
\end{definition}

In turn, the full-blowup assumption \ref{assump2} admits a more natural but equivalent formulation in the time-changed picture:

\begin{assumption}\label{assump3}
At the blowup trigger time $S_1$,  the smooth instantaneous flux $g$ is such that $\partial_\sigma g (S_1)>0$.
\end{assumption}

The above formulation of the full-blowup assumption is natural in view of the fixed-point problem \ref{def:fixedpoint_int}.
Indeed, such an assumption implies that the function $\sigma \mapsto \sigma-\lambda G(\sigma)$ becomes decreasing on a finite interval to the left of $S_1$.
Thus, as a solution to the fixed-point problem \eqref{eq:fixedPoint_int}, $\Psi$ must remain flat on a nonempty maximum interval $[S_1,U_1)$, which corresponds to a synchronous event of size $\pi_1=(S_1-U_1)/\lambda$ (see Fig. \ref{fig:PDE}).
Incidentally, this observation sheds light on Definition \ref{def:blowup_int} characterizing full blowups.
During a full blowup episode, $\Psi$ remains flat so that the backward-delay function is constant on $[S_1,U_1)$.
As a result, the renewal-type integral term in \eqref{eq:DuHamel_int} vanishes and $G$ is revealed as the cumulative flux of a linear time-changed dynamics $Y_\sigma$, but without birth term due to reset.
This amounts to stopping the clock for the original time coordinate $t$, while letting the clock run in the changed coordinate $\sigma$.
Such a stoppage can only be maintained until $\Psi$ exceeds its blowup trigger value $\Psi(S_1)$ (see Fig. \ref{fig:PDE}), justifying the definition of blowup exit times $U_1$ as:

\begin{definition}\label{def:exitTime_intro}
The first blowup exit time satisfies
\begin{eqnarray}\label{eq:exit_int}
U_1 
=
\inf \{ \sigma > 0  \, \vert \, \Psi(\sigma) > \Psi(S_1) \} 
= S_1+ \lambda \pi_1 \, ,
\end{eqnarray}
where $ \pi_1$ is determined by the self-consistent condition \eqref{eq:selfProb}.
\end{definition}

The above definition of the blowup exit time is consistent with the characterization of the blowup size $\pi_1$ given in \eqref{eq:selfProb}.
Indeed, one can check that
\begin{eqnarray}
U_1 
=
\inf \{ \sigma > 0  \, \vert \, \Psi(\sigma) > \Psi(S_1) \}  
=
\inf \{ \sigma > 0  \, \vert \, \sigma-S_1> \lambda (G(\sigma) -G(S_1)) \}  \, . \nonumber
\end{eqnarray}
Then, introducing the reduced variable $p=(\sigma-S_1)/\lambda$, we consistently have 
\begin{eqnarray}
U_1&=&
S_1 + \lambda \inf \{p> 0  \, \vert \, p> G(S_1+\lambda p) -G(S_1) \}  \, ,  \nonumber\\
&=&
S_1 + \lambda \inf \left\{ p>0 \,  \vert \, p>\Prob{\tau_{X_{T_1}} < \lambda p } \right\}  \, , \nonumber
\end{eqnarray}
where the last equality follows from the fact that the process $Y_\sigma$ does not reset on $[S_1,U_1)$.
For nonzero vanishing period $\epsilon>0$, the processes that inactivate during a blowup episode reset after a period of duration $\epsilon$ elapses in original time coordinate.
Thus, the reset time of a representative process  $R_1$ shall satisfy $\Psi(R_1) =\Psi(S_1)+\epsilon$.
The latter relation uniquely specifies $R_1$ if $\Phi=\Psi^{-1}$ is continuous in $T_1+\epsilon=\Psi(S_1)+\epsilon$, which holds true if no blowup happens in $T_1+\epsilon$ (see Fig. \ref{fig:PDE}).
Following on \cite{TTPW}, without any continuity assumption, the blowup reset time $R_1$ is generally defined as the leftmost solution of $\Psi(R_1) =\Psi(S_1)+\epsilon$:

\begin{definition}\label{def:resetTime_intro}
The first blowup reset time satisfies
\begin{eqnarray}\label{eq:exit_int}
R_1 
=
\inf \{ \sigma > 0  \, \vert \, \Psi(\sigma) \geq \Psi(S_1)+\epsilon \} 
 \, .
\end{eqnarray}
\end{definition}

The definitions of the blowup trigger time $S_1$, the exit time $U_1$, and the reset time $R_1$ illuminate how the time-changed process $Y_\sigma$ resolves a blowup episode by alternating two types of dynamics (see Fig. \ref{fig:PDE}).
Before a full blowup,  the process $Y_\sigma$ follows a linear diffusion with absorption in zero and resets at $\Lambda$. 
These resets occur with time-inhomogeneous delays, which depends on the inverse time change $\Psi$.
At the blowup onset $S_1$, $\Psi$ becomes locally flat, indicating that the original time $t=\Psi(\sigma)$ freezes, thereby stalling resets.
As a result, after the blowup onset in $S_1$, the dynamics  of $Y_\sigma$ remains that of an absorbed linear diffusion but without resets.
Such a dynamics persists until a self-consistent blowup exit condition is met in $S_1+\lambda \pi_1$.
This condition states that it must take $\lambda \pi_1$  time-changed units for a fraction $\pi_1$ of processes to inactivate during a blowup episode.
Finally, the original time $t=\Psi(\sigma)$ resume flowing past $U_1$ and the inactivated fraction $\pi_1$ is reset at $\Lambda$ at time $R_1$.
Importantly, observe that the time-change process $Y_\sigma$ always behaves regularly in the sense that $G$ remains a smooth function throughout the blowup episode.


\subsection{Results}

In principle, dPMF dynamics could be continued past a blowup episode, and possibly even extended to the whole half-line $\mathbbm{R}^+$.
Simulating the particle systems conjectured to approximate dPMF dynamics suggests that these dynamics can sustain repeated synchronization events over the whole half-line $\mathbbm{R}^+$ (see Fig. \ref{fig:sim}).
This leads to conjecture the existence of solutions $\Psi$ globally defined on $\mathbbm{R}^+$ with a countable infinity of blowup episodes with associated sequence of blowup times $\{ S_k, U_k, R_k\}_{k \in \mathbbm{N}}$.
Extending results from \cite{TTPW} to show the existence of such global solutions requires showing: 

\begin{itemize}
\item \emph{Unconditional explosiveness}:  $(i)$ the so-called non-false-start exit condition $\partial_\sigma q(U_k,0)/2<1/\lambda$ and $(ii)$ the full-blowup condition $\partial^2_\sigma \Psi(S_k^-)<0$ are satisfied for all blowup episodes, not just for the first one.
\item \emph{Full-time domain}: the blowup times $T_k=\Psi(S_k)=\Psi(U_k)$ do not have an accumulation point, i.e., $\lim_{k \to \infty} T_k = \infty$, which means that the explosive solution is defined over the whole half-line $\mathbbm{R}^+$.
\end{itemize}

The main difficulties in establishing the two above properties is that blowup times may not be well ordered, i.e., $R_k > S_{k+1}$ for some $k>0$ and that blowup sizes $\pi_k$ may exhibit vanishing sizes: $\lim_{k \to \infty} \pi_k =0$.
However, one can exclude these possibilities in the strong interaction regime $\lambda \gg \Lambda$, for which the stable repetition of synchronization events can be understood intuitively:

$(i)$ 
 Consider an initial condition of the form $q_0=\pi_0 \delta_\Lambda$ with $\pi_0=1$, which corresponds to a full reset at time $R_0=0$ and which satisfies Assumption \ref{assump1}.
For a full blowup to occur next, the post-reset flux $g$ needs to satisfy the full-blowup Assumption \ref{assump3} at some trigger time $S_1$, where   we must have $g(S_1) =1/\lambda$ and $\partial_\sigma g(S_1)>0$.
For large enough $\lambda$, we expect this flux $g$ to be primarily due to processes that do not reset more than once on $[0,S_1]$.
In other words, we have $g(\sigma) \simeq  h(\sigma,\Lambda)$, where $h(\cdot,\Lambda)$ denotes the first-passage density of a Wiener process started at $\Lambda$ with unit negative drift and with a zero boundary.
The function $h(\cdot,\Lambda)$ is known to be a strictly increasing convex function on some nonempty interval $[0,\sigma^\dagger)$ with $h(0, \Lambda)=0$.
Thus, choosing a large enough $\lambda>\Lambda$ guarantees that the full-blowup Assumption \ref{assump3} will be met at some finite time $S_1>0$ with $S_1<\sigma^\dagger$.
One can check the scaling $S_1 \propto 1/\ln \lambda$,
so that the fraction of inactive processes at trigger time $S_1$, denoted $P(S_1)$, can be made arbitrarily small for large enough $\lambda$.
Actually, provided the refractory period is such that $\epsilon=O(1/\lambda)$, one can find an upper bound for $P(S_1)$ of the form $O(1/\lambda)$.

$(ii)$ 
Past the trigger time $S_1$, a fraction $1-P(S_1)$ of processes are susceptible to synchronize during the blowup episode.
It turns out that the actual size of the blowup $\pi_1<1-P(S_1)$ can be made arbitrary close to $1-P(S_1)$ for large enough $\lambda$.
Again, this is because we still have $g(\sigma) \simeq  h(\sigma,\Lambda)$ past the trigger time $S_1$, which satisfies $S_1<\sigma^\dagger$.
This observation, together with the monotonicity properties of $h(\cdot,\Lambda)$, implies that $g(\sigma)>1/\lambda$ on $[\sigma^\dagger, \sigma^\star]$, where $\sigma^\star$ is the unique maximizer of $h(\cdot,\Lambda)$.
This means that the blowup episode includes at least the nonempty time interval $[\sigma^\dagger, \sigma^\star]$, during which at least a fraction $\Delta_H=\int_{\sigma^\dagger}^{\sigma^\star} h(\sigma, \Lambda) \dd \sigma>0$ synchronizes.
Thus, we must have $\pi_1 \geq \Delta_H>0$, so that the duration of the blowup $\lambda \pi_1\geq \lambda \Delta_H$ can be made arbitrary long for large enough $\lambda$.
The key is then to observe that the tail behavior of $h(\cdot,\Lambda)$ indicates that inactivation during blowup asymptotically happens with hazard rate $1/2$ for large times. 
This observation allows one to deduce an exponential lower bound for the blowup size of the form, e.g., $\pi_1 \geq \big(1-P(S_1)\big) \big(1-O(e^{-\lambda \Delta_H /4})\big)>1/2$.

$(iii)$ 
Altogether, less than a fraction $1-\pi_1\leq O(1/\lambda)$ of processes survive the blowup at exit time $U_1=S_1+\lambda \pi_1$ with $\pi_1>1/2$.
The distribution of these surviving processes is such that $\partial_x q(U_1,0)/2 \simeq h(U_1,\Lambda) \leq O(e^{-U_1/2})=O(e^{-\lambda/4})<1/\lambda$ for large enough $\lambda$
so that the no-false-start Assumption \ref{assump1} can be met for large enough $\lambda>\Lambda$.
Moreover, we expect the fraction of surviving processes at $U_1$, i.e., $\Vert q(U_1, \cdot) \Vert_1$, to be so small that they cannot trigger a blowup alone.
In other words, the assumption of well-ordered blowups holds:
 the fraction of inactivated processes $\pi_1$ must reset at some time $R_1>U_1$ before any subsequent blowup can occur.
Then, choosing $\pi_1=1- O(1/\lambda)$ close enough to one ensures that $g$ remains primarily shaped by those processes that reset at time $R_1$, so that $g(\sigma) \simeq \pi_1 h(\sigma-R_1,\Lambda)$.
In turn, if the well-ordered property of blowups holds, we can apply $(i)$ anew but starting from $R_1$ with a controlled probability mass $\pi_1$, rather than the full mass $\pi_0=1$.

$(iv)$  To prove the well ordered property of blowups, one has to establish some \emph{a priori} bounds about the boundary flux $g$ with respect to the $L_1$ norm $\Vert q(U_1, \cdot) \Vert_1$.
Establishing such \emph{a priori} bounds requires to restrict the set of considered initial conditions to the so-called set of natural initial conditions.
In a nutshell, these are those---not necessarily normalized---initial distributions such that all the active and inactive processes at the starting time have been reset at $\Lambda$ at some point in the past.
For bearing on processes started away from zero at $\Lambda$, natural initial conditions cannot lead to large transient variations in boundary fluxes.
As a result, one can show that in the absence of reset, $\vert g \vert$ and $\vert \partial_\sigma g \vert$ are uniformly upper bounded by $O(\Vert q(U_1, \cdot) \Vert_1)$.
Then, the well-ordered property of blowups follows from showing that $\vert g \vert =  O(\Vert q(U_1, \cdot) \Vert_1) <1/\lambda$ in the absence of the blowup reset.
This requires that one has at least that $\Vert q(U_1, \cdot) \Vert_1 = O(1/\lambda)$, consistent with the fact that $\Vert q(U_1, \cdot) \Vert_1 \leq 1-\pi_1 = O(1/\lambda)$.
but not precise enough to conclude without explicit knowledge of the bounding coefficients..

$(v)$  Establishing that the no-false-start condition (Assumption \ref{assump1}), the full-blow up condition (Assumption \ref{assump2}), and the well ordered condition can be all met requires to estimate coefficients in the various $\lambda$-dependent bounds at play.
Then, the strategy is to use these explicit bounds to show that given a large enough $\lambda>\Lambda$, there exists a threshold mass $\Delta_\pi$, $1/2<\Delta_\pi<1$, such that for all blowups of size $\pi_1 \geq \Delta_\pi$, the next full blowup also satisfies $\pi_2>\Delta_\pi$.
From there, a direct induction argument on the number of blowups shows that there is a unique inverse time change $\Psi: \mathbbm{R}^+ \mapsto \mathbbm{R}^+$ dynamics as long as the initial mass $\pi_0 \geq \Delta_\pi$.
The fact that the domain  of $\Psi$ is the full half-line $\mathbbm{R}^+$ follows from $\Psi$ having a countable infinity of blowup with size at least $\Delta_\pi>1/2$.
The fact that the domain of $\Phi=\Psi^{-1}$ is also the full half-line $\mathbbm{R}^+$ follows from the fact that  well ordered blowups must be separated by the duration of the nonzero refractory period $\epsilon>0$.\\

Making the arguments above rigorous leads to our first main result:

\begin{theorem}\label{th:main1}
Consider some normalized natural initial conditions $(p_{\epsilon,0},f_{\epsilon,0})$ in $\mathcal{M}(\mathbbm{R}^+) \times \mathcal{M}([\xi_0,0))$ such that $p_{\epsilon,0}$ contains a Dirac-delta mass $\pi_{\epsilon,0}$, with $0 < \pi_{\epsilon,0} \leq 1$, at reset value $\Lambda$.

For small enough $\epsilon>0$, large enough $\lambda>\Lambda$, and large enough $\pi_{\epsilon,0}<1$, there exists a unique explosive dPMF dynamics parametrized by the time change $\Phi_\epsilon: \mathbbm{R}^+ \rightarrow  \mathbbm{R}^+$, with a countable infinity of blowups.
These blowups occurs at consecutive times $T_{\epsilon,k}$, $k \in \mathbbm{N}$, with size $\pi_{\epsilon,k}=(\Phi_\epsilon(T_{\epsilon,k})-\Phi_\epsilon(T_{\epsilon,k}^-))/\lambda$, and are such that $\pi_{\epsilon,k}$ and $T_{\epsilon,k+1}-T_{\epsilon,k}$ are both uniformly bounded away from zero.
\end{theorem}

In the following, we prove an equivalent version of the above theorem for the time-changed dynamics, i.e., Theorem \ref{th:globSol}.
The proof of Theorem \ref{th:globSol} only relies on  a few properties of the smooth first-passage density function $h(\cdot,\Lambda)$, which plays the key role throughout this analysis.
Although we do not refer to these properties explicitly in the proof, we list them as follows:

\begin{property1}
The function $h(\cdot, \Lambda)$ is both strictly convex and log-concave on some finite interval $[0,\sigma^\dagger]$ with $h(0)=0$.
\end{property1}

\begin{property2}
The hazard rate function of $h(\cdot, \Lambda)$ is asymptotically bounded away from zero and has at most exponential growth, i.e.:
\begin{eqnarray*}
\liminf_{\sigma \to \infty} \;  \frac{h(\sigma, \Lambda)}{\int_\sigma^\infty h(\xi, \Lambda)} \geq r > 0 \, \quad \mathrm{and} \quad \limsup_{\sigma \to \infty} \;  \ln \left(\frac{h(\sigma, \Lambda)}{\int_\sigma^\infty h(\xi, \Lambda)} \right) < \infty \, .
\end{eqnarray*}
\end{property2}

Given some large enough Dirac-delta mass $\pi_k$ at reset time $R_k$, Property 1 ensures that a full blowup occurs for large enough $\lambda$ at trigger time $S_{k+1}>R_k$ and with blowup size $\pi_{k+1}$ lower bounded by some constant $\Delta_H$.
This corresponds to the ``ignition'' phase of the blowup.
Given that  $\pi_{k+1} \geq \Delta_H$, Property 2 ensures that an exponentially small fraction of processes survives the blowup and that this fraction cannot lead to a subsequent blowup without reset, i.e., reignition.
This corresponds to the ``combustion'' phase of the blowup.
A potentially useful consequence of the above observations is that we expect our blowup analysis to extend to more generic diffusion processes.
Indeed, the existence and uniqueness of explosive dPMF dynamics shall hold for all  dynamics derived from autonomous diffusion processes whose first-passage time to a constant level has smooth densities satisfying Property 1 and 2.
This class of diffusion processes includes the Ornstein-Uhlenbeck process, which forms the basis of a widely used class of models in computational neuroscience, called leaky integrate-and-fire neurons \cite{Lapicque:1907,burkitt2006review}.

In stating Theorem \ref{th:main1}, we do note give a precise criterion for how to jointly choose the interaction parameter $\lambda>\Lambda$, the refractory period $\epsilon>0$, and the initial mass $\pi_0$ to yield explosive dPMF dynamics.
This is because the precise criterion given in the proof of Theorem \ref{th:main1} is not tight, our main goal being to establish existence and uniqueness of global solution in the large interaction regime.
Actually, we expect explosive dPMF dynamics to arise as soon as $\lambda>\Lambda$ for small $\epsilon \geq 0$ and we  expect this emergence to be essentially independent of the size of the initial mass $\pi_0$.
One can see the latter point by considering the limit case of PMF dynamics with zero refractory period  in the time change picture.
In the absence of blowups, plugging $\epsilon =0$ in Definition \ref{def:fixedpoint_int} yields valid, nondelayed dynamics with explicitly known instantaneous boundary flux $g$ via renewal analysis \cite{Asmussen:2003}. 
Moreover, classical arguments from renewal analysis show that $g$ is analytic away from zero and that $\lim_{\sigma \to \infty} g(\sigma) = 1/\Lambda$, independent of the  initial conditions.
Thus, it is natural to expect that PMF dynamics exhibit a countable infinity of blowups (at least) as soon as $\lambda>\Lambda$.

One general caveat is that checking the mean-field conjecture becomes exceedingly costly when $\lambda \simeq \Lambda$.
Indeed we find that the convergence to an asymptotically independent regimes for large network size $N\to \infty$ drastically slows down when $\lambda \simeq \Lambda$
---a numerical evidence of critical slowing down~\cite{Casella:2004}.
However, the conjectured PMF dynamics can always be analyzed in the time-changed picture by setting $\epsilon=0$ in Definition \ref{def:fixedpoint_int}.
One specific caveat remains that it is not clear why blowups are resolved in PMF dynamics via the mechanism implied by Definition \ref{def:fixedpoint_int} with $\epsilon =0$.
This is by contrast with dPMF dynamics for which the nonzero refractory period $\epsilon>0$ makes it clear that blowup episodes correspond to halting reset at the time-changed picture.
In view of this, our second main result is to justify that PMF dynamics ``physically'' resolve blowups by showing that these dynamics are recovered from dPMF ones in the limit $\epsilon \to 0^+$.

\begin{theorem}\label{th:main2}
Consider some normalized natural initial condition $p_0$ in $\mathcal{M}(\mathbbm{R}^+)$ such
 that $p_{0}$ contains a Dirac-delta mass $\pi_{0}$, with $0 < \pi_{0} \leq 1$, at reset value $\Lambda$.

For large enough $\lambda>\Lambda$ and large enough $\pi_{0}<1$,
the unique explosive PMF dynamics $\Phi: \mathbbm{R}^+ \rightarrow  \mathbbm{R}^+$ satisfies  $\Phi^{-1} = \lim_{\epsilon \to 0^+} \Phi^{-1}_{\epsilon}$ with compact convergence over $\mathbbm{R}^+$, where $\Phi_{\epsilon}$  are dPMF dynamics with same initial condition $p_0$.
Moreover, we also have convergence of the blowup times and blowup sizes: $\lim_{\epsilon \to 0^+} T_{\epsilon,k} = T_k$ and $\lim_{\epsilon \to 0^+} \pi_{\epsilon,k} = \pi_k$.
\end{theorem}

In the following, we prove an equivalent version of the above theorem for the time-changed dynamics, i.e., Theorem \ref{thm:epsilonCont}.
Given an explosive dPMF dynamics parametrized by $\Phi_\epsilon=\Psi_\epsilon^{-1}$, proving Theorem \ref{thm:epsilonCont}  amounts to showing that the $\epsilon$-dependent backward-delay function 
\begin{eqnarray}\label{eq:def_eta}
\eta_\epsilon(\sigma) = \sigma -\Phi_\epsilon (\Psi_\epsilon(\sigma) - \epsilon)
\end{eqnarray}
is well behaved in the limit $\epsilon \to 0^+$.
By well behaved, we mean that  $(1)$ $\eta_\epsilon$ converges to zero in between blowups on intervals $[R_k,S_k]$, consistent with the nondelayed nature of PMF dynamics,
or $(2)$ $\eta_\epsilon$ converges toward the unit slope  function $\sigma \mapsto \sigma-S_k$ on blowup intervals $[S_k,U_k]$, consistent with halting reset at $S_k$.
Such an alternative covers all cases as we expect instantaneous reset after blowup exit: $U_k=R_k$ for $\epsilon=0$.
This corresponds to assuming a limit of the form
\begin{eqnarray*}
\eta(\sigma) = \sigma - \Phi (\Psi(\sigma)) \, , \nonumber 
\end{eqnarray*}
where crucially, $\Phi$ is the right-continuous inverse of $\Psi$, just as if one consider the fixed-point problem of Definition \ref{def:fixedpoint_int} with $\epsilon=0$. 
The main difficulty is that  \eqref{eq:def_eta}  does not behave continuously for standard topologies when it bears on  functions $\Phi_\epsilon$ with jump discontinuities. 
One route to address this point is to consider topologies specially designed to deal with \emph{c\`adl\`ag} functions such as the Skorokhod topologies \cite{Billingsley:2013,Delarue:2015b}.
However, the continuity of the fixed-point problem specified in Definition \ref{def:fixedpoint_int} with respect to such topologies does not appear straightforward.

A more pedestrian route is to leverage the fact that for small enough $\epsilon>0$, the time change $\Phi_\epsilon$ has only a countable infinity of blowups for large enough interaction parameter $\lambda>\Lambda$ and large enough initial reset mass $\pi_0<1$.
Under these standard conditions, one can show that the associated blowup times $(S_{\epsilon,k},U_{\epsilon,k},R_{\epsilon,k})$, $k$ in $\mathbbm{N}$, can be uniformly controlled with respect to $\epsilon \to 0^+$.
Such control can be established by considering separately the three stages of a time-change blowup dynamics: $(1)$ the blowup trigger stage on $[R_{\epsilon,k},S_{\epsilon,k+1})$, $(2)$ the blowup resolution stage on $[S_{\epsilon,k+1},U_{\epsilon,k+1})$, and $(3)$ the blowup reset stage $[U_{\epsilon,k+1},R_{\epsilon,k+1})$. 
The key observation is that each of these stages has a duration that continuously depends on $\epsilon$ and on the $L_1$ norm of the current state of the dynamics at the start of the stage.
Note that these current states represent natural initial conditions in $\mathcal{M}((0,\infty)) \times \mathcal{M}([-\eta_\epsilon(\sigma), 0)) $ specified by
$\big(q_\epsilon(\sigma,\cdot) , \{ g_\epsilon(\xi -\sigma)\}_{-\eta_\epsilon(\sigma)\leq \xi < 0} \big)$
for $\sigma=R_{\epsilon,k}, S_{\epsilon,k+1}, U_{\epsilon,k+1}$.
Both the continuity of the trigger time $S_{\epsilon,k+1}$  and of the exit time $U_{\epsilon,k+1}$ follow from nondegeneracy conditions, namely that the full-blowup condition $\partial_\sigma g(S_{k+1}) >0$ and the non-false-start condition $\partial_\sigma g(U_{k+1}) <1/\lambda$ holds uniformly with respect to $\epsilon$.
The continuity of the reset time $R_{k+1}$ is straightforward for well-ordered blowups.
Thus, to propagate these continuity results by induction on the number of blowups, one just has to show that for each stage, considering  the terminal dynamical state as a function of the initial dynamical state specifies a continuous mapping in the $L_1$ norm.
Technically, establishing such $L_1$ continuity results, jointly with $\epsilon$-continuity, is primarily made possible by the availability of certain uniform bounds for the heat kernel at times bounded away from zero and infinity.

The above discussion lays out our strategy to prove that PMF dynamics with zero refractory period represents ``physical'' solutions, which is one of our main objectives.
PMF dynamics are of special interest because they admit an explicit iterative construction in terms of well-known analytic functions. 
Moreover, PMF can be easily simulated within the PDE setting. 
Such simulations confirm the prediction from particle-system simulations that explosive PMF solutions become asymptotically periodic (see Fig. \ref{fig:sim}).
In that respect, we conclude by mentioning that the existence of periodic PMF dynamics directly follows from our $L_1$ continuity analysis with $\epsilon=0$.
This is a consequence of  Schauder's fixed-point theorem \cite{Granas:2003} applied to the inter-exit-time mapping
\begin{eqnarray*}
\mathcal{U}: \big(q(U_k,\cdot),\pi_k \big) \mapsto \big(q(U_{k+1},\cdot),\pi_{k+1} \big)
\end{eqnarray*}
defined on $\mathcal{C}$, the convex set of normalized natural initial conditions with large enough blowup mass $\pi_k$ (given some large enough $\lambda>\Lambda$).
Note that we automatically have  $\Vert q(U_k,\cdot) \Vert_1 + \pi_k=\Vert q(U_{k+1},\cdot) \Vert_1 + \pi_{k+1} = 1$ when $\epsilon=0$ so that the set $\mathcal{C}$  is stable by $\mathcal{U}$.
Our analysis shows that the mapping $\mathcal{U}$ is continuous with respect to the $L_1$ norm.
Then, all there is left to show in order to apply Schauder's fixed-point theorem is that the mapping $\mathcal{U}$ is compact in $\mathcal{C}$ with respect to the $L_1$ norm.
But this follows directly from the fact that blowups are lower bounded by some $\Delta_\pi>1/2$.
The latter point implies that the image of $\mathcal{U}$ is included in the image of a compact operator defined in terms of the heat kernel for times at least $1/2$.
Thus there must be a fixed point to $\mathcal{U}$ in $\mathcal{C}$, which corresponds to a periodic PMF dynamics.
Studying the general contractive property of PMF and dPMF dynamics is beyond the scope of this work.


\subsection{Structure}

In Section \ref{sec:modelDef}, we recall results from \cite{TTPW} that allows one to resolve a single dPMF blowup episode in the time-changed picture.
In Section \ref{sec:global}, we leverage these results to show that dPMF solution can be continued past blowup episodes to specify explosive solutions over the whole half-line $\mathbbm{R}^+$ in the large interaction regime.
In Section \ref{sec:limRef}, we show that explosive PMF dynamics, which can be defined explicitly for zero refractory period $\epsilon=0$, are recovered from dPMF dynamics in the limit $\epsilon \to 0^+$.
Appendices \ref{appA} and \ref{appB} contain two technical lemmas that are needed to demonstrate the result of Section \ref{sec:limRef}.


\section{Regularization of explosive dynamics via change of time}\label{sec:modelDef}

In this section, we define the time-changed dynamics obtained by change of variable and introduce the delay functions that parametrize such a dynamics.
We then recall the results from renewal theory that characterize the cumulative flux function associated with a time-changed dynamics
Finally, we utilize these results to justify the definition of the fixed-point formulation of dPMF dynamics given in Definition \ref{def:fixedpoint_int}.

\subsection{Definition of the time-inhomogeneous linear dynamics}\label{sec:defDyn}

Following \cite{TTPW}, our approach is based on representing possibly explosive dPMF dynamics $X_t=Y_{\Phi(t)}$  as time-changed versions of nonexplosive dynamics $Y_\sigma$.
A good choice for the time change function $\Phi$ is one for which the dynamics $Y_\sigma$ is simple enough to be analytically tractable.
The Poisson-like attributes of dPMF dynamics suggest defining $\Phi$ implicitly as the integral function of the drift:
\begin{eqnarray}\label{eq:Phi}
t \mapsto \Phi(t) = \nu t+ \lambda F(t) \, .
\end{eqnarray}
Such a definition equates blowups/synchronous events with singularities/discontinuities in the time change $\Phi$, while allowing for the corresponding time-changed dynamics $Y_\sigma$ to always be devoid of blowups.
As a result, showing the existence and uniqueness of a dPMF dynamics $X_t$ reduces to showing the existence and uniqueness of a time change $\Phi$ satisfying relation \eqref{eq:Phi}.
We shall seek such time changes $\Phi$ in a set of candidate time changes denoted by $\mathcal{T}$.
To define $\mathcal{T}$, we utilize the fact that the cumulative function $F: \mathbbm{R}^+ \to \mathbbm{R}^+$ featured in \eqref{eq:Phi} can be  safely assumed to be a nondecreasing function.
This leads to the following definition:

\begin{definition}\label{def:Phi}
The set of candidate time change $\mathcal{T}$ is the set of \emph{c\`adl\`ag} functions $\Phi:[-\epsilon, \infty)\to[\xi_0, \infty)$, such that their difference quotients are lower bounded by $\nu$: for all $y,x \leq -\epsilon$, $x \neq y$, we have
\begin{eqnarray}
w_\Phi(y,x)=\frac{\Phi(y)-\Phi(x)}{y-x} \geq \nu \, .  \nonumber
\end{eqnarray}
\end{definition}

As alluded to in the introduction, it is actually most convenient to consider dPMF dynamics as parametrized by the inverse time change:

\begin{definition}\label{def:Psi}
Given a time change $\Phi$ in $\mathcal{T}$, the inverse time change $\Phi^{-1}:[\xi_0,\infty) \to [-\epsilon,\infty)$ is defined as the continuous function 
\begin{eqnarray}
\sigma \mapsto \Psi(\sigma)=\Phi^{-1}(\sigma) = \inf \left\{ t \geq 0 \, \vert \, \Phi(t) > \sigma \right\} \, .  \nonumber
\end{eqnarray}
\end{definition}

In order to specify the time changed dynamics $Y_\sigma$, we further need to introduce the so-called backward-delay function $\eta$, which represents the time-wrapped version of the refractory period $\epsilon$:

\begin{definition}\label{def:Eta}
Given a time change $\Phi$ in $\mathcal{T}$, we define the corresponding backward-delay function $\eta:[0,\infty) \to \mathbbm{R}^+$ by
\begin{eqnarray}
\eta(\sigma) = \sigma - \Phi \left( \Psi(\sigma) - \epsilon \right) \, , \quad \sigma \geq 0 \, .  \nonumber
\end{eqnarray}
We denote  the set of backward functions $\lbrace \eta[\Phi]\rbrace_{\Phi \in \mathcal{T}}$ by $\mathcal{W}$.
\end{definition}

As $w_\Phi \geq \nu$ for all $\Phi$ in $\mathcal{T}$, it is clear that for all $\eta$ in $\mathcal{W}$, we actually have $\eta  \geq \nu \epsilon$, so that all delays are bounded away from zero.
%
Time-wrapped-delay function $\eta$ in $\mathcal{W}$  will serve to parametrize the time-changed dynamics obtained via $\Phi$ in $\mathcal{T}$.
In \cite{TTPW},  we showed that these time-changed dynamics are that of a modified Wiener process $Y_\sigma$ with negative unit drift, inactivation on the zero boundary, and reset at $\Lambda$ after a refractory period specified by $\eta$.
Consequently, we defined time-changed dynamics as the processes $Y_\sigma$ solutions to the following stochastic evolution:

\begin{definition}\label{def:Ysigma}
Denoting the canonical Wiener process by $W_\sigma$, we define the time-changed processes $Y_\sigma$ as solutions to the stochastic evolution
\begin{eqnarray}\label{eq:Ysigma}
Y_\sigma = -\sigma+ \int_0^\sigma \mathbbm{1}_{\{Y_{\xi^-}>0\}}\, \dd W_\xi + \Lambda N_{\sigma-\eta(\sigma)} \, ,\quad \mathrm{with} \quad N_\sigma = \sum_{n>0} \mathbbm{1}_{[\xi_0,t]}(\xi_n) \, ,
\end{eqnarray}
 where the process $N_\sigma$ counts the successive first-passage times $\xi_n$ of the process $Y_\sigma$ to the absorbing boundary:
\begin{eqnarray}
\xi_{n+1} = \inf \left\{  \sigma >0  \, \big  \vert \, \sigma-\eta(\sigma) > \xi_n, Y_\sigma  \leq 0 \right\}  \, .  \nonumber
\end{eqnarray}
\end{definition}

A time-changed process $Y_\sigma$ is uniquely specified by imposing elementary initial conditions, which take an alternative formulation: either the process is active $Y_0=x>0$ and $N_0=0$, either the process has entered refractory period at some earlier time $\xi$ so that $Y_0=0$ and $N_\sigma=\mathbbm{1}_{\sigma \geq \xi}$ for $\xi_0 \leq \sigma < 0$.
Generic initial conditions are given by considering that $(x,\xi)$ is sampled from some probability distribution on $\{ (0,\infty) \times \{ 0 \} \} \cup \{ \{ 0 \} \times [\xi_0,0) \}$.
In all generality, this amounts to choosing a normalized pair of distributions $(q_0,g_0)$  in $\mathcal{M}((0,\infty)) \times \mathcal{M}([\xi_0,0))$. 
Given generic initial conditions  $(q_0,g_0)$  in $\mathcal{M}((0,\infty)) \times \mathcal{M}([\xi_0,0))$, the dynamics defined by \ref{def:Ysigma} is well posed as long as the backward-delay function $\eta \geq \nu \epsilon$ is locally bounded, which is always the case for valid time change $\Phi$.
Moreover, the corresponding density function $(\sigma, x) \mapsto q(\sigma, x)= \dd \Prob{0<Y_\sigma \leq x} /\dd x $ must solve the PDE problem defined by \eqref{eq:qPDE_int} and \eqref {eq:abscons_int} in the introduction.
To establish the existence and uniqueness of global explosive solutions, we will actually consider a restricted class of initial conditions, referred to as natural conditions.
We postpone the introduction of this so-called notion of natural initial conditions to Section \ref{sec:apriori}.

Interestingly, interactions are entirely mediated by the time-delayed counting process $N_{\sigma-\eta(\sigma)}$ in the linear time-changed dynamics  \eqref{eq:Ysigma}.
This motivates introducing the so-called backward-time function $\sigma \mapsto \xi(\sigma)$, which marks the inactivation time of the processes being reset at time $\sigma$, and the associated forward-time function $\sigma \mapsto \tau(\sigma)$, which marks the reset time of the processes inactivated at time $\sigma$.

\begin{definition}\label{def:xitauDef}
Given a time change $\Phi$ in $\mathcal{T}$, we define the backward-time function by 
\begin{eqnarray}
\xi: \mathbbm{R}^+ \to [\xi_0,\infty ) \, , \quad \xi(\sigma) =  \sigma-\eta(\sigma) = \Phi \left( \Psi(\sigma) - \epsilon \right) \, ,  \nonumber
\end{eqnarray}
and the forward-time function  by 
\begin{eqnarray}\label{eq:tauDef}
\tau:  [\xi_0,\infty ) \to \mathbbm{R}^+ \, , \quad \tau(\sigma) = \inf \{ \tau > 0 \, \vert \, \xi(\tau) = \tau-\eta(\tau) \geq \sigma \} \, .
\end{eqnarray}
\end{definition}

When unambiguous, we will denote these functions by $\xi$ and $\tau$ refer to them as the ``functions $\xi$ and $\tau$'' to differentiate from when $\xi$ and $\tau$ play the role of real variables.
By construction,  the function $\xi$ and $\tau$ are nondecreasing \emph{c\`adl\`ag} and nondecreasing \emph{c\`agl\`ad} functions, respectively. 
In particular, both functions can admit discontinuities and flat regions.
In \cite{TTPW}, we justify the definition of the function $\tau$ as the left-continuous inverse of the function $\xi$ with a detailed analysis of the delayed reset mechanism in the dynamics  \eqref{eq:Ysigma}.
We also show there that if we denote the sequence of inactivation times by $\{  \xi_k \}_{k \geq 0}$, then the corresponding sequences of reset times $\{  \tau_k \}_{k \geq 0}$ satisfies $\tau_k=\tau(\xi_k)$, $k \geq 0$.

\subsection{Results from time-inhomogeneous renewal analysis}\label{sec:renew}

Given a backward-delay function $\eta$ in $\mathcal{W}$, one can consider the PDE problem defined by \eqref{eq:qPDE_int} and \eqref {eq:abscons_int} independently of the requirement that relation \eqref{eq:Phi} be satisfied.
Then, the main hurdle to solving this PDE problem is due to the presence of an inhomogeneous, delayed reset term featuring the cumulative flux function $G$.
Luckily, assuming the cumulative function $G$ known allows one to express the full solution of the inhomogeneous PDE in terms of its homogeneous solutions.
These homogeneous solutions are known in closed form~\cite{Karatzas}:
\begin{eqnarray}\label{eq:kappaDef}
\kappa(\sigma,y,x) = \frac{e^{-\frac{(y - x +  \sigma)^2}{2\sigma}}}{\sqrt{2 \pi \sigma}}  \left(1-e^{ -\frac{2 x y}{\sigma}} \right) \, .
\end{eqnarray}
Applying Duhamel's principle yields the following integral representation for the density function $(\sigma,x) \mapsto q(\sigma,x)=\dd \Prob{0<Y_\sigma \leq x}$
\begin{eqnarray}\label{eq:DuHamel2}
q(\sigma,x) =  \int_0^\infty \kappa(\sigma,y,x) q_0(x) \, \dd x  + \int_0^\sigma \kappa(\sigma-\tau,x,\Lambda) \, \dd G(\xi(\tau))   \, ,
\end{eqnarray}
where $\xi=\mathrm{id}-\eta$ and where $G$ remains to be determined. 
The above integral representation shows that $G$ is the key determinant parametrizing the time-changed process $Y_\sigma$.

Just as $Y_\sigma$,  $G$ only depends on the backward-delay functions $\eta$, or equivalently on the functions $\xi$ or $\tau$.
These dependences can be made explicit by adapting results from renewal analysis~\cite{TTPW}.
Informally, this follows from writing the cumulative flux function as the uniformly convergent series
\begin{eqnarray}\label{eq:Ftx}
G( \sigma)  = \Exp{\sum_{k=1}^{\infty}  \mathbbm{1}_{ \left\{ \xi_k <  \sigma \right\} }  } = \sum_{k=1}^{\infty} \Prob{ \xi_k \leq  \sigma }  \, ,
\end{eqnarray}
and realizing that for all $k \geq 1$, $\Prob{ \xi_k \leq  \sigma}$ can be expressed in terms of the first-passage time cumulative distribution~\cite{Karatzas}:
\begin{eqnarray}\label{eq:HFPT}
H(\sigma,x)  
= \frac{\mathbbm{1}_{\{\sigma \geq 0 \}}}{2} 
\left( 
\mathrm{Erfc} \left( \frac{x-\sigma}{\sqrt{2 \sigma}}\right) + e^{2 x} \mathrm{Erfc} \left( \frac{x + \sigma}{\sqrt{2 \sigma}}\right) 
\right) \, .
\end{eqnarray}
Specifically, considering the initial condition $q_0=\delta_x$ for simplicity,  we have
\begin{eqnarray}
 \Prob{ \xi_k \leq  \sigma  \, \vert \, Y_0=x} = H^{(k)}(\sigma, x) \, ,  \nonumber
\end{eqnarray}
where  the functions $H^{(k)}(\sigma, x)$, $k \geq 1$, are defined as inhomogeneous iterated convolution:
\begin{eqnarray}
H^{(k)}(\sigma, x) = \int_0^\infty H(\sigma-\tau(\xi)) \, \dd H^{(k-1)}(\xi, x) \, , \quad H^{(1)}(\sigma, x) = H(\sigma,x) \, . \nonumber
\end{eqnarray}
These iterated functions $H^{(k)}(\sigma, x)$ take a particularly simple form for constant backward-delay function $\eta=d=-\xi_0$.
Indeed, by divisibility of the first-passage distributions
\begin{eqnarray}
\Prob{\xi_k \leq \xi \, \vert \, Y_0=x} = H(\xi -(k-1)d, x+(k-1)\Lambda) \, ,  \nonumber
\end{eqnarray}
so that for the elementary initial condition $q_0=\delta_x$, the cumulative flux function is
\begin{eqnarray}
G_d(\sigma,x)= \sum_{k=1}^\infty H(\xi -(k-1)d, x+(k-1)\Lambda)  \, .  \nonumber
\end{eqnarray}

Due to the  quasi-renewal character of the dynamics $Y_\sigma$, the cumulative flux $G$ admits an alternative characterization in terms of an integral equation.
Because of the presence of inhomogeneous delays, this integral equation differs from standard convolution-type renewal equations.
Specifically, we show in~\cite{TTPW} that for generic initial conditions:

\begin{proposition}\label{prop:Renewal}
Given a backward-delay function $\eta$ in $\mathcal{W}$, the cumulative flux function $G$  is the unique solution of the renewal-type equation:
\begin{eqnarray}\label{eq:DuHamel}
G(\sigma) = \int_0^\infty H(\sigma,x) q_0(x) \, \dd x  + \int_0^\sigma H(\sigma-\tau,\Lambda) \, \dd G(\xi(\tau))   \, .
\end{eqnarray}
\end{proposition}

As a solution to \eqref{eq:DuHamel}, it is clear that independent of the functions $\eta$, $\xi$, $\tau$, the cumulative function $G$ inherits all the regularity properties of $\sigma \mapsto H( \sigma, \Lambda)$, i.e., $G$ must be a smooth function for $\sigma>0$.
In particular, one can differentiate \eqref{eq:DuHamel} with respect to $\sigma$ and obtain a new renewal-type equation for the instantaneous flux $g$:
\begin{eqnarray}\label{eq:gRenew}
g(\sigma) 
= 
\int_0^\infty h(\sigma,x) q_0(x) \, \dd x
+
\int_0^\sigma  h(\sigma-\tau,\Lambda) \, \dd G( \xi(\tau)) \, ,
\end{eqnarray}
and where $h(\cdot, x)$ denotes the density function of the cumulative function $H(\cdot, x)$:
\begin{eqnarray}\label{eq:hFPT}
h(\sigma,x) = \partial_\sigma H(\sigma,x)  = \mathbbm{1}_{\{\sigma \geq 0 \}}\frac{xe^{-(x -\sigma)^2/2 \sigma}}{\sqrt{2 \pi \sigma^3}}  \, .
\end{eqnarray}
Alternatively, \eqref{eq:gRenew}  can be directly obtained from differentiating \eqref{eq:DuHamel2} with respect to $x$ in zero and using that fact that $h(\sigma,x) = \partial_y \kappa(\sigma, 0,x)/2$.
In \cite{TTPW}, renewal-type equations such as \eqref{eq:DuHamel} and \eqref{eq:gRenew} play a crucial role in showing the local existence of explosive dPMF solutions.
Here, these will provide us with \emph{a priori} estimates that are useful to globally define explosive dPMF solutions over the whole half-line $\mathbbm{R}^+$.

Finally, for constant delay, the renewal-type equation \eqref{eq:DuHamel}  turned into an time-homogeneous equation by a change of variable:
\begin{eqnarray}\label{eq:dDuHamel}
G(\sigma) = \int_0^\infty H(\sigma,x) q_0(x) \, \dd x  + \int_0^\sigma H(\sigma-\tau-d,\Lambda) \, \dd G(\tau)   \, .
\end{eqnarray}
Deriving the above equation with respect to $\sigma$ leads to a convolution equation for the the flux $g_d=\partial_\sigma G_d$.
This equation can be solved algebraically in the Laplace domain yielding:
\begin{eqnarray}
 \hat{g}_d(u,x) = \int_0^\infty g_d(\sigma,x) e^{-u \sigma}\, d\sigma =  \frac{e^{x \left( 1-\sqrt{1+2u} \right)-du}}{1-e^{\Lambda \left( 1-\sqrt{1+2u} \right)-du}} \, . \nonumber
\end{eqnarray}
Finally, the above explicit expression allows one to evaluate the long-term behavior of the flux as $\lim_{\sigma \to \infty} g_d(\sigma,x) = \lim_{u\to0}  u\hat{g}_d(u,x)$~\cite{Feller:1941}, which gives the following useful result:

\begin{proposition}\label{prop:gflux}
For constant delay $d$, the flux function $g_d$ admit a limit that is independent of the initial conditions:
\begin{eqnarray}
\lim_{\sigma \to \infty} g_d(\sigma) =\frac{1}{\Lambda+d} \, . \nonumber
\end{eqnarray}
\end{proposition}

Observe that the above result is valid for $d=0$, in which case  $\lim_{\sigma \to \infty} g(\sigma)=1/\Lambda>1/\lambda$ as soon as $\lambda>\Lambda$.
This shows that for zero refractory period $\epsilon$ and for 
$\lambda > \Lambda$, a blowup must occur in finite time.
However, it is unclear how to resolve the ensuing blowup episode, since defining spiking avalanches can be ambiguous when $\epsilon=0$.
We resolve this ambiguity by viewing PMF dynamics as the limit of dPMF dynamics. That is, we define PMF by formally setting $\epsilon=0$ in the 
dPMF fixed-point problem. We then show that, for large enough interaction
parameter $\lambda > \Lambda$, the resulting PMF dynamics are the limit of dPMF dynamics as $\epsilon \to 0^+$. 


\subsection{Fixed-point problem and local solutions}\label{sec:fixedPoint}

Renewal analysis allows one to prove the existence, uniqueness, and regularity of the time-changed dynamics $Y_\sigma$ assuming the backward-delay function $\eta$ known.
However, $\eta$ is actually an unknown of the problem for being ultimately defined in terms of the time change $\Phi$ via Definition \ref{def:Eta}.
Moreover,  the time-changed dynamics $Y_\sigma$ exists for any $\eta$ in $\mathcal{W}$, whereas given some  initial conditions, we expect that a unique time change $\Phi$ in $\mathcal{T}$ parametrizes a dPMF dynamics.
In \cite{TTPW}, we characterize this unique time change $\Phi$ by imposing that the defining relation \eqref{eq:Phi} be satisfied.
Observing that $F$ is the cumulative flux of $X_t=Y_\Phi(t)$, relation \eqref{eq:Phi} can be naturally reformulated as:
\begin{eqnarray}\label{eq:sigmaTime}
\Phi(t) = \nu t + \lambda F(t) =\nu t + \lambda G(\Phi(t)) \quad \Leftrightarrow \quad \nu t = \Phi(t) - \lambda G(\Phi(t))  \, .
\end{eqnarray}

Assuming $G$ known, one can specify the sought-after time change $\Phi(t)=\sigma$ by solving $\nu t = \sigma - \lambda G(\sigma)$ for $\sigma$ given arbitrary values of $t \geq 0$.
However, one has to be mindful about the existence and multiplicity of solutions.
By construction, we have $G(0)=0$, so that we can always set $\Phi(0)=0$.
Then, we can uniquely specified $\Phi(t)=\sigma$  for increasing value of $t$ as long as $\sigma \mapsto \sigma - \lambda G(\sigma)$ is an increasing, one-to-one function.
Let us consider the smallest time  $S_1$ after which the smooth function $\sigma \mapsto \sigma - \lambda G(\sigma)$ fails to be  an increasing, one-to-one function.
Then, $\Phi(t)$ is uniquely defined within the range $[0,S_1)$, or rather the inverse time change $\Psi=\Phi^{-1}$ is uniquely defined as 
\begin{eqnarray}\label{eq:Psi1}
 \Psi(\sigma) = (\sigma - \lambda G(\sigma))/\nu \, , \quad \forall \sigma  \; \in \; [0, S_1)\, .
\end{eqnarray}

If $S_1$ is finite, we must have $g(S_1)=1/\lambda$ so that $S_1$ is a blowup trigger time as defined in \eqref{eq:blowup_int}, and setting $T_1=\Psi(S_1)$, the equation $\nu T_1 = \sigma - \lambda G(\sigma)$ must admit multiple solutions for $\sigma$.
This indicates that $\Phi$  has a jump discontinuity in $T_1$, with $\Phi(T_1)=U_1>S_1=\Phi(T_1^-)$.
To be consistent, $U_1$ shall be the smallest solution such that  $U_1>S_1$ and $\sigma \mapsto \sigma - \lambda G(\sigma)$ is increasing, one-to-one in the right vicinity of $U_1$.
This corresponds to defining $U_1$ as an blowup exit time as in \eqref{eq:exit_int} and imposing that 
\begin{eqnarray}\label{eq:Psi2}
 \Psi(\sigma) = \Psi(S_1)  \, , \quad \forall \sigma  \; \in \; [S_1, U_1)\, .
\end{eqnarray}

In principle, one can hope to continue defining the time change value $\Phi(t)$ by extending the same reasoning for times $t \geq T_1$, leading to defining an intertwined sequence of blowup trigger times $\{ S_k \}_{k \geq 1}$ and blowup exit time $\{ U_k\}_{k \geq 1}$, with $U_1 \leq S_1 \leq U_2 \leq S_2 \ldots$. 
Such a construction produces a global solution over $\mathrm{R}^+$ if the resulting sequences are devoid of accumulation points: $\lim_{k \to \infty} S_k=\lim_{k \to \infty} U_k=\infty$.
In this work, we are primarily concerned with showing that the latter limits hold for large enough interactions and natural initial conditions so that there are no caveats due to accumulation of blowup times.

However, one can specify the inverse time change $\Phi$ independent of the possible accumulation of blowup times.
Following on the informal discussion above, this implicit specification of $\Phi(t)=\sigma$ as a solution $\nu t = \sigma - \lambda G(\sigma)$ is best formulated in terms of the inverse time change $\Psi$ featuring in \eqref{eq:Psi1} and \eqref{eq:Psi2}.
Then to circumvent possible caveats of accumulating blowup times, we leverage the fact that $\Psi$ must be a nondecreasing function to write relation \eqref{eq:Phi} as a fixed-point problem, whereby $\Psi$ is defined as a running maximum function.
The fixed-point nature of $\eqref{eq:Phi}$ naturally follows from the fact that the cumulative function $G=G[\Psi]$ also depends on $\Psi$ via $\eta=\eta[\Phi]=\eta[\Psi]$.
Specifically, we show in \cite{TTPW} that:

\begin{proposition}\label{prop:fixedPoint}
The inverse time-change $\Psi$ parametrizes a dPMF dynamics if and only if it solves the fixed-point problem
\begin{eqnarray}\label{eq:fixedPoint}
\forall \; \sigma \geq 0 \, , \quad \Psi(\sigma) = \sup_{0 \leq \xi \leq \sigma} \big( \xi - \lambda G[\Psi](\xi)\big) / \nu \, .
\end{eqnarray}
\end{proposition}

The full formulation of the fixed-point problem given in Definition \ref{def:fixedpoint_int} follows directly from the above proposition together with the renewal-type equation \eqref{eq:DuHamel} for the cumulative flux $G$ and Definitions \ref{def:Eta} and \ref{def:xitauDef} for the backward functions $\eta$ and $\xi$.
The fixed point problem \ref{def:fixedpoint_int} allows one to characterize the occurrence of a single blowup in the time changed-picture for generic initial conditions satisfying Assumption \ref{assump1} and under full-blowup Assumption \ref{assump3}.
This local result is stated as follows in \cite{TTPW}:

\begin{theorem}\label{th:smooth}
Under Assumption \ref{assump1}, there is a unique smooth solution $\Psi$ to the fixed-point problem \ref{def:fixedpoint_int} up to the possibly infinite time 
\begin{eqnarray}\label{eq:defS0}
S_1
=
\inf \left\{ \sigma>0 \, \big \vert \, \Psi'(\sigma)  \leq 0 \right\}   
=
\inf \left\{ \sigma>0 \, \big \vert \, g[\Psi](\sigma)  \geq 1/\lambda \right\}  > 0 \, .
\end{eqnarray}
\end{theorem}

\begin{theorem}\label{th:jump}
Under Assumption \ref{assump1} and \ref{assump3}, the solution $\Psi$ to the fixed-point problem \ref{def:fixedpoint_int} can be uniquely extended as a constant function on $[S_1,U_1]$ with $U_1=S_1+\lambda \pi_1$ where $\pi_1$ satisfies $0<\pi_1<\Vert q(S_1, \cdot) \Vert_1 \leq 1$ and is defined as 
\begin{eqnarray}\label{eq:defjump}
\pi_1 = \inf \left\{ p \geq 0 \, \bigg \vert \, p > \int_{0}^\infty H(\lambda p, x )  q(S_1,x ) \, \dd x \right\} \, .
\end{eqnarray}
\end{theorem}


We show our two main results in Section \ref{sec:global} and in Section \ref{sec:limRef}. 
In preparation of these, let us state a few properties that solutions to the fixed-point problem \ref{def:fixedpoint_int} must satisfy.
The first such property shows that for any solutions $\Psi$,  the associated backward-delay functions $\eta[\Psi]$ must be uniformly bounded.

\begin{proposition}
For all solutions $\Psi$ to the fixed-point problem \ref{def:fixedpoint_int}, the associated backward-delay functions $\eta$ are bounded by $\lambda+\nu \epsilon$.
\end{proposition}

\begin{proof}
Solutions $\Psi$  to the fixed-point problem \ref {def:fixedpoint_int} are necessarily continuous by smoothness of the associated cumulative function $G[\Psi]$.
Moreover by definition of the backward function $\xi$, such solutions must satisfy $\Psi(\xi(\sigma)) = \Psi(\sigma) - \epsilon$.
This implies that
\begin{eqnarray}
\nu \epsilon =  \nu \big( \Psi(\sigma)- \Psi(\xi(\sigma)) \big) = \sigma-\xi(\sigma) - \lambda \big( G(\sigma)-G(\xi(\sigma)) \big) \nonumber \, .
\end{eqnarray}
We conclude by remembering that $\eta = \mathrm{Id}-\xi$ and that by conservation of probability, we must have $G(\sigma)-G(\xi(\sigma)) \leq 1$.
\end{proof}

The second property states that global inverse time change $\Psi$  provide us with global time changes $\Phi=\Psi^{-1}: \mathbbm{R}^+ \to \mathbbm{R}^+$ for the dPMF dynamics. 
Suppose that the fixed-point problem \ref {def:fixedpoint_int} admits an inverse time change $\Psi$ as a solution.
Such a solution cannot explode in finite time as it is nondecreasing with bounded difference quotient: $w_\Psi \leq 1/\nu$.
Let us further assume that it is a global solution, i.e., that $\Psi$ is defined on the whole set $\mathbbm{R}^+$.
Then, by boundedness of the backward-delay function $\eta = \mathrm{Id}-\xi$, it must be that $S_\infty = \lim_{\sigma \to \infty} \xi(\sigma) = \infty$.
For finite refractory period $\epsilon>0$, it turns out that this observation is enough to show the function $\Phi=\Psi^{-1}$ is also a global time change, as stated in the following proposition.

\begin{proposition}\label{prop:horizon}
For $\epsilon>0$, solutions to the fixed-point problem \ref {def:fixedpoint_int} that are defined over $\mathbbm{R}^+$ necessarily satisfy $T_\infty=\lim_{\sigma \to \infty} \Psi(\sigma)=\infty$.
\end{proposition}

\begin{proof}
Given a solution $\Psi:\mathbbm{R}^+ \to \mathbbm{R}^+$  to the fixed-point problem \ref {def:fixedpoint_int}, suppose that $T_\infty = \lim_{\sigma \to \infty} \Psi(\sigma) < \infty$.
Then, as $\Psi$ and $\Phi=\Psi^{-1}$ are nondecreasing, for all $\epsilon>0$
\begin{eqnarray}
S_\infty = \lim_{\sigma \to \infty} \xi(\sigma) = \lim_{\sigma \to \infty} \Phi(\Psi(\sigma)-\epsilon) =  \lim_{t \to T_\infty^-} \Phi(t-\epsilon) \leq  \Phi(T_\infty^--\epsilon) < \infty \, . \nonumber
\end{eqnarray}
This contradicts the fact that $\eta = \mathrm{Id}-\xi$ is uniformly bounded on $\mathbbm{R}^+$.
Thus we must have $T_\infty=\infty$.
\end{proof}

Note that the above result only holds for nonzero refractory period  $\epsilon>0$.
Establishing a similar result for vanishing refractory period $\epsilon=0$ will require a detailed analysis, which we will conduct only for large enough interaction parameters $\lambda > \Lambda$.

\section{Global blowup solutions}\label{sec:global}

In this section, we establish the existence and uniqueness of explosive dPMF dynamics over the whole half-line $\mathbbm{R}^+$ in the large interaction regime.
First, we introduce useful \emph{a prori} estimates about the fluxes associated to the fixed-point solutions $\Psi$ for a natural class of initial conditions. 
Second, we utilize these estimates to show that in the large interaction regime, a large enough blowup is followed by a subsequent blowup in a well ordered fashion, in the sense that this next blowup can only happen after all the processes that previously blew up have reset.
Third, we show that the size of the next blowup is lower bounded away from zero, which allows one to establish the existence and uniqueness of explosive dPMF dynamics over the the whole half-line $\mathbbm{R}^+$ by induction on the number of blowup episodes.


\subsection{{\emph A priori} bounds for natural initial conditions}\label{sec:apriori}
Given normalized initial conditions $(q_0,g_0)$ in $\mathcal{M}((0,\infty)) \times \mathcal{M}([\xi_0, 0))$, assume that the fixed-point problem  \ref {def:fixedpoint_int} admits a solution $\Psi$ up to some possibly infinite time $S_\infty$.
Then, the associated instantaneous flux function $g$ satisfies the renewal-type equation \eqref{eq:gRenew} on $[0,S_\infty)$.
We aim at utilizing this characterization to show that for a restricted class of initial conditions, $g$ is uniformly bounded on $[0,S_\infty)$, with an upper bound that is independent of the coupling parameter $\lambda$ and that scales with the $L_1$ norms of the initial conditions.
Informally, the restricted initial conditions at stake are these distributions of time-changed processes that have been reset at some time in the past.
We will refer to this restricted set of initial conditions as {\em natural initial conditions}, since they turn out to have the same regularities as generic dPMF solutions.
We formally define natural initial conditions as follows:

\begin{definition}
Natural initial conditions $(q_0,g_0)$ in $\mathcal{M}((0,\infty)) \times \mathcal{M}([\xi_0, 0))$ are those for which there exists a measure $\mu_0$ in $\mathcal{M}((0,\infty))$ such that:

$(i)$ The measure $q_0$ corresponds to the density 
\begin{eqnarray}
q_0(x) = \int_0^\infty  \frac{\kappa(\tau, x, \Lambda)}{1-H(\tau,\Lambda)} \, \mu_0( \dd \tau )   \, , \nonumber
\end{eqnarray}

$(ii)$ The measure $g_0$ corresponds to the density 
\begin{eqnarray}
g_0(\sigma)=\int_0^\infty \frac{h(\sigma+\tau,\Lambda)}{1-H(\tau,\Lambda)} \, \mu_0( \dd \tau )  \, . \nonumber
\end{eqnarray}

$(iii)$ We have the degenerate normalization condition
\begin{eqnarray}
\int_0^\infty  \mu_0(\dd \tau)  + \int_{\xi_0}^0 g_0(\tau) \, \dd \tau  \leq 1 \, . \nonumber
\end{eqnarray}
\end{definition}

We derive the above formal definition for natural initial conditions by considering absorbed, drifted Wiener dynamics that have started at $\Lambda$ at some time in the past with cumulative rate function $R$.
Such dynamics admit density functions under the integral form 
\begin{eqnarray}
q(\sigma,x) = \int_{-\infty}^0 \kappa(\sigma-\xi, x, \Lambda) \, \dd R(\xi) \, , \nonumber
\end{eqnarray}
where integration by part guarantees convergence for all cumulative rate functions with, say, polynomial growth.
Note that we only requires $R$ to be a nondecreasing \emph{c\`adl\`ag} function for being a cumulative rate function.
The instantaneous flux function associated to this density is 
\begin{eqnarray}
g(\sigma) = \partial_x q(\sigma,0)/2 = \int_{-\infty}^0 h(\sigma-\xi, \Lambda) \, \dd R(\xi) \, . \nonumber
\end{eqnarray}
From there, the parametrization of natural initial conditions is recovered by setting
\begin{eqnarray}
\mu_0(\dd \tau) = \dd R(-\tau) (1-H(\tau,\Lambda)) \, , \nonumber
\end{eqnarray}
and recognizing that $q_0(x)=q(0,x)$ for $x>0$ and $g_0(\sigma)=g(\sigma)$, $-\xi_0 \leq \sigma < 0$.
Bearing the above remarks in mind, we have the following interpretation for the parametrizing measure $\mu_0$:
given a fraction of active processes distributed according to $q_0$ at time $\sigma=0$, $\mu_0$ represents the distribution of times elapsed since their last resets.
In particular, definition $(i)$ implies that $\int_0^\infty  \mu_0(\tau) \, \dd \tau = \int_0^\infty  q_0(x) \, \dd x$.
Furthermore, by definition $(i)$ and $(ii)$ and the remarks made above, we have
\begin{eqnarray}
\lim_{\sigma \to 0^-} g_0(\sigma)=\partial_\sigma q_0(0)/2 \, , \nonumber
\end{eqnarray}
so that conservation of probability holds as in \eqref{eq:abscons_int}. 
The above observation directly shows that given any natural initial conditions $(q_0,g_0)$ in $\mathcal{M}((0,\infty)) \times \mathcal{M}([\xi_0, 0))$, 
the current state of a solution dPMF dynamics specifies natural initial conditions
\begin{eqnarray*}
\big(q(\sigma,\cdot) , \{ g(\xi -\sigma)\}_{-\eta(\sigma)\leq \xi < 0} \big) \quad \in \quad \mathcal{M}((0,\infty)) \times \mathcal{M}([-\eta(\sigma), 0)) \, .
\end{eqnarray*}
In other word, natural initial conditions are stabilized by dPMF dynamics.
Finally, note that we do not require normalization to one in $(iii)$ so that natural initial conditions can represent a smaller-than-one fraction of the processes.
This latter point will be useful to discuss the following \emph{a priori} estimates for the instantaneous flux associated to solutions of the fixed-point equation \ref {def:fixedpoint_int}:

\begin{proposition}\label{prop:gBound}
For natural initial conditions $(q_0, g_0)$ in $\mathcal{M}((0,\infty)) \times \mathcal{M}([\xi_0, 0))$, we have
\begin{eqnarray}
\Vert g \Vert_{[0,\infty)} \leq  A_\Lambda  \, \left( \int_0^\infty q_0(x) \, \dd x + \int_{\xi_0}^0 g_0(\sigma)\, \dd \sigma  \right) = A_\Lambda  \, \Vert (q_0, g_0)\Vert_{1} \, . \nonumber
\end{eqnarray}
for a constant $A_\Lambda$ that only depends on $\Lambda$.
\end{proposition}

\begin{proof}
Considering  natural initial conditions allows one to utilize the Markov property of the first-passage kernels $\kappa$ to write:
\begin{eqnarray}
\int_0^\infty h(\sigma,x) q_0(x) \, \dd x 
&=& 
\int_0^\infty  \frac{1}{2} \partial_y \kappa(\sigma, 0, x) q_0(x) \, \dd x  \, , \nonumber\\
&=&
\int_0^\infty  \left( \int_0^\infty  \frac{1}{2} \partial_y \kappa(\sigma, 0, x) \kappa(\tau, x, \Lambda) \, \dd x \right) \frac{\mu_0(\tau) \,  \, \dd \tau}{1-H(\tau,\Lambda)}    \, , \nonumber\\
&=&
\int_0^\infty  \left( \int_0^\infty  \frac{1}{2} \partial_y \kappa(\tau+\sigma, 0,  \Lambda) \, \dd x \right) \frac{\mu_0(\tau) \,  \, \dd \tau}{1-H(\tau,\Lambda)}    \, , \nonumber\\
&=&
\int_0^\infty  \frac{h(\tau+\sigma, \Lambda) }{1-H(\tau,\Lambda)} \mu_0(\tau)   \, \dd \tau  \, . \nonumber
\end{eqnarray}
Injecting the above in equation \eqref{eq:gRenew}, we deduce the inequality 
\begin{eqnarray}
\vert g(\sigma) \vert
&\: \leq \: & 
\bigg \Vert \frac{h(\cdot+\sigma, \Lambda)}{1-H(\cdot,\Lambda)} \bigg \Vert_\infty \int_0^\infty  \mu_0(\tau) \, \dd \tau + \nonumber\\
&&
 \quad \quad \bigg \Vert  \int_0^\sigma  h(\sigma - \tau,\Lambda) \, \dd G( \xi(\tau)) \bigg \Vert_\infty  \, , \nonumber\\
 &\: \leq \: & 
\bigg \Vert \frac{h(\cdot+\sigma, \Lambda)}{1-H(\cdot,\Lambda)} \bigg \Vert_\infty \int_0^\infty  \mu_0(\tau) \, \dd \tau + \nonumber\\
&&
 \quad \quad \bigg \Vert \frac{h(\cdot, \Lambda)}{1-H(\cdot,\Lambda)} \bigg \Vert_\infty \int_0^\sigma  (1-H(\sigma - \tau,\Lambda)) \, \dd G( \xi(\tau)) \, , \nonumber
\end{eqnarray}
where both uniform norms are finite.
Indeed, for all $\sigma \geq 0$, L'Hospital rule yields that
\begin{eqnarray}
\lim_{\tau \to \infty} \frac{h(\tau + \sigma, \Lambda)}{1-H(\tau,\Lambda)} = \frac{e^{-\sigma/2}}{2} \, ,  \nonumber
\end{eqnarray}
which shows that the infinity norms are finite at fixed $\sigma$.
Moreover, exploiting the fact that $h$ is unimodal with maximum value $h^\star_\Lambda$ at $\sigma^\star = (-3+\sqrt{9+4\Lambda})/2$ and that $H(\cdot,\Lambda)$ is increasing,
we have for $\sigma \geq \sigma^\star$
\begin{eqnarray}
 \bigg \Vert \frac{h(\cdot+\sigma, \Lambda)}{1-H(\cdot,\Lambda)} \bigg \Vert_\infty
\leq 
 \bigg \Vert \frac{h(\cdot, \Lambda)}{1-H(\cdot,\Lambda)} \bigg \Vert_\infty \stackrel{\mathrm{def}}{=} M < \infty \, , \nonumber
\end{eqnarray}
whereas for $\sigma < \sigma^\star$, we have
\begin{eqnarray}
 \bigg \Vert \frac{h(\cdot+\sigma, \Lambda)}{1-H(\cdot,\Lambda)} \bigg \Vert_\infty
\leq 
\frac{\Vert h(\cdot , \Lambda) \Vert_\infty}{1-H(\sigma^\star,\Lambda)}  
=
\frac{h^\star_\Lambda}{1-H(\sigma^\star,\Lambda)} 
\stackrel{\mathrm{def}}{=}  N < \infty \, . \nonumber
\end{eqnarray}
Thus, we have
\begin{eqnarray}\label{eq:MN}
\vert g(\sigma) \vert
&\leq& 
\max(M,N) \int_0^\infty q_0(x) \, \dd x
+ 
M \int_0^\sigma  (1-H(\sigma-\tau,\Lambda)) \, \dd G( \xi(\tau)) \, , 
\end{eqnarray}
where the constant $M$ and $N$ only depends on $\Lambda$.
To bound the integral term in the inequality above, we make two observations:
First, we write the renewal-type equation \eqref{eq:DuHamel} for the cumulative flux $G$ under the form
\begin{eqnarray}
\int_0^\sigma  H(\sigma-\tau,\Lambda)) \, \dd G( \xi(\tau)) = G(\sigma) - G(0) - \int_0^\infty H(\sigma,x)  q_0(x) \, \dd x \, . \nonumber
\end{eqnarray}
Second, we write the conservation of probability for processes originating from $\mu_0$ as 
\begin{eqnarray}
G( \sigma) - G(\xi(\sigma))   - \big[ G(0)-G(\xi_0) \big] = P(\sigma) - P(0) \, , \nonumber
\end{eqnarray}
where $P$ denotes the fraction of processes originating from $\mu_0$ that is inactive.
These two observations allows one to express the integral term in  inequality \eqref{eq:MN} as
\begin{eqnarray}
\int_0^\sigma(1-H(\sigma-\tau,\Lambda)) \, \dd G( \xi(\tau))
=
P(0) - P(\sigma) + \int_0^\infty H(\sigma,x)  q_0(x) \,  \dd x \, . \nonumber
\end{eqnarray}
Therefore, as $\Vert H \Vert_\infty \leq 1$, we have
\begin{eqnarray}
\Vert g \Vert_{[0,\infty)} 
&\leq& 
\max(M,N) \int_0^\infty q_0(x) \, \dd x + M \left(   P(0) + \int_0^\infty q_0(x) \, \dd x  \right) \, . \nonumber
\end{eqnarray}
We conclude by observing that by conservation of probability mass $ P(0)  \leq \Vert (q_0,g_0)\Vert_{1}$, so that
\begin{eqnarray}
\Vert g \Vert_{[0,\infty)} 
&\leq& 
3 \max(M,N) \Vert (q_0,g_0)\Vert_{1} \, . \nonumber
\end{eqnarray}
\end{proof}

The arguments of the above proof can be adapted to obtain a uniform bound to $\partial_\sigma g$, which consitutes another useful \emph{a priori} estimate :

\begin{proposition}\label{prop:dgBound}
For natural initial conditions $(q_0,g_0)$ in $\mathcal{M}((0,\infty)) \times \mathcal{M}([\xi_0, 0))$, we have
\begin{eqnarray}
\Vert \partial_\sigma g \Vert_{[0,\infty)} \leq  B_\Lambda  \, \left( \int_0^\infty q_0(x) \, \dd x + \int_{\xi_0}^0 g_0(\sigma)  \right) = B_\Lambda  \, \Vert (q_0, g_0)\Vert_{1} \, . \nonumber
\end{eqnarray}
for a constant $B_\Lambda$ that only depends on $\Lambda$.
\end{proposition}

\begin{proof}
The proof proceeds in exactly the same fashion as the proof of Proposition \eqref{prop:gBound} but starting from the renewal-type equation
\begin{eqnarray}
\partial_\sigma g(\sigma) 
= 
\int_0^\infty \partial_\sigma h(\sigma,x) q_0(x) \, \dd x
+
\int_0^\sigma  \partial_\sigma h(\sigma-\tau,\Lambda) \, \dd G( \xi(\tau)) \, , \nonumber
\end{eqnarray}
where we observe that
\begin{eqnarray}
\lim_{\tau \to \infty} \frac{\vert \partial_\sigma h(\tau + \sigma, \Lambda) \vert}{1-H(\tau,\Lambda)} = \frac{e^{-\sigma/2}}{4} \, . \nonumber
\end{eqnarray}
\end{proof}

The above \emph{a priori} bounds will be crucial to exhibit global solutions to the fixed-point problem \ref{prop:fixedPoint} that exhibit an infinite number of blowups.
Their main use will be in showing that for blowup of large enough size $0<\pi_1<1$, the small fraction $1-\pi_1$ of surviving active processes cannot trigger another blowup before the original fraction $\pi_1$ has reset.
In other words, this will guarantee that blowups are well ordered.

\subsection{Well-ordered blowups for large interactions}

Given natural initial conditions $(q_0, g_0)$ in $\mathcal{M}((0,\infty)) \times \mathcal{M}([\xi_0, 0))$, consider a solution to the fixed-point problem \ref{def:fixedpoint_int} with associated instantaneous flux $g$.
By Propositions \ref{prop:gBound} and \ref{prop:dgBound}, we know that the flux $g$ and its derivative $\partial_\sigma g$, are both uniformly bounded on the domain of the solution. 
Here, we aim at refining the bounding analysis of $g$ to ensure that for large enough synchronization events, blowups are well ordered in the sense that no blowup can occur before the reset of processes that synchronized in the past.  
Large synchronization events inactivate so many processes that the remaining processes are too few to trigger a blowup on their own.

To prove this, we distinguish between two contributions to the current state of the dynamics, according to whether processes have inactivated in the most recent blowup or not.
Specifically, assume that a blowup of size $\pi_1$ triggers at $S_1 \geq 0$ with exit time $U_1$ and reset time $R_1$.
At time $\sigma \geq R_1$, we can always split the current state of the dynamics into a contribution from the synchronized processes that reset at $R_1$ (marked by a subscript $\pi_1$) and a contribution from the other processes (marked by a subscript $\cancel{\pi_1}$).
This means that there are $(q_{\pi_1}, g_{\pi_1})$ and $(q_{\cancel{\pi_1}}, g_{\cancel{\pi_1}})$ in $\mathcal{M}((0,\infty)) \times \mathcal{M}([\xi(\sigma), \sigma))$ such that for all $\sigma \geq R_1$
\begin{eqnarray}
q(\sigma, \cdot) = q_{\pi_1}(\sigma, \cdot)+q_{\cancel{\pi_1}}(\sigma, \cdot) \quad \mathrm{and} \quad g(\sigma)=g_{\pi_1}(\sigma)+g_{\cancel{\pi_1}}(\sigma) \, , 
\end{eqnarray}
with $q_{\pi_1}(R_1,\cdot)=\pi_1 \delta_\Lambda$ and $g_{\pi_1}(\sigma)=0$ for $\sigma \leq R_1$ and with $q_{\cancel{\pi_1}}(R_1,\cdot)=q(R_1^-, \cdot)$ and $g_{\cancel{\pi_1}}(\sigma)=g(\sigma)$ for $U_1 \leq \sigma< R_1$.
Given a fixed backward function $\xi$, it is clear that the governing renewal-type equation \eqref{eq:gRenew} applies to both contribution $(q_{\pi_1}, g_{\pi_1})$ and $(q_{\cancel{\pi_1}}, g_{\cancel{\pi_1}})$ separately. 
As a result, Propositions \ref{prop:gBound} and \ref{prop:dgBound} also apply to both contributions separately, allowing to control the role of each contribution in triggering the next possible blowup. 

The next proposition states this point concisely by making use of the real numbers $\sigma^\dagger$ and $h^\dagger$ defined by
\begin{eqnarray}
\sigma^\dagger = \inf \{ \sigma>0 \, \vert \, \partial^2_\sigma h(\sigma,\Lambda) = 0 \} \quad \mathrm{and} \quad h^\dagger_\Lambda=h( \sigma^\dagger,\Lambda)\, . \nonumber
\end{eqnarray}
Observe that $\sigma^\dagger$ and $h^\dagger_\Lambda$ only depends on $\Lambda$, and that $\sigma^\dagger$ satisfies $0<\sigma^\dagger < \sigma^{\star}$, where $\sigma^\star$ is the unique maximizer of $h(\cdot, \Lambda)$:  $h(\sigma^\star,\Lambda)=h^\star_\Lambda>h^\dagger_\Lambda$.
Moreover, by definition of $\sigma^\dagger$,  observe that $\partial_\sigma h(.,\Lambda)$ is increasing on $[0,\sigma^\dagger]$.
Then, equipped with $h^\star_\Lambda$ and $\sigma^\dagger$, the following proposition states that for large enough interaction parameters $\lambda$, a sufficiently large blowup $\pi_1$ will trigger a next blowup in a well ordered fashion.

\begin{proposition}\label{prop:nextBlowup}
Given  an interaction parameter $\lambda$ with $\lambda>l_\Lambda = 1/h^\dagger_\Lambda$, there exists a constant $C_\Lambda$ that only depends on $\Lambda$ such that a full blowup of size $\pi_1$ with
\begin{eqnarray}\label{eq:pi0}
 \pi_1 > \max \left( \frac{l_\Lambda}{\lambda}, 1- \frac{C_\Lambda}{\lambda}\right)\, ,
\end{eqnarray}
is followed by another full blowup in finite time. Moreover, if the first blowup happens at time $S_1$, the next blowup time $S_2$ is such that $S_1 \leq R_1 < S_2 \leq R_1+\sigma^\dagger$, where we have defined the reset time $R_1=\tau(U_1)$.
\end{proposition}

\begin{proof} The proof proceeds in four steps: $(i)$ we show that a blowup cannot happen before the fraction $\pi_1$ of processes reset, $(ii)$ we show that the fraction $\pi_1$ of reset processes is enough to trigger a blowup, $(iii)$ we show the full blowup condition, and $(iv)$ we conclude by specifying the choice of $l_\Lambda$ and  $C_\Lambda$.

$(i)$
Suppose that the last blowup triggers at time $S_1$ with blowup size $\pi_1$, then at blowup exit time $U_1=S_1+\lambda \pi_1$,  the exit states
\begin{eqnarray}
\big(q(U_1, \cdot), \{g(\sigma)\}_{\xi(U_1) \leq \sigma < U_1}\big) \nonumber
\end{eqnarray}
determine natural initial conditions.
These are such that $G(U_1)-G(S_1) = \int_{S_1}^{U_1} g(\sigma) \, \dd \sigma = \pi_1$, representing the fraction of processes to be instantaneously reset at time $R_1=\tau(S_1)$.
Let us consider the part of the exit initial conditions excluding this fraction, i.e., the active processes that either survive the blowup or that are already inactive before the blowup:
\begin{eqnarray}
(q_{\cancel{\pi_1},0},g_{\cancel{\pi_1},0})=\big(q(U_1, \cdot), \{g(\sigma)\}_{\xi(U_1) \leq \sigma < S_1}\big) \, , \quad \mathrm{with} \quad \Vert (q_{\cancel{\pi_1},0},g_{\cancel{\pi_1},0})\Vert_{1} \leq 1-\pi_1 \, . \nonumber
\end{eqnarray}
Choosing $C_\Lambda<1/(4A_\Lambda)$  in \eqref{eq:pi0} implies that $1-\pi_1<1/(4A_\Lambda \lambda)$
so that by Proposition \ref{prop:gBound}, the instantaneous flux $g_{\cancel{\pi_1}}$ due to the linear dynamics of the processes arising from the partial initial conditions $(q_{\cancel{\pi_1},0},g_{\cancel{\pi_1},0})$ is bounded above with $\Vert g_{\cancel{\pi_1}} \Vert_{[U_1,\infty)} \leq 1/(4\lambda)$.
This implies that no blowup can occur before resetting the fraction of inactivated processes $\pi_1$ as the blowup condition ($g(\sigma) = 1/\lambda$) cannot be met.
Moreover, before the reset time $R_1$, the inverse time change $\Psi$ only depends on the partial partial initial conditions $(q_{\cancel{\pi_1},0},g_{\cancel{\pi_1},0})$ so that we have
\begin{eqnarray}
\epsilon 
&=& \Psi(R_1)-\Psi(U_1) \,  ,\nonumber \\
&=& \big(R_1-U_1-\lambda (G(R_1)-G(U_1) \big)/\nu \,  ,\nonumber \\
&\geq&(1-\lambda \Vert g_{\cancel{\pi_1}} \Vert_{[U_1,\infty)}) (R_1-U_1)/\nu \, ,\nonumber \\
&\geq&(R_1-U_1)/(2\nu) \, .\nonumber
\end{eqnarray}
Thus the reset of the processes inactivated during the last blowup must happen in finite time at $R_1$ with
\begin{eqnarray}\label{eq:tauS0U0}
\nu \epsilon < R_1 -U_1< 2\nu \epsilon \, .
\end{eqnarray}

$(ii)$ We denote by $g_{\pi_1}$ the instantaneous flux due the fraction $\pi_1$ of the processes that reset at forward time $R_1$.
For all $\sigma \geq R_1$, this flux satisfies
\begin{eqnarray}\label{eq:pi0ineq}
\pi_1 h(\sigma-R_1,\Lambda) \leq g_{\pi_1}(\sigma) \leq  \pi_1 \sum_{n=0}^{\infty}h(\sigma-R_1,n\Lambda) = \pi_1 \tilde{g}(\sigma-R_1, \Lambda)   \, ,
\end{eqnarray}
where left term is the flux due to the fraction $\pi_1$ excluding any reset after $R_1$, whereas the right term is the flux due to the fraction $\pi_1$ with instantaneous reset after $R_1$.
Remembering that by definition $h^\star_\Lambda= \sup_{\sigma \geq 0} h(\sigma,\Lambda)>h^\dagger_\Lambda$,
choosing $\lambda> 1/(h^\dagger_\Lambda \pi_1)$ implies that the blowup condition $g(\sigma) \geq 1/\lambda$ will be met at some time $S_2>R_1$.
The upper bound flux term $\tilde{g}$ in inequality \eqref{eq:pi0ineq} only depends on $\Lambda$ and is uniformly bounded with respect $\sigma$ in $\mathbbm{R}^+$.
Moreover, $\tilde{g}$ satisfies the renewal equation
\begin{eqnarray}
\tilde{g}(\sigma,\Lambda) = h(\sigma,\Lambda) + \int^\sigma_0 h(\sigma-\tau,\Lambda)\tilde{g}(\tau,\Lambda)\, \dd \tau \, , \nonumber
\end{eqnarray}
so that for all $0 \leq \sigma \leq \sigma^{\star}$, remembering that $\Vert \tilde{g} \Vert_\infty \leq A_\Lambda$, we have the upper bound 
\begin{eqnarray}
\tilde{g}(\sigma,\Lambda) 
\leq h(\sigma,\Lambda) + \Vert \tilde{g} \Vert_\infty \int^\sigma_0 h(\sigma-\tau,\Lambda) \, \dd \tau 
\leq h(\sigma,\Lambda) (1+ A_\Lambda \sigma^\star ) \, , \nonumber
\end{eqnarray}
where the last inequality follows from the fact that $h(\cdot,\Lambda)$ is increasing on $[0,\sigma^\star]$.
Therefore, defining $a_\Lambda = (1+ A_\Lambda \sigma^\star )^{-1}/2<1/2$, we have
\begin{eqnarray}
s_1<S_2-R_1<s_2 \, , \nonumber
\end{eqnarray}
where the delay times $s_1$ and $s_2$ only depend on $\Lambda$ and $\pi_1\lambda$ via:
\begin{eqnarray}\label{def:s1s2}
\quad s_1 = \inf \left\{ s>0 \, \bigg \vert \,  h(s,\Lambda) = \frac{a_\Lambda}{\pi_1 \lambda} \right\} 
\; \; \mathrm{and} \; \;
s_2 = \inf \left\{ s>0 \, \bigg \vert \,  h(s, \Lambda) = \frac{1}{ \pi_1 \lambda} \right\} \, .
\end{eqnarray}

$(iii)$ It remains to check the full-blowup condition, i.e., $\partial_\sigma g(S_2)>0$.
In this perspective, let us first establish a lower bound for $\partial_\sigma g_{\pi_1}(S_2)$.
Such a bound is obtained by differentiating \eqref{eq:gRenew} with respect to $\sigma$ for $R_1 \leq \sigma \leq R_1+\sigma^\dagger$, which yields
\begin{eqnarray}\label{eq:rentypeg1}
\partial_\sigma g_{\pi_1}(\sigma) 
= \pi_1 \partial_\sigma h(\sigma-R_1,\Lambda) + \int^\sigma_0 \partial_\sigma h(\sigma-\tau,\Lambda) \, g_{\pi_1}(\xi(\tau))  
\geq \pi_1 \partial_\sigma h(\sigma-R_1,\Lambda) \, , \nonumber
\end{eqnarray}
where we utilize the fact that $\partial_\sigma h(\cdot,\Lambda)$ is increasing on $[0,\sigma^\dagger]$.
Thus, under the condition that $s_2\leq \sigma^\dagger$, we have
\begin{eqnarray}
\partial_\sigma g_{\pi_1}(S_2) \geq \pi_1 \partial_\sigma h(S_2-R_1,\Lambda) \geq \pi_1 \partial_\sigma h(s_1,\Lambda) \, , \nonumber
\end{eqnarray}
where $s_1$ is defined by \eqref{def:s1s2}. Moreover, from the explicit expression of $h(\cdot, \Lambda)$, we evaluate
\begin{eqnarray}
\frac{\partial_\sigma h(s_1,\Lambda)}{h(s_1,\Lambda)}= \partial_\sigma \ln  h(s_1,\Lambda) = \frac{\Lambda^2-(3s_1+s_1^2)}{2s_1^2}  \, , \nonumber
\end{eqnarray}
so that by the definition of $s_1$ given in \eqref{def:s1s2} we have
\begin{eqnarray}
\pi_1 \partial_\sigma h(s_1,\Lambda) = \frac{\Lambda^2-(3s_1+s_1^2)}{2s_1^2 } \frac{1}{a_\Lambda  \lambda} \, . \nonumber
\end{eqnarray}
To show that the above quantity is nonnegative, it is enough to remember that under the condition that $s_2\leq \sigma^\dagger$, we have $s_1<s_2\leq \sigma^\dagger<\sigma^\star$. Moreover, the function $s_1 \mapsto (\Lambda^2-(3s_1+s_1^2))/(2s_1^2)$ is decreasing and strictly positive on $(0,\sigma^\star)$ since it has the same sign as $\partial_\sigma h(\cdot, \Lambda)$.
Thus, we have
\begin{eqnarray}
\partial_\sigma g_{\pi_1}(S_2) \geq \pi_1 \partial_\sigma h(s_1,\Lambda) \geq \frac{b_\Lambda}{ \lambda}  \, , \quad \mathrm{with} \quad b_\lambda=\frac{\Lambda^2-\left(3{\sigma^\dagger}+\sigma^{\dagger2}\right)} {2  \sigma^{\dagger2} a_\Lambda} >0\, , \nonumber
\end{eqnarray}
where $b_\Lambda$  only depends on $\Lambda$.
Choosing $C_\Lambda <b_\Lambda/(2B_\Lambda)$ in \eqref{eq:pi0} implies that $1-\pi_1<b_\Lambda/(2B_\Lambda \lambda)$
so that by Proposition \ref{prop:dgBound}, the absolute value of the term $\vert \partial_\sigma g_{\cancel{\pi_1}} \vert$ due to the linear dynamics of the processes arising from the partial initial conditions $(q_{\cancel{\pi_1,0}},g_{\cancel{\pi_1,0}})$ is bounded above with $\Vert \partial_\sigma g_{\cancel{\pi_1}} \Vert_{[U_1,\infty)} \leq b_\Lambda/(2 \lambda)$.
This shows that a full blowup happens in $S_2$ since
\begin{eqnarray}\label{eq:dgS1}
\partial_\sigma g(S_2) 
=
\partial_\sigma g_{\pi_1}(S_2) + \partial_\sigma g_{\cancel{\pi_1}}(S_2) 
\geq
\partial_\sigma g_{\pi_1}(S_2)-\vert \partial_\sigma g_{\cancel{\pi_1}}(S_2) \vert 
\geq
\frac{b_\Lambda}{2 \lambda}>0 \, .
\end{eqnarray}

$(iv)$ The blowup condition due to reset processes alone requires in $(i)$ that $C_\Lambda<1/(4A_\Lambda)$ and in $(ii)$ that $\pi_1 h^*>1/\lambda$.
The full blowup condition requires in $(iii)$ that $s_2 \leq \sigma^\dagger$, which is equivalent to $\pi_1 h^\dagger> 1/\lambda$, and that $C_\Lambda<b_\Lambda/(2B_\Lambda)$.
This shows that choosing 
\begin{eqnarray}\label{eq:choiceC}
l_\Lambda > \frac{1}{h^\dagger_\Lambda} \quad \mathrm{and} \quad 0 < C_\Lambda < \frac{1}{2} \min \left(\frac{1}{2A_\Lambda} ,\frac{b_\Lambda}{B_\Lambda} \right) \, ,
\end{eqnarray}
suffices to ensure that after a full blowup of size $\pi_1$ at time $S_1$ satisfying \eqref{eq:pi0}, the next blowup must happen in a well ordered fashion at time $S_2$ such that $S_1 \leq R_1 < S_2 \leq R_1+\sigma^\dagger$.
 \end{proof}


\subsection{Persistence of blowups for large interactions}

Proposition \ref{prop:nextBlowup} exhibits a criterion for a large enough full blowup to be followed by another full blowup in a well ordered fashion.
By well ordered, we mean that the next blowup must occur after the processes that last blew up have all reset.
In this context, for blowups to occur indefinitely in a sustained fashion, we further need to check that the size of the next blowup remains larger than that of the criterion of Proposition \ref{prop:nextBlowup}.
In this perspective, it is instructive to find an upper bound to the fraction of inactive processes at the trigger time of the next blowup $S_2$.
Indeed, processes that are inactive at $S_2$ will remain so throughout the blowup, and therefore cannot contribute to the next blowup of size $\pi_2$.
The next proposition shows that  the fraction of inactive process $P(S_2)$ can be made arbitrarily small for large enough $\lambda$.
This proposition only considers interaction parameter values $\lambda > \max(2 l_\Lambda,2/C_\Lambda)$, so that the constraint on the full blowup size in Proposition \ref{prop:nextBlowup} reduces to 
\begin{eqnarray}\label{eq:pi0'}
 \pi_1 > 1- \frac{C_\Lambda}{\lambda} > 1/2\, .
\end{eqnarray}

\begin{proposition}\label{prop:inactiveFrac}
Consider an interaction parameter $\lambda$ such that $\lambda>\max(2 l_\Lambda,2/C_\Lambda)$ and a dynamics for which a full blowup happens at time $S_1$ with size $ \pi_1 > 1- C_\Lambda/\lambda>1/2$.
If $\lambda$ further satisfies that $\lambda> 2 \sqrt{2 \pi \Lambda }  e^{\Lambda^2/2C_\Lambda}$ and if the refractory period is such that $\epsilon < C_\Lambda/(2\nu)$, then the next full blowup triggers at time $S_2<\infty$ with a fraction of inactive processes:
\begin{eqnarray}
P(S_2) \leq \frac{C_\Lambda}{2\lambda} \, . \nonumber
\end{eqnarray}
\end{proposition}

\begin{proof}
Let us first consider $s_2$, the upper bound to $S_2-R_1$ defined by \eqref{def:s1s2}.
Noticing that $h(\Lambda, \Lambda) = 1/(\sqrt{2 \pi \Lambda})$, by definition of $s_2$, if $\lambda>\sqrt{2 \pi \Lambda}/\pi_1$, then we necessarily have $s_2 < \Lambda$.
Moreover, if $s_2 < \Lambda$, we have
\begin{eqnarray}
\quad \frac{1}{\pi_1 \lambda}=h(s_2,\Lambda) = \frac{\Lambda}{\sqrt{2 \pi s_2^3}} \, e^{-\frac{(\Lambda-s_2)^2}{2 s_2}} > \frac{e^{-\frac{\Lambda^2}{2 s_2}} }{\sqrt{2 \pi \Lambda}}    \, , \nonumber
\end{eqnarray}
which together with $\pi_1>1/2$, implies the following upper bound on $s_2$:
\begin{eqnarray}
s_2 <   \frac{\Lambda^2}{2 \lambda \ln \left( \pi_1 \lambda / \sqrt{2 \pi \Lambda }\right)}  <  \frac{\Lambda^2}{2 \lambda \ln \left( \lambda / \big( 2\sqrt{2 \pi \Lambda } \big)\right)}  \, . \nonumber
\end{eqnarray}
Therefore the cumulative flux of inactivated processes between the last blowup exit time $U_1$ and the next blowup trigger time $S_2$ satisfies
\begin{eqnarray}
G(S_2)-G(U_1) &\leq& \Vert g \Vert_{U_1,R_1} (R_1-U_1) +\Vert g \Vert_{R_1,S_2} (S_2-R_1)  \, , \nonumber\\
&\leq& \frac{ \nu \epsilon}{2\lambda}  + \frac{\Lambda^2}{2 \lambda \ln \left( \lambda / \big(2\sqrt{2 \pi \Lambda } \big)\right)}\, , \nonumber
\end{eqnarray}
where we utilize that $\Vert g \Vert_{U_1,R_1}<1/(4\lambda)$ and $R_1-U_1<2\nu\epsilon$ by \eqref{eq:tauS0U0}.
Then for all $\epsilon < C_\Lambda/(2\nu)$, assuming that 
\begin{eqnarray}
\lambda> 2\sqrt{2 \pi \Lambda } e^{2\Lambda^2/C_\Lambda} \, , \nonumber
\end{eqnarray}
yields an upper bound on the fraction of inactive processes at the next blowup time $S_2$:
\begin{eqnarray}
P(S_2)=1-\int_0^\infty q(S_2,x) \, dx \leq G(S_2)-G(U_1) \leq \frac{C_\Lambda}{2\lambda} \, . \nonumber
\end{eqnarray}
\end{proof}

We are now in a position to show that for large enough interaction parameter $\lambda$, the next blowup will have a size $\pi_2$ that satisfies the criterion of Proposition \ref{prop:nextBlowup}.
This is a consequence of the following observations:
For well ordered blowup, the reset of a fraction  $\pi_1$ of processes at time $R_1$ is enough to trigger a blowup of finite size, which will be bounded below by a quantity $\Delta H_\Lambda/2>0$ that only depends on $\Lambda$.
Then, the duration of the blowup during which reset is halted is lower bounded by $\lambda \Delta H_\Lambda/2>0$.
In turn, for large enough $\lambda$, the probability of inactivation during blowup can be made exponentially small with respect to $\lambda$ by realizing that inactivation has an asymptotically constant, nonzero hazard rate function. 
Altogether, these observations allow us to form a self-consistent condition guaranteeing that $\pi_2$ satisfies the criterion of Proposition \ref{prop:nextBlowup}, leading to the following result:

\begin{proposition}\label{prop:persistence}
Assume that the refractory period satisfies $\epsilon < C_\Lambda/\nu$.
There exists a constant $L_\Lambda>0$, which only depends on $\Lambda$, such that for all interaction parameter $\lambda > L_\Lambda$, every full blowup at time $S_1$ with size $\pi_1 \geq 1-C_\Lambda/\lambda$ is directly followed by a full blowup at time $S_2$ with size $\pi_2 \geq 1-C_\Lambda/\lambda$.
\end{proposition}

\begin{proof}
The proof proceeds in three steps: $(i)$ we exhibit the asymptotically exponential regime of processes survival during blowup, $(ii)$ we show that this exponential regime is met for large enough interaction parameter $\lambda$, and $(iii)$ we exhibit a sufficient constant $L_\Lambda$.

$(i)$ We can always choose $\lambda$ large enough so that 
\begin{eqnarray}\label{eq:lambdaChoice}
\lambda \geq \max \big( 2 l_\Lambda, 2/C_\Lambda , 2 \sqrt{2 \pi \Lambda}  e^{\Lambda^2/2C_\Lambda} \big) \, .
\end{eqnarray}
Then by Proposition \ref{prop:nextBlowup}, the occurrence of a full blowup  in $S_1$ with size $\pi_1 \geq 1-C_\Lambda/\lambda$ triggers the next full blowup at a time $S_2$ with $R_1<S_2\leq R_1+\sigma^\dagger$. 
By Proposition \ref{prop:inactiveFrac}, the fraction of active process $1-P(S_2)$ at time $S_2$ is such that $1-P(S_2) > 1-C_\Lambda/2$. 
Let us consider a process $Y_\sigma$ that is active at time $S_2$ and let us denote by $\xi$ its next inactivation time.
Then for all $\sigma \geq \tau \geq S_1$, we have 
\begin{equation}
 \quad \Prob{\xi \geq \sigma \, \Big \vert \,  \xi \geq \tau} = \frac{\Prob{\xi \geq \sigma}}{\Prob{\xi \geq \tau}} = \frac{1-H(\sigma,\Lambda)}{1-H(\tau,\Lambda)} \, , \nonumber
\end{equation}
where the survival probabilities $1-H(\sigma,\Lambda)$ can be written in term of a hazard rate as
\begin{eqnarray}
1-H(\sigma,\Lambda) =\exp{\left(\int_0^{\sigma} \partial_\tau \ln \left( 1-H(\tau, \Lambda)\right) \, \dd\tau \right)} \, . \nonumber
\end{eqnarray}
The key observation is to notice that the hazard rate appearing above has the finite limit 
\begin{eqnarray}
\lim_{\sigma \to \infty} -\partial_\sigma \ln \left( 1-H(\sigma, \Lambda)\right) =  \lim_{\sigma \to \infty} \frac{h(\sigma, \Lambda)}{1-H(\sigma, \Lambda)} =  \frac{1}{2} \, , \nonumber
\end{eqnarray}
showing that for large time, inactivation asymptotically follows a memoryless exponential law for all active processes from the fraction $1-P(S_2)$.

$(ii)$ We now exploit the asymptotic exponential behavior of the hazard function associated to $1-H(\cdot, \Lambda)$ to show that $\pi_2$, the size of the next full blowup triggered in $S_2$, can be made arbitrary close to $1-P(S_2)$.
In this perspective, let us define the finite time
\begin{eqnarray}
\sigma^\sharp = \sup \left\{ \sigma \geq \sigma^\star \, \bigg \vert \,  \partial_\sigma \ln \left( 1-H(\sigma, \Lambda)\right)  <-\frac{1}{4}\right\} < \infty \, . \nonumber
\end{eqnarray}
For all $\sigma>\sigma^\sharp$, integrating the hazard rate inequality $\partial_\sigma \ln \left( 1-H(\sigma, \Lambda)\right) <-1/4$ yields
\begin{eqnarray}\label{eq:surv1}
\Prob{\xi \geq \sigma \, \Big \vert \,  \xi \geq \sigma^\sharp} =  \frac{1-H(\sigma)}{1-H(\sigma^\sharp)} \leq e^{-\frac{\sigma-\sigma^\sharp}{4}} \, , 
\end{eqnarray}
showing that the survival probability of active processes decays at least exponentially past time $S_2+\sigma^\sharp$.
Moreover, the fraction of processes $\Delta P_2$ that inactivates during the full blowup but before $S_2+\sigma^\sharp$ satisfies
\begin{eqnarray}
\Delta P_2 = P(S_2+\sigma^\sharp)-P(S_2)  \geq \pi_1 \int_{\sigma^\dagger}^{\sigma^\star} h(\sigma, \Lambda) \, \dd \sigma  \geq \Delta H_\Lambda/2 >0  \, . \nonumber
\end{eqnarray}
where we utilize the assumption that $\pi_1>1-C_\Lambda/\lambda>1/2$ and where we define the constant $\Delta H_\Lambda = \int_{\sigma^\dagger}^{\sigma^\star} h(\sigma, \Lambda) \, \dd \sigma>0$, which only depends on $\Lambda$.
The above lower bound follows from the fact that we choose $\lambda$ and $\pi_1$ so that $S_2 -R_1 \leq \sigma^\dagger < \sigma^\star \leq \sigma^\sharp$.
By \eqref{eq:surv1}, the lower bound $ \Delta H_\Lambda/2 \leq \Delta P_2 \leq \pi_2$ implies that 
\begin{eqnarray}
\Prob{\xi \geq \lambda \pi_2 \, \Big \vert \,  \xi \geq \sigma^\sharp} 
\leq
e^{-\frac{\lambda \pi_2-\sigma^\sharp}{4}} 
\leq
e^{-\frac{\lambda \Delta H_\Lambda -2 \sigma^\sharp}{8}}  \, . \nonumber
\end{eqnarray}
Thus, the fraction of inactivated process during blowup can be made arbitrarily close to $1-P(S_2)$ at the cost of choosing larger $\lambda$, as shown by:
\begin{eqnarray}
\pi_2
&=&
\Delta P_2 +\big(1-P(S_2+\sigma^\sharp)\big)  \Prob{\xi \leq \lambda \pi_2 \, \Big \vert \,  \xi \geq \sigma^\sharp} \, , \nonumber\\
&\geq&
\Delta P_2  +\big(1-P(S_2+\sigma^\sharp)\big)   \left(1-e^{-\frac{\lambda \Delta H_\Lambda -2 \sigma^\sharp}{8}} \right)   \, , \nonumber\\
&\geq&
 1-P(S_2)     - \big(1-P(S_2+\sigma^\sharp)\big)  e^{-\frac{\lambda \Delta H_\Lambda -2 \sigma^\sharp}{8}}  \xrightarrow{\lambda \to \infty} 1-P(S_2) \, . \nonumber
\end{eqnarray}

$(iii)$ It remains to exhibit a criterion for $\lambda$ ensuring that the blowup size $\pi_2$ is larger than $1-C_\Lambda/2$.
Such a criterion can be obtained by first observing that 
\begin{eqnarray}
\pi_2
\geq
 1-P(S_2)     - \big(1-P(S_2+\sigma^\sharp)\big)  e^{-\frac{\lambda \Delta H_\Lambda -2 \sigma^\sharp}{8}}  
\geq
 1-P(S_2)   -   e^{-\frac{\lambda \Delta H_\Lambda -2 \sigma^\sharp}{8}} \nonumber \, .
\end{eqnarray}
Then let us define $\lambda^\sharp>0$ as
\begin{eqnarray}
\lambda^\sharp = \sup \left\{\lambda \geq 0 \, \bigg \vert \,  \lambda \leq \frac{2}{\Delta H_\Lambda} \left(\sigma^\sharp-4 \ln{\left( \frac{C_\Lambda}{2\lambda} \right)}\right)  \right\} 
<  \frac{C_\Lambda}{2} e^{-\sigma^\sharp/4}
< \infty\, , \nonumber
\end{eqnarray}
such that for all $\lambda \geq \lambda^\sharp$, we have
\begin{eqnarray}
e^{-\frac{\lambda \Delta H_\Lambda -2 \sigma^\sharp}{8}} \leq  \frac{C_\Lambda}{2\lambda} \, .\nonumber
\end{eqnarray}
Utilizing the fact that $P(S_2)<C_\Lambda/(2\lambda)$  by Proposition \ref{prop:inactiveFrac}, we have
\begin{eqnarray}
\pi_2
 \geq
 1-P(S_2) -  \frac{C_\Lambda}{2\lambda} 
 \geq
1-  \frac{C_\Lambda}{\lambda}  \, . \nonumber
\end{eqnarray}
Thus choosing 
\begin{eqnarray}
\lambda \geq L_\Lambda =  \max \left(2 l_\Lambda,2/C_\Lambda,2 \sqrt{2 \pi \Lambda }  e^{\Lambda^2/2C_\Lambda}, \lambda^\sharp  \right) \nonumber
\end{eqnarray}
suffices to ensure that the next full blowup has size at least $1-C_\Lambda/\lambda$.
\end{proof}

The above proposition directly implies the existence and uniqueness of global explosive solutions under the following assumptions:

\begin{assumption}\label{mainAssump}
Consider a dPMF dynamics with refractory period $0 \leq \epsilon < C_\Lambda/\nu$, with interaction parameter $\lambda \geq L_\Lambda$, and with natural initial conditions such that $q_0$ contains an Dirac-delta mass $\pi_0 \delta_\Lambda$ with $\pi_0 \geq 1-C_\Lambda/\lambda$.
\end{assumption}

\begin{theorem}\label{th:globSol}
Under Assumption \ref{mainAssump} , the fixed-point problem \ref{def:fixedpoint_int} admits a unique global solution $\Psi$ defined over the whole half-line $\mathbbm{R}^+$. 
Moreover, this solution presents an infinite but discrete set of full blowups with trigger times $\{ S_k \}_{k \in \mathbbm{N}}$ and with sizes $\{ \pi_k \}_{k \in \mathbbm{N}}$ such that for all $k \geq 1$, $\pi_k \geq 1-C_\Lambda/\lambda$ and $S_{k+1}-U_k \geq \nu \epsilon$ where $U_k = S_k+\lambda \pi_k$.
\end{theorem}

\begin{proof}
By Theorem \ref{th:smooth}, the fixed-point problem \ref{def:fixedpoint_int} admits a local smooth solution up to time $S_1=\inf \left\{ \sigma>0 \, \vert \, \Psi'(\sigma)  \geq 0 \right\} $.
By the same arguments as in the proof of Proposition \ref{prop:nextBlowup}, the blowup condition must be satisfied in finite time: $S_1<\infty$. 
By the same arguments as in the proof of Proposition \ref{prop:persistence}, the blowup size satisfies $\pi_1>1-C_\Lambda/\lambda$.
From there on, one can iteratively apply Proposition \ref{prop:persistence} to define a global solution to the fixed-point problem \ref{def:fixedpoint_int} with an infinite number of blowups.
The uniqueness of the global solution follows from the uniqueness of the local solutions up to blowups and the uniqueness of the continuation process at the exit of a blowup.
Finally, this unique global solution $\Phi$ is defined over the whole half-line $\mathbbm{R}^+$ since all blowups are lower bounded by $1-C_\Lambda/\lambda$, so that all local solutions are defined over a domain of duration at least $\lambda(1-C_\Lambda/\lambda)$ during blowup episodes.
Finally, a blowup can only happen after resetting the processes that inactivated during the previous blowup, which must happens after a duration $\nu \epsilon$ by the proof of Proposition \ref{prop:nextBlowup}. 
\end{proof}

For $\epsilon>0$, Theorem \ref{th:globSol} directly implies the existence of global time change $\Phi$ parametrizing the original dPMF dynamics over the the half-line $\mathbbm{R}^+$ with the properties listed in Theorem \ref{th:main1}.
However, Theorem \ref{th:globSol} only implies that $T_{k+1} - T_k = \Psi(S_{k+1}) - \Psi(S_k) = \Psi(S_{k+1}) - \Psi(U_k) \geq \epsilon$, where $\{ T_k = \Psi(S_k)\}_{k \in \mathrm{N}}$ is the sequence of blowup times for the original dPMF dynamics.
Thus, in agreement with the remark following Proposition \ref{prop:horizon}, Theorem \ref{th:globSol} does not allow us to conclude about the existence of a global time change $\Phi$ in the limit $\epsilon \to 0^+$. 
Indeed, it could be that $\lim_{\epsilon \to 0^+} T_{k+1} - T_k = 0$ for all large enough $k$, which is compatible with accumulating the blowup times so that 
\begin{eqnarray}
T_\infty = \lim_{k \to \infty} T_k=\lim_{k \to \infty} \Psi(S_k)=\lim_{S \to \infty} \Psi(S)<\infty \, .
\end{eqnarray}
As a result, $\Phi=\Psi^{-1}$ would only define a solution time change $\Phi$ up the finite time $T_\infty$.
We address this point in the following section, where we characterize PMF dynamics (with zero refractory period) as limits of dPMF dynamics when $\epsilon \to 0^+$.

\section{Limit of vanishing refractory periods}\label{sec:limRef}

In this section, we show that explosive PMF dynamics can be defined consistently over the whole half-line $\mathbbm{R}^+$ for zero refractory period, i.e., with instantaneous reset.
First, we define PMF dynamics for zero refractory period $\epsilon =0$, which by contrast dPMF dynamics with $\epsilon > 0$, admit an explicit iterative construction.
Second, we introduce a series of continuity results showing that PMF dynamics are ``physical'' in the sense that they are recovered from dPMF dynamics in the limit $\epsilon \to 0^+$.
Finally, we provide the proofs for these continuity results, which essentially amounts to show that at all blowup trigger, exit, and reset times, the spatial densities of dPMF dynamics converge toward their PMF counterpart in $L_1$ norm.

\subsection{Solutions with zero refractory periods}

Theorem \ref{th:globSol} applies to the case of zero refractory period $\epsilon=0$, that is for PMF dynamics, assuming natural initial condition with normalized spatial component,
$\int_0^\infty q_0(x) \, \dd x = 1$, and no inactive processes.
In particular, for $\epsilon=0$, the fixed-point problem \ref{def:fixedpoint_int} admits a unique global solution for large enough interaction parameter $\lambda$ and for sufficiently large initial mass concentrated at $\Lambda$.
Let us denote this solution by $\Psi$ for simplicity.
As an inverse time change, $\Psi$ parametrizes a PMF dynamics with an infinite but discrete set of well ordered blowups with successive trigger times $\{ S_k \}_{k \in \mathbbm{N}}$ and corresponding size $\{ \pi_k \}_{k \in \mathbbm{N}}$.
A major benefit of considering  PMF dynamics is that the corresponding inverse time changes $\Phi$ admit an explicit iterative construction.
Specifically, the sequences $\{ S_k \}_{k \in \mathbbm{N}}$ and $\{ \pi_k \}_{k \in \mathbbm{N}}$ can be defined explicitely via the adjunction of the sequence of initial spatial conditions $\{q_{0,k}(\cdot)=q(S_k, \cdot) \}_{k \in \mathbbm{N}}$. 
 This explicit definition is made possible by the fact that for zero refractory period $\epsilon=0$, the renewal problems at stake loose their delayed character and can be solved analytically.
In particular, in between blowup episodes, we have the fundamental solution 
 \begin{eqnarray}
g(\sigma,x)=\sum_{k=0}^\infty h(\sigma,x+k\Lambda) \, , \nonumber
\end{eqnarray}
where  $g(\sigma,x)$ is the instantaneous flux of a process started in $x>0$ with zero delay reset.
This explicit fundamental solution allows one to define the announced sequence $\{S_k,\pi_k,q_{0,k}\}_{k \in \mathbbm{N}}$ in $\mathbbm{R}^+ \times (0,1] \times \mathcal{M}((0,\infty))$ as follows:

\begin{definition}\label{def:global0}
With the convention that $S_0=0$ and for a natural initial condition $q_0$ with Dirac-delta mass $\pi_0 \delta_\Lambda$ satisfying Assumption \ref{mainAssump}, let us define the sequence $\{S_k,\pi_k,q_{0,k}\}_{k \in \mathbbm{N}}$   by setting $q_{0,0}=q_0$ and iterating for all $k\geq 1$:
\begin{enumerate}
\item {\it Smooth dynamics}:
\begin{eqnarray}
g_k(\sigma) = \int_{0}^\infty g(\sigma,x) q_{0,k}(x) \, \dd x \, . \nonumber
\end{eqnarray}
\begin{eqnarray}
q_k(\sigma,y) = \int_0^{\sigma} \kappa(\sigma-\tau,y,\Lambda) g_k(\tau) \, \dd \tau  + \int_{0}^\infty \kappa(\sigma,y,x) q_{0,k}(x) \, \dd x \, . \nonumber
\end{eqnarray}
\item {\it Blowup trigger time}:
\begin{eqnarray}
S_{k+1} -S_k=  \inf \left\{ \sigma> 0\, \Big \vert \, g_k(\sigma) > 1/\lambda \right\} < \infty \, . \nonumber
\end{eqnarray}
\item  {\it Blowup size}:
\begin{eqnarray}
\pi_{k+1}= \inf \left\{ p \geq 0 \, \bigg \vert \, p > \int_{0^+}^\infty H(\lambda p, x)  q_k(S_{k+1},x) \, \dd x \right\} >0 \, . \nonumber
\end{eqnarray}
\item  {\it Blowup exit/reset distribution}: 
\begin{eqnarray}
q_{0,k+1}(y) =  \int_{0}^\infty \kappa(\lambda \pi_{k+1},y,x) q_k(S_{k+1},x) \, \dd x + \pi_{k+1} \delta_\Lambda(x) \, . \nonumber
\end{eqnarray}
\end{enumerate}
\end{definition}

The distributions $\{q_{0,k}\}_{k \in \mathbbm{N}}$ defined above  are in fact the reset distributions for PMF dynamics, $q_{0,k}=q(R_k,\cdot)$.
Moreover, with $\epsilon=0$, post-blowup reset is instantaneous so that $U_k=R_k$ and exit distributions are recovered by excluding the reset mass $\pi_{k}$: $q(R_k^-,\cdot)=q_{0,k}-\pi_k \delta_\Lambda$, .
The sequence $\{S_k,\pi_k,q_{0,k}\}_{k \in \mathbbm{N}}$ allows one to give an explicit piecewise representation of the global solution $\Psi$.
This representation makes use of the functions 
\begin{eqnarray}
\sigma \mapsto \Psi_k(\sigma) = \mathbbm{1}_{\{\sigma>0 \}} \left( \sigma - \lambda \int_0^x  G (\sigma,x)q_k(x) \, \dd x \right) / \nu \, ,  \nonumber
\end{eqnarray}
where $G(\cdot, x)$ denotes the cumulative flux function associated to  $g(\cdot, x)$.
Together with the trigger time $S_k$, the functions $\Psi_k$ allows one to specify the sequence of blowup time $\{ T_k \}_{k \in \mathbbm{N}}$ for the original dynamics as:
\begin{eqnarray}
T_0=0 \, , \quad T_{k+1} = T_k + \Psi_k\big(S_{k+1}\big) \, . \nonumber
\end{eqnarray}
We then obtain the following piecewise representation:

\begin{definition}\label{def:globaltime}
For zero refractory period $\epsilon=0$, the  global solution $\Psi$ is given by
\begin{eqnarray} 
\Psi(\sigma) = \sum_{k=0}^{\infty}  \mathbbm{1}_{\{ S_k<\sigma \le S_{k+1} \}} \left[ T_k  + \Psi_{k+1}\big(\sigma-(S_k + \pi_k) \big) \right] \, . \nonumber
\end{eqnarray}
\end{definition}

Observe that at this stage, although $\Psi$ is defined over the whole half-line $\mathbbm{R}^+$ by Theorem \ref{th:globSol}, we cannot invoke Proposition \ref{prop:horizon} to show that the corresponding time change $\Phi=\Psi^{-1}$ is defined over the whole half-line $\mathbbm{R}^+$.
However, we can establish this point by using the following estimate for inter-blowup times:

\begin{proposition}\label{prop:0est}
For natural initial condition $q_0$ with Dirac-delta mass $\pi_0 \delta_\Lambda$ satisfying Assumption \ref{mainAssump}, we have:
\begin{eqnarray}
\Psi(S_{k+1}) -\Psi(S_k) > \frac{1}{4 \nu} \inf \left\{ \sigma >0 \, \bigg \vert \, h(\sigma, \Lambda) > \frac{a_\Lambda}{\lambda} \right\} > 0 \, , \nonumber
\end{eqnarray}
where $a_\Lambda$ is a positive constant that only depends on $\Lambda$.
Thus, $T_\infty = \lim_{\sigma \to \infty} \Psi(\sigma) = \infty$ and $\Phi=\Psi^{-1}$ is defined over the whole half-line $\mathbbm{R}^+$.
\end{proposition}

\begin{proof}
Let us first observe that for $\epsilon=0$, we have $U_k=R_k$ and
\begin{eqnarray}
\nu \big( \Psi(S_{k+1}) -\Psi(S_k)  \big) 
&=&
\nu \big( \Psi(S_{k+1}) -\Psi(R_k)  \big) \, , \nonumber \\
&=&
S_{k+1} -R_k - \lambda \big( G(S_{k+1})-G(R_k) \big) \, . \nonumber
\end{eqnarray}
Then the proposed lower bound will follow from bounding above $G(S_{k+1})-G(R_k)$, the cumulative flux of inactivation  between the $k$-th blowup and the $(k\!+\!1)$-th blowup.
In this view, we split that cumulative flux as $G(S_{k+1})-G(R_k)=\Delta G_{\cancel{\pi_k},k}+\Delta G_{\pi_k,k}$, where $\Delta G_{\cancel{\pi_k},k}$ and $\Delta G_{\pi_k,k}$  represent the contributions of the processes that survive or inactivate during the last blowup, respectively.
As Assumption \ref{mainAssump} ensures that $\Vert g_{\cancel{\pi_k}} \Vert_{R_k,S_{k+1}} \leq 1/(4\lambda)$, we have
\begin{eqnarray}
\Delta G_{s,k} \leq \Vert g_{\cancel{\pi_k}} \Vert_{R_k,S_{k+1}}(S_{k+1}-R_k) \leq \frac{S_{k+1}-R_k}{4 \lambda} \, . \nonumber
\end{eqnarray}
Moreover, Assumption \ref{mainAssump} ensures that $g_{\pi_k}$ remains a convex function up to the next blowup so that $g_{\pi_k}(R_k)=0$ and $g_{\pi_k}(S_{k+1}) \leq 1/\lambda$ implies that
\begin{eqnarray}
\Delta G_{\pi_k,k} \leq  \frac{\Vert g_{\pi_k} \Vert_{R_k,S_{k+1}} (S_{k+1}-R_k)}{2} \leq \frac{S_{k+1}-R_k}{2 \lambda} \, . \nonumber
\end{eqnarray}
Considering the above bounds together leads to
\begin{eqnarray}
\nu \big( \Psi(S_{k+1}) -\Psi(S_k) \big)  = S_{k+1} -R_k - \lambda \big( \Delta G_{{\cancel{\pi_k}},k} + \Delta G_{\pi_k,k}  \big) 
&\geq&  (S_{k+1}-R_k)/4 \, . \nonumber
\end{eqnarray}
The announced lower bound follows from the lower bound $s_1$ defined in \eqref{def:s1s2} which satisfies
\begin{eqnarray}
S_{k+1}-R_k \geq s_1 \geq  \inf \left\{ \sigma >0 \, \bigg \vert \, h(\sigma, \Lambda) > \frac{a_\Lambda}{\lambda} \right\} = \frac{\Lambda^2}{2 \ln \left( \lambda \right)} + o(1/\ln \lambda) \, . \nonumber
\end{eqnarray}
\end{proof}

Proposition \ref{prop:0est} directly implies that $T_\infty=\lim_{S \to \infty} \Phi(S)=\infty$, so that $\Phi$ parametrizing a PMF dynamics over the whole half-line $\mathbbm{R}^+$.
We conclude by giving the explicit, iterative construction of explosive PMF dynamics with zero refractory period $\epsilon=0$.

\begin{theorem}
Under Assumption \ref{mainAssump}, PMF dynamics admits the density function 
$p(t,y) = \dd \Prob{ X_t  \leq x \vert X_t > 0} / \dd x$
\begin{eqnarray}
p(t,x) = \sum_{k=0}^{\infty}  \mathbbm{1}_{\{ T_k\le t < T_{k+1} \}} q_k\big(\Phi(t)-T_k, x \big) \, , \nonumber
\end{eqnarray}
where the sequence $(T_k,q_k)$ is specified in Definition  \ref{def:global0} and the time change $\Phi=\Psi^{-1}$ is specified in Definition \ref{def:globaltime}.
\end{theorem}

\subsection{Global continuity in the limit of vanishing refractory periods}

Considering Theorem \ref{th:globSol} for zero refractory period $\epsilon=0$ allows for the explicit construction of explosive PMF dynamics in the large interaction regime. 
To justify that the thus-constructed dynamics are ``physical'', we need to check that these non-delayed solutions can be recovered from dPMF dynamics in the limit of vanishing refractory period $\epsilon \to 0^+$. 
In this perspective, let us consider some natural initial conditions for explosive PMF dynamics that satisfy Assumption \ref{mainAssump}.
Such initial conditions are entirely specified by the spatial component $q_0$, which includes an Dirac-delta mass of size $\pi_0$ at $\Lambda$.
By Theorem \ref{th:globSol}, the corresponding explosive PMF dynamics is fully parametrized by the inverse time-changed function $\Psi$ specified  in Definition \ref{def:global0}.
Given the same natural initial condition,  Theorem \ref{th:globSol} also guarantees the existence and uniqueness of explosive dPMF dynamics for small enough refractory period $\epsilon$ under Assumption \ref{mainAssump}.
Let us denote the corresponding inverse time change by $\Psi_\epsilon$.
By Theorem \ref{th:globSol}, $\Psi_\epsilon$ specifies a countable infinity of jumps with trigger times $\{S_{\epsilon,k} \}_{k \in \mathbbm{N}}$, exit times $\{U_{\epsilon,k} \}_{k \in \mathbbm{N}}$, and reset times $\{R_{\epsilon,k} \}_{k \in \mathbbm{N}}$ such that $0=R_{\epsilon,0}<S_{\epsilon,1}<U_{\epsilon,1}<R_{\epsilon,1}<S_{\epsilon,2}<U_{\epsilon,2}<R_{\epsilon,2}<\dots$. 
Our goal is to justify the following continuity result:

\begin{theorem}\label{thm:epsilonCont}
Given the same purely spatial natural initial condition satisfying Assumption \ref{mainAssump}, the dPMF solution $\Psi_\epsilon$ converges compactly toward the PMF solution $\Psi$ on $\mathbbm{R}^+$ and the dPMF blowup times $S_{\epsilon,k}$, $U_{\epsilon,k}$, $R_{\epsilon,k}$ converge toward their PMF counterparts $S_{k}$ and $U_{k}=R_{k}$ as well.
\end{theorem}

The proof of the above Proposition relies on a simple recurrence argument that makes use of four continuity results, which we first give without proof.
The first set of results, comprising Proposition \ref{prop:convEps1} and  Proposition \ref{prop:convEps2}, deal with the continuity of dPMF dynamics in between blowup episodes in the limit of vanishing refractory period $\epsilon \to 0^+$.
These results essentially follow from the uniform boundedness of the various functions involved in the integral representation of the density of the dPMF dynamics.
Proposition \ref{prop:convEps1} bears on the continuity of the cumulative fluxes, which boils down to the continuity of blowup trigger times with respect to the initial conditions for the $L_1$ norm.

\begin{proposition}\label{prop:convEps1}
Consider  some natural initial conditions $(q_{\epsilon,0},g_{\epsilon,0})$ in $\mathcal{M}((0,\infty)) \times \mathcal{M}([-\xi_{\epsilon,0},0))$ such that $q_{\epsilon,0} \to q_0$ in $L_1$ norm when $\epsilon \to 0^+$, where $q_0$ is a probability density with Dirac-delta mass $\pi_0 \delta_\Lambda$.
Then under Assumption \ref{mainAssump}, we have:

$(1)$  $\lim_{\epsilon \to 0^+}S_{\epsilon,1} = S_1$ where $S_{\epsilon,1}$ and $S_1$ denotes the next blowup trigger times of the time changes $\Psi_\epsilon$ and $\Psi$, respectively.

$(2)$ for all $k$ in $\mathbbm{N}$, the $k$-iterated derivatives of $G_\epsilon$ and $G$ satisfy $G^{(k)}_\epsilon \to G^{(k)}$ uniformly on $[0,S_1+\lambda/3]$.
\end{proposition}

Proposition \ref{prop:convEps1} establishes that PMF dynamics are locally recovered from dPMF dynamics in the limit of vanishing refractory period, at least up to the first blowup episode.
Then, a natural strategy to prove Proposition \ref{thm:epsilonCont} is to apply Proposition \ref{prop:convEps1} iteratively after each blowup episodes in conjunction with some continuity result about the blowup episodes.
It turns out that these continuity results will hold with respect to the $L_1$ norm of the spatial densities of the surviving processes at trigger, exit, and reset times.
Thus, we also need to make sure that the density of surviving processes at blowup trigger time $q_\epsilon(S_{\epsilon,1},\cdot)$ converges toward $q(S_1, \cdot)$ for the $L_1$ norm,
as stated in the following proposition: 
 
  \begin{proposition}\label{prop:convEps2}
Consider  some natural initial conditions $(q_{\epsilon,0},g_{\epsilon,0})$ in $\mathcal{M}((0,\infty)) \times \mathcal{M}([-\xi_{\epsilon,0},0))$ such that $q_{\epsilon,0} \to q_0$ in $L_1$ norm when $\epsilon \to 0^+$, where $q_0$ contains a Dirac-delta mass $\pi_0 \delta_\Lambda$.
Then under Assumption \ref{mainAssump}, we have $\lim_{\epsilon \to 0^+} \Vert q_\epsilon(S_{\epsilon,1}, \cdot) - q(S_1, \cdot) \Vert_1 = 0$.
\end{proposition}

The second set of results, comprising Proposition \ref{prop:convEps1} and  Proposition \ref{prop:convEps2}, bears on the continuity of dPMF and PMF dynamics during blowup episodes.
Proposition \ref{prop:convEps1} states that given a blowup trigger time $S_1$, the next exit blowup time $U_1=S_1+\lambda \pi_1$ and the distribution of surviving processes $q(U_1,\cdot)$ continuously depend on the distribution of surviving processes $q(S_1,\cdot)$ at trigger time with respect to the $L_1$ norm.
Specifically:

\begin{proposition}\label{prop:BlowupCont1} 
Suppose $S_1$ marks the first blowup trigger time after $S_0$ under Assumption \ref{mainAssump}.
Then there exists a neighborhood $\mathcal{N}$ of $q(S_1, \cdot)$ 
such that the maps 
\begin{eqnarray}
q \mapsto \pi_1[q]  \quad \mathrm{and} \quad q \mapsto \left(y \mapsto  \int_0^\infty \kappa(\lambda \pi_1[q], y , x) q(x)\, \dd x \right) \nonumber
\end{eqnarray}
are continuous on $\mathcal{N}$ with respect to the $L_1$ norms.
\end{proposition}

As blowup resolutions involve dynamics without reset, there is no role for the refractory period in establishing the above result, which follows from the Banach-space version of the implicit function theorem.
Proposition \ref{prop:convEps2} directly implies that if $\lim_{\epsilon \to 0^+} \Vert q_\epsilon(S_{\epsilon,1}, \cdot) - q(S_1, \cdot) \Vert_1 = 0$, then
$\lim_{\epsilon \to 0^+} \pi_{\epsilon,1} = \pi_1$ and $\lim_{\epsilon \to 0^+} \Vert q_\epsilon(U_{\epsilon,1}, \cdot) - q(U_1^-, \cdot) \Vert_1 = 0$, where $q(U_1^-, \cdot)$ denotes the distribution of surviving processes for PMF dynamics just before reset at $U_1=R_1$.
In order to propagate these convergence results via recurrence, we then only need to check the $L_1$ convergence of the distribution of processes $q_\epsilon(R_{\epsilon,1},\cdot)$ toward $q(U_1,\cdot)=q(R_1,\cdot)$, allowing one to apply Proposition \ref{prop:convEps1} and \ref{prop:convEps2} anew with $q_\epsilon(R_{\epsilon,1},\cdot)$ and $q(U_1,\cdot)=q(R_1,\cdot)$ as natural  initial conditions.

\begin{proposition}\label{prop:BlowupCont2} 
Under Assumption \ref{mainAssump}, suppose that $R_{\epsilon,1}$ marks the first blowup reset time after $S_0$.
Then we have $\lim_{\epsilon \to 0^+} \Vert q_\epsilon(R_{\epsilon,1}, \cdot) - q(R_1, \cdot) \Vert_1 = 0$.
\end{proposition}

For the sake of completeness, we provide the proof of Proposition \ref{thm:epsilonCont} using Propositions  \ref{prop:convEps1}, \ref{prop:convEps2},  \ref{prop:BlowupCont1}, and  \ref{prop:BlowupCont2}, whose proofs are given in Section Section \ref{sec:smooth}, Section \ref{sec:trigger}, Section \ref{sec:exit}, and \ref{sec:reset}, respectively.

 \begin{proof}[Proof of Theorem \ref{thm:epsilonCont}]
Let us consider the fixed-point solution $\Psi_\epsilon$ for a refractory period $\epsilon> 0$ with purely spatial  natural initial condition $q_0$.
This amounts to considering that all processes are active at an initial reset time $R_{\epsilon,0}= 0$ under Assumption \ref{mainAssump}.
By Theorem \ref{th:globSol}, $\Psi_\epsilon$ presents a countable infinity of jumps with well ordered trigger times $\{S_{\epsilon,k} \}_{k \in \mathbbm{N}}$, exit times $\{U_{\epsilon,k} \}_{k \in \mathbbm{N}}$, and reset times $\{R_{\epsilon,k} \}_{k \in \mathbbm{N}}$ such that $0=R_{\epsilon,0} <S_{\epsilon,1}<U_{\epsilon,1}< R_{\epsilon,1} <S_{\epsilon,2}<U_{\epsilon,2}< R_{\epsilon,2} <\dots$.
Moreover, we have the following facts: 
\begin{enumerate}
\item The post-reset durations to the next trigger time $U_{\epsilon,k}-R_{\epsilon,k}$ are uniformly bounded away from zero by the $\epsilon$-independent constant $s_1$ defined in \eqref{def:s1s2}.
\item
The blowup durations $S_{\epsilon,k}-U_{\epsilon,k}$ satisfy $S_{\epsilon,k}-U_{\epsilon,k} = \lambda \pi_{\epsilon,k}$, where the blowup sizes $\{\pi_{\epsilon,k} \}_{k \in \mathbbm{N}}$ are uniformly bounded away from zero by the $\epsilon$-independent constant $1-C_\Lambda/\lambda>1/2$ where $C_\Lambda$ is chosen according to \eqref{eq:choiceC}.
\item
The post-exit delays to reset times are uniformly bounded by $\nu \epsilon \leq R_{\epsilon,k}-U_{\epsilon,k} \leq 2 \nu \epsilon$.
\end{enumerate}

Consider then the well ordered blow up times $0=R_{0} <S_{1}<U_{1}< R_{1} <S_{2}<U_{2}< R_{2} <\dots$ for the corresponding PMF solution $\Psi$ obtained for the same purely spatial initial condition $q_0$ but with $\epsilon=0$. 
Assuming that $\lim_{\epsilon \to 0^+} R_{\epsilon,k} = R_k$ and that $\lim_{\epsilon \to 0^+}  \Vert q_\epsilon(R_{\epsilon,k}, \cdot) - q_\epsilon(R_{k}, \cdot) \Vert_1=0$, let us show that $\lim_{\epsilon \to 0^+} R_{\epsilon,k+1} = R_{k+1}$ and that $\lim_{\epsilon \to 0^+}  \Vert q_\epsilon(R_{\epsilon,k+1}, \cdot) - q_\epsilon(R_{k+1}, \cdot) \Vert_1=0$, which is the key to propagate our recurrence argument.

$(i)$ \emph{Inter-blowup evolution}: For all $\epsilon>0$,  the densities 
$\big(q_\epsilon(R_{\epsilon,k}, \cdot) , \{ g_\epsilon(\sigma)\}_{\xi_\epsilon(R_{\epsilon,k}) \leq \sigma < R_{\epsilon,k}}\big)$
defines natural initial conditions at blowup reset time $R_{\epsilon,k}$.
Assuming that $q_\epsilon(R_{\epsilon,k}, \cdot)$ converges toward $q_\epsilon(R_{\epsilon,k}, \cdot)$ in $L_1$ norm allows one to invoke Proposition \ref{prop:convEps1}:
 the next blowup trigger time $S_{\epsilon,k+1}$ is such that $\lim_{\epsilon \to 0^+}S_{\epsilon,k+1}-R_{\epsilon,k} = S_{k+1}-R_{k}$ and $\Psi_\epsilon(\cdot + R_{\epsilon,k})$ uniformly converges toward $\Psi(\cdot + R_{k})$ on an interval containing $[0,S_{k+1}-R_{k}+\lambda/3]$.
Moreover, by Proposition \ref{prop:convEps2}, the spatial density at trigger time $q_\epsilon(S_{\epsilon,k+1}, \cdot)$ converges toward $q(S_{k+1}, \cdot)$ in $L_1$ norm.

$(ii)$ \emph{Blowup episode}:  By Proposition \ref{prop:BlowupCont1}, $\lim_{\epsilon \to 0^+} \Vert q_\epsilon(S_{\epsilon,k+1}, \cdot) -q_\epsilon(S_{k+1}, \cdot)\Vert_1 =0$ entails that the next blowup size satisfies $\lim_{\epsilon \to 0^+} \pi_{\epsilon,k+1} = \pi_{k+1}$.
Consequently, since the inverse time changes $\Psi_\epsilon$ are uniformly Lipschitz with respect to $\epsilon>0$, $\Psi_\epsilon(\cdot + R_{\epsilon,k})$ converges uniformly toward $\Psi(\cdot + R_{k})$ on $[0,S_{k+1}-R_{k}+\lambda \pi_{k+1}]$.
Moreover, by Proposition \ref{prop:BlowupCont1}, the spatial density at exit time $q_\epsilon(U_{\epsilon,k+1}, \cdot)$ converges toward $q(U_{k+1}^-, \cdot)$ in $L_1$ norm.
Finally, by Proposition \ref{prop:BlowupCont2}, the spatial density at reset time $q_\epsilon(R_{\epsilon,k+1}, \cdot)$ converges toward $q(R_{k+1}, \cdot)$ in $L_1$ norm, with  $q(R_{k+1}, \cdot)$ still satisfying Assumption \ref{mainAssump}.

As we consider identical purely spatial initial conditions at $R_{\epsilon,0}=R_0=0$: $q_{\epsilon,0}(0,\cdot)=q_{0}(0,\cdot)$, the recurrence naturally initializes so that for all $k$ in $\mathbbm{N}$, the inverse time $\Psi_\epsilon$ uniformly converges toward $\Psi$ on the interval $[0,R_k]$ and the dPMF blowup times converge toward their PMF counterparts: 
$\lim_{\epsilon \to 0^+} (S_{\epsilon,k},U_{\epsilon,k},R_{\epsilon,k}) = (S_k,U_k,R_k)$.
We conclude by observing that $\lim_{k \to \infty} R_{\epsilon,k} \geq \lim_{k \to \infty}  k (\lambda/2+s_1)=\infty$, which establishes the compact convergence of $\Psi_\epsilon$ toward $\Psi$ on $\mathbbm{R}^+$.
\end{proof}

\subsection{Continuity in between blowups}\label{sec:smooth}

\begin{proof}[Proof of Proposition \ref{prop:convEps1}]
The proof proceeds in five steps: 
$(i)$ we show that the backward-delay function $\eta_{\epsilon}$ converges compactly toward zero when $\epsilon \to 0^+$ on  some interval $[0,\underline{S}_1]$;
$(ii)$ we show that the cumulative flux $G_{\epsilon}$ converges compactly  toward the cumulative flux $G$ obtained for $\epsilon=0$ on $[0,\underline{S}_1]$;
$(iii)$ we show that the convergence result of $(ii)$ extends to the iterated derivative of $G_{\epsilon}$;
$(iv)$ using the result from $(iii)$, we show that the interval bound $\underline{S}_1$ defined in $(i)$ coincides with $S_1$, the next blowup trigger time for $\epsilon=0$;
$(v)$ finally, we extend the convergence results to the interval $[0,\underline{S}_1+\lambda/2]$, which partially covers the next blowup episode.

$(i)$ Under Assumption \ref{mainAssump}, for all $\epsilon>0$, the natural initial conditions $(q_{\epsilon,0},g_{\epsilon,0})$ in $\mathcal{M}((0,\infty)) \times \mathcal{M}([-\xi_{\epsilon,0},0))$ are necessarily such that $g_{\epsilon,0}  \leq 1/ (4\lambda) < 1/\lambda$ with $g_\epsilon(0) = \lim_{\sigma \to 0} g_{\epsilon,0}(\sigma) \leq 1/ (4\lambda) < 1/\lambda$ as well.
Therefore, one can consider for sufficiently small $\epsilon>0$, the positive times $\tilde{S}_{\epsilon,1}$ and $S_{\epsilon,1}$
 \begin{eqnarray}
\tilde{S}_{\epsilon,1} = \inf \left\{ \sigma> 0\, \bigg \vert \, g_\epsilon(\sigma) > \frac{1- \sqrt{\nu \epsilon}}{\lambda} \right\} < \inf \left\{ \sigma> 0\, \bigg \vert \, g_\epsilon(\sigma) > \frac{1}{\lambda} \right\}= S_{\epsilon,1}   \, , \nonumber
\end{eqnarray}
which are uniformly bounded above.
The proof of Proposition \ref{prop:nextBlowup} shows that the next blowup trigger time $S_{\epsilon,1}$ coincides with the first time $g_\epsilon$ attains $1/\lambda$.
Moreover, this first time corresponds to a crossing with a slope that is bounded away from zero. 
Specifically, we have $\partial_\sigma g_\epsilon(S_{\epsilon,1})  \geq b_\Lambda /\lambda>0$.
This implies that $\lim_{\epsilon \to 0} \vert S_{\epsilon,1}-\tilde{S}_{\epsilon,1} \vert = 0 $.
Then, by definition of $\tilde{S}_{\epsilon,1}$, we  have $ \Psi'_\epsilon =( 1-\lambda g_\epsilon ) /\nu\geq \sqrt{\epsilon/\nu}$ on $[-\xi_{\epsilon,0},\tilde{S}_{\epsilon,1})$, so that on  $[0,\tilde{S}_{\epsilon,1})$, the corresponding backward-delay function satisfies:
\begin{eqnarray}
\eta_\epsilon(\sigma) 
= \sigma - \Phi_\epsilon(\Psi_\epsilon(\sigma)-\epsilon)
= \Phi_\epsilon(\Psi_\epsilon(\sigma)) - \Phi_\epsilon(\Psi_\epsilon(\sigma)-\epsilon) 
&\leq& \epsilon / \Vert  \Psi'\Vert_{-\xi_{\epsilon,0},\tilde{S}_{\epsilon,1}} \leq \sqrt{\nu \epsilon} \, .\nonumber
\end{eqnarray}
Thus, the backward-delay function $\eta_\epsilon$ and the backward function $\xi_\epsilon$ compactly converge toward zero and $\xi=\mathrm{id}$, respectively, on $[0,\min(S_1,\underline{S}_1))$ with $\underline{S}_1=\liminf_{\epsilon \to 0^+}  \tilde{S}_{\epsilon,1}=\liminf_{\epsilon \to 0^+}  S_{\epsilon,1}$.

$(ii)$ By integration by part of the renewal-type equation \eqref{eq:gRenew}, the cumulative fluxes $G_\epsilon$ and $G$ respectively satisfy:
\begin{eqnarray}\label{eq:gConv}
G_\epsilon(\sigma)
&=&
\int_0^\infty H(\sigma, x)  q_{\epsilon,0}(x) \, \dd x  - H(\sigma, \Lambda) G_\epsilon(\xi_\epsilon(0)) + \\
&& \hspace{40pt} \int_0^\sigma h(\sigma-\tau, \Lambda) G_\epsilon(\xi_\epsilon(\tau))  \, \dd \tau \, , \nonumber\\
G(\sigma) 
&=&
\int_0^\infty H(\sigma, x)  q_{0}(x) \, \dd x  - H(\sigma, \Lambda) G(0) + \\
&& \hspace{40pt} \int_0^\sigma  h(\sigma-\tau, \Lambda) G(\xi(\tau))  \, \dd \tau \, . \nonumber
\end{eqnarray}
For all $0 \leq \sigma \leq \min(S_1,\underline{S}_1)$, we have $\xi(\sigma)=\sigma$.
Using the convention $G(0)=0$, this observation allows one to write
\begin{eqnarray}
\vert G_\epsilon(\sigma) -G(\sigma) \vert
&\leq& 
 \int_0^\infty H(\sigma, x) \big \vert q_{\epsilon,0}(x)-q_{0}(x) \big \vert \, \dd x +  \nonumber\\
&& \hspace{20pt} H(\sigma, \Lambda) \vert G_\epsilon(\xi_\epsilon(0)) \vert + \int_0^\sigma h(\sigma-\tau, \Lambda)  \big \vert  G_\epsilon(\xi_\epsilon(\tau)) -  G(\tau) \big \vert \, \dd \tau \, , \nonumber\\
&\leq& 
\Vert H(\cdot, x) \Vert_\infty \left(\int_0^\infty  \big \vert q_{\epsilon,0}(x)-q_{0}(x) \big \vert \, \dd x + \vert G_\epsilon(\xi_\epsilon(0)) \vert \right)
 + \nonumber\\ 
 \label{eq:ineqGconv}
&& \hspace{20pt} \int_0^\sigma h(\sigma-\tau, \Lambda)  \left( \big \vert  G_\epsilon(\xi_\epsilon(\tau)) -  G_\epsilon(\tau) \big \vert + \big \vert  G_\epsilon(\tau) -  G(\tau) \big \vert  \right)\, \dd \tau \, . \nonumber
\end{eqnarray}
By uniform boundedness of the instantaneous flux $ \Vert g_\epsilon \Vert_\infty \leq A_\Lambda$,  we have 
\begin{eqnarray}
 \vert G_\epsilon(\xi_\epsilon(0)) \vert \leq A_\Lambda \epsilon
 \quad \mathrm{and} \quad
 \big \vert  G_\epsilon(\xi_\epsilon(\tau)) -  G_\epsilon(\tau) \big \vert  \leq  A_\Lambda    \vert  \xi_\epsilon(\tau) -  \tau \big \vert  =A_\Lambda    \vert  \eta_\epsilon(\tau)  \big \vert \, . \nonumber
 \end{eqnarray}
Then, using that  $\Vert H(\cdot, x) \Vert_\infty=1$ and the bound
\begin{eqnarray}
\int_0^\sigma h(\sigma-\tau, \Lambda)  \big \vert  G_\epsilon(\tau) -  G(\tau) \big \vert \, \dd \tau \leq H(\sigma,\Lambda) \Vert G_\epsilon - G \Vert_{0,\sigma} \nonumber \, ,
\end{eqnarray}
we deduce from inequality \eqref{eq:ineqGconv} and the bounds given above that
\begin{eqnarray}
\big(1-H(\sigma,\Lambda) \big)\Vert G_\epsilon - G \Vert_{0,\sigma} 
\leq 
\Vert q_{\epsilon,0}-q_{0} \Vert_1  +  A_\Lambda \left( \epsilon  + H(\sigma,\Lambda) \Vert \eta_\epsilon \Vert_{0,\sigma} \right) \, . \nonumber
\end{eqnarray}
As for all $\sigma \leq \min(S_1,\underline{S}_1)$, we have $H(\sigma, \Lambda) \leq H(S_1, \Lambda)<1$, the above inequality implies the compact convergence of $G_\epsilon$ towards $G$ on  $[0,\min(S_1,\underline{S}_1))$ when $\epsilon \to 0^+$.
This can be extended to the uniform convergence over the closed interval $[0,\min(S_1,\underline{S}_1)]$ by uniform boundedness of $g_\epsilon$.

$(iii)$ Similar calculations but starting from the renewal-type equation \eqref{eq:gRenew} shows that the instantaneous flux $g_\epsilon$ converges uniformly toward $g$, the instantaneous flux with zero refractory period on $[0,\min(S_1,\underline{S}_1))$.
Indeed, by integration by part of  \eqref{eq:gRenew}, we have:
\begin{eqnarray}
g_\epsilon(\sigma) 
&=&
\int_0^\infty h(\sigma, x)  q_{\epsilon,0}(x) \, \dd x  - h(\sigma, \Lambda) G_\epsilon(\xi_\epsilon(0)) + \nonumber\\
&& \hspace{80pt} \int_0^\sigma \partial_\sigma h(\sigma-\tau, \Lambda) G_\epsilon(\xi_\epsilon(\tau))  \, \dd \tau \, , \nonumber\\
g(\sigma) 
&=&
\int_0^\infty h(\sigma, x)  q_{0}(x) \, \dd x  - h(\sigma, \Lambda) G(\xi(0)) + \nonumber\\
&& \hspace{80pt} \int_0^\sigma \partial_\sigma h(\sigma-\tau, \Lambda) G(\tau)  \, \dd \tau \, ,\nonumber
\end{eqnarray}
where we use that $\xi(\sigma)=\sigma$ for $\sigma<\min(S_1,\underline{S}_1)$.
The integrability and uniform boundedness of the integration kernels appearing above directly imply the announced uniform convergence.
Similar arguments can be made about the equations obtained by further differentiation of \eqref{eq:gRenew} with respect to $\sigma$, showing the uniform convergence of all derivatives of $g_\epsilon$ toward the corresponding derivatives of $g$ on $[0,\min(S_1,\underline{S}_1)]$ when $\epsilon \to 0^+$.

$(iv)$ The observations above allow one to show $(1)$ in Proposition \ref{prop:convEps1}, or equivalently that
\begin{eqnarray}
\lim_{\epsilon \to 0^+} S_{\epsilon,1}= S_1 = \inf \left\{ \sigma> 0\, \Big \vert \, g(\sigma) > 1/\lambda \right\} \, . \nonumber
\end{eqnarray}
This amounts to show $\underline{S}_1=\overline{S}_1=S_1$, where $\underline{S}_1= \liminf_{\epsilon \to 0^+} S_{\epsilon,1}$ and $\overline{S}_1= \limsup_{\epsilon \to 0^+} S_{\epsilon,1}$.
Let us first show that $S_1 \leq \underline{S}_1$.
By definition, there is a sequence $\{ \epsilon_n \}_{n \in \mathbbm{N}}$ with $\epsilon_n \to 0^+$ and such that:
\begin{eqnarray}
\lim_{n \to \infty} S_{\epsilon_n} =  \underline{S}_1 \, , \quad g_{\epsilon_n}(S_{\epsilon_n,1}) = 1/\lambda \, , \quad \mathrm{and} \quad  \partial_\sigma  g_{\epsilon_n} \, .\nonumber(S_{\epsilon_n,1})\geq b_\Lambda /\lambda>0 \, .
\end{eqnarray}
Thus, by the uniform convergence results of $(iii)$, we have $g(\underline{S}_1) =1/\lambda$ and  $\partial_\sigma g(\underline{S}_1)  \geq b_\Lambda /\lambda>0$, which implies that $\underline{S}_1$ is a crossing time of $g$ at level $1/\lambda$.
As $S_1$ is defined as the first such crossing time, we have $\underline{S}_1 \geq S_1$.
To show that $\underline{S}_1 = S_1$, let us consider $\overline{S}_1 \geq \underline{S}_1 \geq S_1$ and let us show that $\overline{S}_1=S_1$.
Suppose $\overline{S}_1>S_1$.
Then, there is $S_1<\sigma <\overline{S}_1$ such that $g(\sigma)>1/\lambda$.
Posit $\delta =g(\sigma)-1/\lambda$. 
By pointwise convergence  in $\sigma$, there is $e>0$ such that for all $0<\epsilon<e$, $\vert g_{\epsilon}(\sigma)-g(\sigma) \vert<\delta/2$.
Consider now a sequence $\{ \epsilon_n \}_{n \in \mathbbm{N}}$ with $\epsilon_n \to 0^+$ and such that $\lim_{n \to \infty} S_{\epsilon_n,1} = \overline{S}$.
For $n$ large enough, we have $\epsilon_n<e$ so that
\begin{eqnarray}
g_{\epsilon_n}(\sigma) > g(\sigma) - \delta/2 = (g(\sigma)+1/\lambda)/2 > 1/\lambda \, ,  \nonumber
\end{eqnarray}
so that the first crossing $S_{\epsilon_n,1}$ must happen before $\sigma$.
But then $S_{\epsilon_n,1}<\sigma<\overline{S}_1$ contradicts the definition of $\overline{S}_1$ as the limit of $S_{\epsilon_n,1}$.
Thus we must have $\underline{S}_1 =  \overline{S}_1 = \lim_{\epsilon \to 0^+} S_{\epsilon,1} = S_1$.
This shows $(1)$ in Proposition \ref{prop:convEps1}.

$(v)$
Finally, observe that  $(2)$ in Proposition \ref{prop:convEps1} will directly follows from extending the uniform convergence results shown in $(ii)$ and $(iii)$ to $[0,S_1+\lambda/2]$.
To prove such an extension, observe that by Proposition \ref{prop:persistence}, the full blowup triggered in $S_{\epsilon,1}$ with refractory period $\epsilon>0$ and in $S_{1}$ for zero refractory period all have size $\pi_{\epsilon,1}>1/2$.
Thus, the backward functions $\xi_\epsilon$ and $\xi$ must be constant on $[S_{\epsilon,1}, S_{\epsilon,1}+\lambda/2]$ and $[S_1, S_1+\lambda/2]$, so that we have:
\begin{eqnarray}
S_1 \leq \sigma \leq S_{\epsilon,1}+\lambda/2: \quad G_\epsilon(\sigma)
&=&
\int_0^\infty H(\sigma, x)  q_{\epsilon,0}(x) \, \dd x  - H(\sigma, \Lambda) G_\epsilon(\xi_\epsilon(0)) + \nonumber\\
&& \hspace{40pt} \int_0^{\min(\sigma,S_{1,\epsilon})}h(\sigma-\tau, \Lambda) G_\epsilon(\xi_\epsilon(\tau))  \, \dd \tau \, , \nonumber\\
S_1 \leq \sigma \leq S_1+\lambda/2: \quad G(\sigma) 
&=&
\int_0^\infty H(\sigma, x)  q_{0}(x) \, \dd x  - H(\sigma, \Lambda) G(0) + \nonumber\\
&& \hspace{40pt} \int_0^{S_{1}}  h(\sigma-\tau, \Lambda) G(\tau)  \, \dd \tau \, , \nonumber
\end{eqnarray}
where we use  that $\xi=\mathrm{id}$ for all $0 \leq \sigma \leq S_1$.
From there, given $(ii)$, the only additional point to check, is that for all $0 \leq \sigma \leq S_1+\lambda/2$, the extra term
\begin{eqnarray}
I(\sigma)= \int_{S_1}^{\min(\sigma,S_{1,\epsilon})} h(\sigma-\tau, \Lambda) G_\epsilon(\xi_\epsilon(\tau))  \, \dd \tau \nonumber
\end{eqnarray}
vanishes uniformly over $[S_1, S_1+\lambda/3]$.
This directly follows from the uniform boundedness of $G_\epsilon$:
\begin{eqnarray}
I(\sigma) &\leq& \Vert h( \cdot, \Lambda )\Vert_\infty  \Vert G_\epsilon \circ \xi_\epsilon\Vert_{S_1,S_{1,\epsilon}} \vert S_{\epsilon,1} - S_1\vert \, \nonumber\\
&\leq& \Vert h( \cdot, \Lambda )\Vert_\infty  A_\Lambda(S_1+\lambda/3) \vert S_{\epsilon,1} - S_1\vert \xrightarrow[\epsilon \to 0^+]{} 0 \, . \nonumber
\end{eqnarray}
\end{proof}

\subsection{Continuity at blowup trigger time}\label{sec:trigger}

\begin{proof}[Proof of Proposition \ref{prop:convEps2}]
By Lemma \ref{lem:xiBound}, there is a constant $s_1$ independent of $\epsilon$ such that $0 \leq \xi_\epsilon'(\sigma) \leq 1$ for all $S_{\epsilon,1} -s_1/2\leq \sigma \leq S_{\epsilon,1}+ \lambda/2$. 
Introducing the real number $U=S_1-s_1/3$ and using that 
\begin{eqnarray}
q(S_{\epsilon,1}, y) &=& \int_0^\infty \kappa(S_{\epsilon,1},y,x) p_{\epsilon,0}(x) \, \dd x + \int_0^{S_{\epsilon,1}}  \kappa(S_{\epsilon,1}-\tau, y, \Lambda) \, \dd G_\epsilon(\xi_\epsilon(\tau)) \, , \nonumber\\
q(S_1,y) &=& \int_0^\infty \kappa(S_1,y,x) p_0(x) \, \dd x + \int_0^{S_1}  \kappa(S_1-\tau, y, \Lambda) \, \dd G(\tau) \, , \nonumber
\end{eqnarray}
we upper bound $\Delta q = \Vert q(S_{\epsilon,1}, \cdot)-q(S_1, \cdot)  \Vert_1$ as
\begin{eqnarray}
\Delta q  \leq \Delta q_1 + \Delta q_2+ \Delta q_3 + \Delta q_4 + \Delta q_5 \, , \nonumber
\end{eqnarray}
where the terms $\Delta q_1$, $\Delta q_2$, $\Delta q_3$, $\Delta q_4$, and $\Delta q_5$ are defined as follows:
\begin{eqnarray}
\Delta q_1 = \int_0^\infty \bigg \vert \int_0^\infty \kappa(S_{1,\epsilon},y,x) p_{\epsilon,0}(x) \, \dd x - \int_0^\infty \kappa(S_1,y,x) p_0(x) \, \dd x \bigg \vert \, \dd y  \, ,\nonumber
\end{eqnarray}
\begin{eqnarray}
\Delta q_2 
= 
\int_0^\infty \bigg \vert \int_0^U  \kappa(S_1-\tau, y, \Lambda) \, \dd G_\epsilon(\xi_\epsilon(\tau)) - \int_0^U \kappa(S_1-\tau, y, \Lambda) \, \dd G(\tau) \bigg \vert \, \dd y  \, ,\nonumber
\end{eqnarray}
\begin{eqnarray}
\Delta q_3 
= 
\int_0^\infty \bigg \vert \int_0^U  \kappa(S_{\epsilon,1}-\tau, y, \Lambda) \, \dd G_\epsilon(\xi_\epsilon(\tau)) - \int_0^U \kappa(S_1-\tau, y, \Lambda) \, \dd G_\epsilon(\xi_\epsilon(\tau)) \bigg \vert \, \dd y  \, ,\nonumber
\end{eqnarray}
\begin{eqnarray}
\Delta q_4
= 
\int_0^\infty \bigg \vert \int_U ^{S_1}  \kappa(S_{\epsilon,1}-\tau, y, \Lambda) \, \dd G_\epsilon(\xi_\epsilon(\tau)) - \int_U^{S_1}  \kappa(S_1-\tau, y, \Lambda) \, \dd G(\tau) \bigg \vert \, \dd y  \, , \nonumber
\end{eqnarray}
\begin{eqnarray}
\Delta q_5
= 
\int_0^\infty \bigg \vert \int_{S_1} ^{S_{\epsilon,1}}  \kappa(S_{\epsilon,1}-\tau, y, \Lambda) \, \dd G_\epsilon(\xi_\epsilon(\tau))  \bigg \vert \, \dd y  \, . \nonumber
\end{eqnarray}
Our goal is to show that the terms $\Delta q_2$, $\Delta q_3$, $\Delta q_4$, and $\Delta q_5$ all vanish in the limit $\epsilon \to 0^+$.
The nonrenewal term $\Delta q_1$  can be bounded leveraging the fact $\lim_{\epsilon \to 0^+} S_{\epsilon,1}=S_1>0$, which allows one to use the bound provided by Lemma \ref{lem:supbound}: 
\begin{eqnarray}
\Delta q_1 
&& 
\leq
\int_0^\infty  \bigg \vert \int_0^\infty \kappa(S_{1,\epsilon},y,x) \big(p_{\epsilon,0}(x) - p_0(x) \big) \, \dd x  \bigg \vert  \, \dd y \; + \nonumber\\
&& \hspace{20pt}
\int_0^\infty  \bigg \vert \int_0^\infty \big(\kappa(S_{1,\epsilon},y,x) - \kappa(S_1,y,x) \big) p_0(x)  \, \dd x  \bigg \vert \, \dd y \, , \nonumber\\
&& 
\leq
 \int_0^\infty  \left( \int_0^\infty \kappa(S_{1,\epsilon},y,x)\, \dd y \right)  \big \vert p_{\epsilon,0}(x) - p_0(x) \big \vert   \, \dd x \; + \nonumber\\
&& \hspace{20pt} 
\int_0^\infty \left( \int_0^\infty  \sup_{S_1/2 \leq \sigma \leq S_1+\lambda/3}   \vert  \partial_\sigma \kappa(\sigma,y,x) \vert \, \dd y \right) \vert S_{1,\epsilon} - S_1 \vert p_0(x)  \, \dd x  
 \, , \nonumber\\
 &&
 \leq
 \Vert p_{\epsilon,0} - p_0 \Vert_1
+
M_{S_1/2, S_1+\lambda/3} \vert S_{1,\epsilon} -S_1 \vert   \nonumber
 \, .
\end{eqnarray}
Thus $\lim_{\epsilon \to 0^+} \Delta q_1 =0$ directly follows from the $L_1$ convergence  $\Vert p_{\epsilon,0} - p_0 \Vert_1 \to 0$ as $\epsilon \to 0^+$.
The bound provided by Lemma \ref{lem:supbound} can be similarly used for the renewal term $\Delta q_2$ which has been defined up to time $U<S_1$.
To see this, let us use integration by parts on $[0,U]$ to write
\begin{eqnarray}
\Delta q_2 
&&
\leq
\int_0^\infty \Bigg \vert  \int_0^U \partial_ \sigma \kappa(S_1-\tau, y, \Lambda) \big( G_\epsilon(\xi_\epsilon(\tau)) - G(\tau) \Big) \, \dd \tau \bigg \vert \, \dd y + \nonumber\\
&& 
\hspace{20pt} \int_0^\infty \bigg \vert \Big[ \kappa(S_1-\tau, y, \Lambda) \big(G_\epsilon(\xi_\epsilon(\tau))-G(\tau)  \big)  \Big]_0^U \bigg \vert \, \dd y \, , \nonumber\\
&&
\leq
\Vert G_\epsilon \circ \xi_\epsilon - G  \Vert_{0, U} \int_0^U  \left( \int_0^\infty  \big \vert  \partial_ \sigma \kappa(S_1-\tau, y, \Lambda)  \big \vert \, \dd y \right)\, \dd \tau  + \nonumber\\
&& 
\hspace{20pt}  \Big \vert G_\epsilon(\xi_\epsilon(U))-G(U)    \Big \vert + \Big \vert G_\epsilon(\xi_\epsilon(0))-G(0)    \Big \vert  \nonumber \, .
\end{eqnarray}
Using the bound from Lemma \ref{lem:supbound}, we get
\begin{eqnarray}
\Delta q_2 
\leq
(2+M_{S_1-U,S_1}) \, \Vert G_\epsilon \circ \xi_\epsilon - G  \Vert_{0, U}  \, ,
\nonumber
\end{eqnarray}
and $\lim_{\epsilon \to 0^+} \Delta q_2 =0$ follows from the uniform convergence of $G_\epsilon \circ \xi_\epsilon$ toward $G$ on $[0,S_1+\lambda/3]$ given in Proposition \ref{prop:convEps1}.
The boundary renewal term $\Delta q_3$ can also be bounded using the fact that $S_{\epsilon,1}>(S_1+U)/2>U$ for small enough $\epsilon>0$:
\begin{eqnarray}
\Delta q_3 
&& 
\leq
\int_0^U \left( \int_0^\infty \sup_{(S_1+U)/2 \leq \sigma \leq S_1+\lambda/3} \big \vert \partial_\sigma \kappa(\sigma-\tau, y, \Lambda)\big \vert \, \dd y \right) \vert S_{\epsilon,1} - S_1\vert \, \dd G_\epsilon(\xi_\epsilon(\tau))   \nonumber\\
&& 
\leq
M_{(S_1-U)/2,S_1+\lambda.3} \big( G_\epsilon(\xi_\epsilon(U))-G_\epsilon(\xi_\epsilon(0)) \big) \vert S_{\epsilon,1} - S_1\vert   \nonumber \, .
\end{eqnarray}
We conclude that $\lim_{\epsilon \to 0^+} \Delta q_3 =0$ by boundedness of the cumulative fluxes:
\begin{eqnarray}
\vert G_\epsilon(\xi_\epsilon(U))-G_\epsilon(\xi_\epsilon(0)) \vert
\leq
\vert G_\epsilon(\xi_\epsilon(U)) \vert + \vert G_\epsilon(\xi_\epsilon(0)) \vert 
\leq 
 G_\epsilon(U) + 1
\leq A_\Lambda U + 1 \, .\nonumber 
\end{eqnarray}
The remaining terms $\Delta q_4 $ and $\Delta q_5$ collect the contribution of those processes that reset at the vicinity of the blowup trigger times $S_1$ and $S_{\epsilon,1}$.
Bounding these terms will crucially rely on the fact that under Assumption \ref{mainAssump}, the backward function $\xi_\epsilon$ can be shown to be well behaved on $[U,S_1]$ in the sense that by Lemma \ref{lem:xiBound}, we have $0 \leq \xi_\epsilon' \leq 1$.
Bearing in mind this preliminary remark, observe that for all small enough $\delta>0$, we have
\begin{eqnarray}
\Delta q_4 
&& 
\leq
\int_U^{S_1}  \left( \int_0^\infty  \kappa(S_{\epsilon,1}-\tau, y, \Lambda) \, \dd y  \right) \, \big \vert \dd G_\epsilon(\xi_\epsilon(\tau)) - \dd G(\tau) \big \vert + \nonumber \\
&& \hspace{20pt} 
\int_0^\infty \int_U^{S_1} \big \vert \kappa(S_{\epsilon,1}-\tau, y, \Lambda)- \kappa(S_1-\tau, y, \Lambda)\big \vert \, \dd G(\tau)  \, \dd y \, , \nonumber \\
&& 
\leq
\int_U^{S_1}  \big \vert \dd G_\epsilon(\xi_\epsilon(\tau)) - \dd G(\tau) \big \vert  + \nonumber \\
&& \hspace{20pt} 
\int_0^\infty \int_{S_1-\delta}^{S_1} \big \vert \kappa(S_{\epsilon,1}-\tau, y, \Lambda)- \kappa(S_1-\tau, y, \Lambda)\big \vert \, \dd G(\tau)  \, \dd y \; + \nonumber \\
&& \hspace{20pt} \int_0^\infty \int_{U}^{S_1-\delta} \big \vert \kappa(S_{\epsilon,1}-\tau, y, \Lambda)- \kappa(S_1-\tau, y, \Lambda)\big \vert \, \dd G(\tau)  \, \dd y \, , \nonumber \\
&&
\leq
\int_U^{S_1}  \big \vert \xi_\epsilon'(\tau) g_\epsilon(\xi_\epsilon(\tau)) -   g(\tau) \big \vert  \, \dd \tau + \nonumber \\
&& \hspace{20pt} 
\Vert g \Vert_\infty \delta  \left( \int_0^\infty  \big \vert \kappa(S_{\epsilon,1}-\tau, y, \Lambda) \big\vert + \big \vert \kappa(S_1-\tau, y, \Lambda)\big \vert  \, \dd y \right) \; + \nonumber \\
&& \hspace{20pt}  M_{\delta/2,S_1+\lambda/3} \big( G(S_1-\delta)-G(U) \big) \vert S_{\epsilon,1} - S_1\vert  , \nonumber \\
&& 
\leq
\int_U^{S_1}  \big \vert \xi_\epsilon'(\tau) g_\epsilon(\xi_\epsilon(\tau)) -   g(\tau) \big \vert  \, \dd \tau + 2 A_\Lambda  \delta  + M_{\delta/2,S_1+\lambda/3} A_\Lambda S_1 \vert S_{\epsilon,1} - S_1\vert  \nonumber \, .
\end{eqnarray}
By Lemma \ref{lem:xiBound}, $\xi'_\epsilon(\tau) g_\epsilon(\xi_\epsilon(\tau)) \leq A_\Lambda$ over $[U,S_1]$, whereas for all $\sigma$ in $[U,S_1)$, we have the pointwise convergence
\begin{eqnarray}
\xi_\epsilon'(\sigma)
=
\frac{\Psi_\epsilon'(\sigma)}{\Psi_\epsilon'(\xi_\epsilon(\sigma))}
=
\frac{1-\lambda g_\epsilon(\sigma)}{1- \lambda g_\epsilon(\xi_\epsilon(\sigma))} \xrightarrow[\epsilon \to 0^+]{} 1 \, .
\nonumber
\end{eqnarray}
Thus by dominated convergence, we have
\begin{eqnarray}
\lim_{\epsilon \to 0^+} \int_U^{S_1}  \big \vert \xi'(\tau) g_\epsilon(\xi_\epsilon(\tau)) -   g(\tau) \big \vert \, \dd \tau= 0 \, .
\nonumber
\end{eqnarray}
Then, as $\lim_{\epsilon \to 0^+} S_{\epsilon,1}=S_1$, we deduce that $\limsup_{\epsilon \to 0^+} \Delta q_4 \leq 2 A_\Lambda \delta$ for all small enough $\delta>0$, which implies that $\lim_{\epsilon \to 0^+} \Delta q_4=0$.
To bound the final term $\Delta q_5$, we also exploit that  $\xi_\epsilon'$ is bounded by one on  $[S_1-s_1/\lambda,S_1+\lambda/3]$ so that
\begin{eqnarray}
\Delta q_5
&=& 
\int_0^\infty \bigg \vert \int_{S_1} ^{S_{\epsilon,1}}  \kappa(S_{\epsilon,1}-\tau, y, \Lambda) \xi'_\epsilon(\tau) g_\epsilon(\xi_\epsilon(\tau))  \, \dd \tau \bigg \vert \, \dd y  \, . \nonumber \\
&\leq&
 \int_{S_1} ^{S_{\epsilon,1}} \left( \int_0^\infty  \kappa(S_{\epsilon,1}-\tau, y, \Lambda) \, \dd y  \right)  g_\epsilon(\xi_\epsilon(\tau))  \, \dd \tau  \, . \nonumber \\
 &\leq&
 \int_{S_1} ^{S_{\epsilon,1}}   g_\epsilon(\xi_\epsilon(\tau))  \, \dd \tau 
 \leq
A_\Lambda \Vert S_{\epsilon,1 }-S_1 \vert \xrightarrow[\epsilon \to 0^+]{} 0 \, . \nonumber
\end{eqnarray}
\end{proof}

\subsection{Continuity at blowup exit time}\label{sec:exit}

 \begin{proof}[Proof of Proposition \ref{prop:BlowupCont1}]
 The jump size $\pi_1$ is defined as the solution of  the equation $\mathcal{F}(\pi_1,q(S_1, \cdot))=0$, where $\mathcal{F}$ is the map
\begin{eqnarray}
\mathcal{F}[p,q] =  p - \int_{0^+}^\infty H(\lambda p, x)  q(x) \, \dd x   \, . \nonumber
\end{eqnarray}
By smoothness of $H(\cdot, x)$  and by linearity with respect to $q$, the mapping $\mathcal{F}$ is continuously Fr\'echet  differentiable on $\mathbbm{R} \times \mathcal{M}([0,\infty))$ where $\mathcal{M}([0,\infty))$  is equipped with the $L_1$ norm.
Moreover, under Assumption \ref{mainAssump}, we have
\begin{eqnarray}
\partial_p \mathcal{F}[\pi_1,q(S_1, \cdot)] = 1- \lambda \int_{0^+}^\infty h(\lambda \pi_1, x)  q(S_1, x) \, \dd x = 1 - \lambda g(U_1) \, ,  \nonumber
\end{eqnarray}
where $g(U_1)$ denotes the instantaneous flux at the blowup exit time $U_1=S_1+\lambda \pi_1$.
By Proposition \ref{prop:gBound}, under Assumption \ref{mainAssump}, we have $g(U_1)\leq 1/(4 \lambda)$ so that $\partial_p \mathcal{F}[\pi_1,q(S_1, \cdot)]>3/4$ is  bounded away from zero.
Therefore, by the Banach-space version of the implicit function theorem, there is a neighborhood of $q(S_1, \cdot)$ such that the mapping $q \mapsto \pi_1[q]$ is Fr\'echet differentiable.
Moreover, under Assumption \ref{mainAssump}, we necessarily have that $\pi_1[q]>1/2$ is bounded away from zero.
The latter observation allows one to exploit the continuity of the mapping $q \mapsto \pi_1[q]$ in a neighborhood $\mathcal{N}$ of $q_0(S_1, \cdot)$ to show the continuity of the map
\begin{eqnarray}
q \mapsto   \mathcal{G}[q]=\int_0^\infty\kappa(\lambda \pi_1[q], \cdot , x) q(x)\, \dd x \, , \nonumber
\end{eqnarray}
on that same neighborhood with respect to the $L_1$ norm on $\mathcal{M}([0,\infty))$.
To prove this point, consider $q_a$ and $q_b$ in $\mathcal{N}$ and introduce the shorthand notations $\pi_a=\pi_1[q_a]>1/2$ and $\pi_b=\pi_1[q_b]>1/2$. We have
\begin{eqnarray}
\Vert  \mathcal{G}[q_a] -  \mathcal{G}[q_b]\Vert_1 
&=& \int_0^\infty \bigg \vert \int_0^\infty \big(\kappa(\lambda \pi_a,y , x) q_a(x) - \kappa(\lambda \pi_b,y , x) \big) q_b(x)\, \dd x \bigg \vert \, \dd y \, ,  \nonumber\\
&\leq& \int_0^\infty \int_0^\infty  \kappa(\lambda \pi_a, y , x) \big \vert q_a(x) - q_b(x)  \big \vert \, \dd x \, \dd y  \nonumber\\
 && + \int_0^\infty \int_0^\infty \big \vert \kappa(\lambda \pi_a, y , x) - \kappa(\lambda \pi_b, y , x)) \big \vert q_b(x)\, \dd x \, \dd y \, ,  \nonumber\\
 &\leq& \int_0^\infty( 1- H(\lambda \pi_a, x) ) \big \vert q_a(x) - q_b(x)  \big \vert \, \dd x  \nonumber\\
 &&+ \int_0^\infty \int_0^\infty \left( \int_{\lambda \pi_a}^{\lambda \pi_b} \big \vert  \partial_\sigma \kappa(\sigma, y , x)  \big \vert \, \dd \sigma \right) q_b(x)\, \dd x \, \dd y \, , \ \nonumber \\
 &\leq&  \Vert q_a - q_b  \Vert_1  \,  \nonumber\\
  && \hspace{-10pt}+  \left(   \int_0^\infty    \left(  \int_0^\infty  \sup_{\lambda \pi_a \leq \sigma \leq \lambda \pi_b} \big \vert  \partial_\sigma \kappa(\sigma, y , x)  \big \vert  \, \dd y  \right)  q_b(x)\, \dd x \right)\lambda \vert \pi_b -\pi_a \vert   \nonumber \, .
\end{eqnarray}
As $0<1/2<\pi_a,\pi_b<1/2<1$, by Lemma \ref{lem:supbound}, there exists a finite upper bound
\begin{eqnarray}
M_{ab, \lambda} = \sup_{x>0} \left( \int_0^\infty \sup_{\lambda \pi_a \leq \sigma \leq \lambda \pi_b} \vert  \partial_\sigma \kappa(\sigma,y,x) \vert \, \dd y \right) < \infty \, .  \nonumber
\end{eqnarray}
For fixed $\lambda$, the desired continuity result follows from the continuity of the map $q \mapsto \pi_1[q]$ via
\begin{eqnarray}
\Vert  \mathcal{G}[q_a] -  \mathcal{G}[q_b]\Vert_1 
\leq
\Vert q_a - q_b  \Vert_1 + M_{ab} \lambda \big \vert \pi_1[q_a]- \pi_1[q_b] \big \vert \, . \nonumber
\end{eqnarray}
\end{proof}

\subsection{Continuity at blowup reset time}\label{sec:reset}

 \begin{proof}[Proof of Proposition \ref{prop:BlowupCont2}]
Under Assumption \ref{mainAssump}, the reset dPMF dynamics resumes after the blowup exit time $U_{\epsilon,1}$ until the next blowup time $S_{\epsilon,2}$.
By the well ordered property of blowups, the reset time $R_{\epsilon,1}$ is such that $R_{\epsilon,1}<S_{\epsilon,2}$ and the backward-delay function $\xi_\epsilon$ is smooth on $[U_{\epsilon,1},R_{\epsilon,1})$.
Thus, the distribution of surviving processes just before reset is given by
\begin{eqnarray}\label{eq:resetInt1}
q_\epsilon(R_{\epsilon,1}^-, y)
&=&
\int_0^\infty \kappa(R_{\epsilon,1}^- -U_{\epsilon,1},y,x) q_\epsilon(U_{\epsilon,1}, x) \, \dd x + \nonumber\\
&& \quad \int_{U_{\epsilon,1}}^{R_{\epsilon,1}^-} \kappa(R_{\epsilon,1}^--\sigma,y,\Lambda)  \xi'_\epsilon(\sigma) g_\epsilon(\xi_\epsilon(\sigma))\, \dd \sigma \, .
\end{eqnarray}

$(i)$ We first show that the $L_1$ norm of the second integral term in \eqref{eq:resetInt1} due to reset processes vanishes in the limit $\epsilon \to 0^+$.
By definition of the blowup trigger and reset times, when $\sigma \to R_{\epsilon,1}^-$, we have $\xi_\epsilon(\sigma) \to S_{\epsilon,1}^-$.
Moreover, the inverse time change $\Psi_\epsilon$ is differentiable at $S_{\epsilon,1}$ with derivative $\Psi_\epsilon'(S_{\epsilon,1}) \leq 3/(4 \nu)$, whereas the time change $\Psi_\epsilon$ becomes locally flat with $\Psi'_\epsilon(S_{\epsilon,1})=0$ and $\Psi''_\epsilon(S_{\epsilon,1}) \leq -b_\Lambda/(2 \nu)$.
The latter observations imply that for  $\sigma \leq R_{\epsilon,1}$, we have
\begin{eqnarray}
\Psi_\epsilon(\sigma) -\epsilon &=& \Psi_\epsilon(R_{\epsilon,1})-\epsilon+\Psi'_\epsilon(R_{\epsilon,1}) (\sigma-R_{\epsilon,1})^2+O\big( \sigma-R_{\epsilon,1})^2 \big) \, , \\
\Psi_\epsilon(\xi_\epsilon(\sigma)) &=& \Psi_\epsilon(S_{\epsilon,1})+\frac{\Psi''_\epsilon(S_{\epsilon,1})}{2} (\xi_\epsilon(\sigma)-S_{\epsilon,1})^2+O\big( (\xi_\epsilon(\sigma)-S_{\epsilon,1})^3 \big) \, .
 \end{eqnarray}
Using the definition $\xi_\epsilon(\sigma) = \Phi_\epsilon(\Psi_\epsilon(\sigma)-\epsilon)$ to equate the two quantities above leads to the asymptotic scaling
\begin{eqnarray}
\xi_\epsilon(\sigma) - S_{\epsilon,1} \sim  \sqrt{\frac{2 \Psi'_\epsilon(R_{\epsilon,1})(R_{\epsilon,1}-\sigma)}{\vert \Psi''_\epsilon(S_{\epsilon,1}) \vert}} \, ,
\end{eqnarray}
so that $\xi'_\epsilon$ is uniformly integrable over $[U_{\epsilon,1},R_{\epsilon,1})$.
By uniform boundedness of the instantaneous flux $g_{\epsilon} \leq A_\Lambda$, this implies that the $L_1$ norm of the second integral terms in \eqref{eq:resetInt1} vanishes when $\epsilon \to 0^+$:
\begin{eqnarray}
&&
\int_0^\infty  \bigg \vert \int_{U_{\epsilon,1}}^{R_{\epsilon,1}} \kappa(R_{\epsilon,1}-\sigma,y,\Lambda)  \xi'_\epsilon(\sigma) g_\epsilon(\xi_\epsilon(\sigma)) \, \dd \sigma \bigg \vert\, \dd y \, , \\
&& \hspace{30pt} \leq
 \int_{U_{\epsilon,1}}^{R_{\epsilon,1}} \left( \int_0^\infty \kappa(R_{\epsilon,1}-\sigma,y,\Lambda) \, \dd y  \right) \big \vert \xi'_\epsilon(\sigma)\big \vert g_\epsilon(\xi_\epsilon(\sigma)) \, \dd \sigma \, , \\
&& \hspace{30pt} \leq
A_\Lambda \int_{U_{\epsilon,1}}^{R_{\epsilon,1}}  \big \vert \xi'_\epsilon(\sigma)\big \vert  \, \dd \sigma 
\; = \;
O\big(\sqrt{R_{\epsilon,1}-U_{\epsilon,1}}\big)
\xrightarrow{\epsilon \to 0^+} 0 \, ,
\end{eqnarray}
which follows from the fact that $\vert R_{\epsilon,1}-U_{\epsilon,1} \vert \leq 2 \nu \epsilon$.

$(ii)$ To show the announced result, it is thus enough to restrict the analysis the first integral term in \eqref{eq:resetInt1} due to surviving processes at blowup exit times.
This term results from the same absorption dynamics as during a blowup episode starting at $S_{\epsilon,1}$, so that we have
\begin{eqnarray}
&&  \int_0^\infty \bigg \vert q_\epsilon(U_{\epsilon,1}, y) - \int_0^\infty \kappa(R_{\epsilon,1}-U_{\epsilon,1},y,x) q_\epsilon(U_{\epsilon,1}, x) \, \dd x  \bigg \vert \, \dd y \, , \\
&&  \hspace{30pt} =\int_0^\infty  \bigg \vert \int_0^\infty \kappa(U_{\epsilon,1}-S_{\epsilon,1},y,x) - \kappa(R_{\epsilon,1}-S_{\epsilon,1},y,x) q_\epsilon(S_{\epsilon,1}, x) \, \dd x  \bigg \vert \, \dd y \, , \\
&& \hspace{30pt} \leq  \int_0^\infty  \left( \int_0^\infty \sup_{U_{\epsilon,1} \leq \sigma \leq R_{\epsilon,1}} \vert \partial_\sigma \kappa(\sigma-S_{\epsilon,1},y,x) \vert \, \dd y  \right) \, \vert U_{\epsilon,1} -R_{\epsilon,1}\vert \, q_\epsilon(S_{\epsilon,1}, x) \, \dd x \, , \\
&& \hspace{30pt} \leq  M_{U_{\epsilon,1}-S_{\epsilon,1},R_{\epsilon,1}-S_{\epsilon,1}} \vert U_{\epsilon,1} -R_{\epsilon,1}\vert \, , 
\end{eqnarray}
where the upper bound from Lemma \ref{lem:supbound} satisfies $M_{U_{\epsilon,1}-S_{\epsilon,1},R_{\epsilon,1}-S_{\epsilon,1}} \leq M_{\lambda/2,\lambda+2 \nu \epsilon}$. The latter bound is a decreasing function of $\epsilon$, and thus uniformly bounded with respect to $\epsilon$, so that we have
\begin{eqnarray}
\lim_{\epsilon \to 0^+} \big \Vert q_\epsilon(R_{\epsilon,1}^-, \cdot) - q_\epsilon(U_{\epsilon,1}, \cdot) \big \Vert_1 = 0 \, .
\end{eqnarray}
As we  know that $q_\epsilon(R_{\epsilon,1}^-, \cdot)$ also converges in $L_1$ norm toward $q_\epsilon(U_1^-, \cdot)$ at blowup exit time, the triangular inequality implies that
\begin{eqnarray}
\lim_{\epsilon \to 0^+} \big \Vert q_\epsilon(R_{\epsilon,1}^-, \cdot) - q(U_1^-, \cdot) \big \Vert_1 = 0 \, .
\end{eqnarray}
Finally, splitting the delta Dirac contribution from the pre-reset continuous density part
\begin{eqnarray}
\big \Vert q_\epsilon(R_{\epsilon,1}, \cdot) - q(U_1, \cdot) \big \Vert_1
= \vert \pi_{\epsilon,1} - \pi_1 \vert  + \big \Vert q_\epsilon(R_{\epsilon,1}^-, \cdot) - q(U_1^-, \cdot) \big \Vert_1 \, , 
\end{eqnarray}
we conclude by remembering that we have $\lim_{\epsilon \to 0^+} \pi_{\epsilon,1} = \pi_1$.
\end{proof}

\section*{Acknowledgements}
The authors would like to thank Phillip Whitman and Luyan Yu for insightful discussions and  for simulation results.

The second author was supported by an Alfred P. Sloan Research Fellowship FG-2017-9554 and a CRCNS award DMS-2113213 from the National Science Foundation.


\begin{appendix}

\section{}\label{appA}

This appendix contains a lemma that is useful to control the $L_1$ norm of the density of surviving processes in the limit of vanishing refractory periods $\epsilon \to 0^+$.
It applies to integral terms bearing on reset processes. 

\begin{lemma}\label{lem:xiBound}
Under Assumption \ref{mainAssump}, for small enough refractory period $\epsilon$, we have $\xi'_\epsilon \leq 1$ on $[S_{\epsilon,1}-s_1/2,S_{\epsilon,1}+\lambda/2]$,  where $s_1>0$  only depends on $\Lambda$ and $\lambda$.
\end{lemma} 

\begin{proof} 
Under Assumption \ref{mainAssump}, by $(iii)$ in the proof of Proposition \ref{prop:nextBlowup}, we have 
\begin{eqnarray}\label{eq:startP}
S_{\epsilon,1} - s_1 \leq \sigma \leq S_{\epsilon,1} \quad  \Rightarrow \quad \left(\partial_\sigma g_\epsilon(\sigma) \geq \frac{b_\Lambda}{2 \lambda} \quad \Leftrightarrow  \quad \Psi_\epsilon''(\sigma) \leq \frac{b_\Lambda}{2\nu} \right)   \, ,
\end{eqnarray}
where the constant $s_1$ defined by \eqref{def:s1s2} only depends on $\Lambda$ and $\lambda$.
Thus, for all $S_{\epsilon,1} - s_1 \leq \sigma \leq S_{\epsilon,1}$, we have
\begin{eqnarray}
\Psi_\epsilon'(\sigma) = \Psi_\epsilon'(\sigma) -\Psi_\epsilon'(S_{1,\epsilon})  = -\int_\sigma ^{S_{1,\epsilon}} \Psi_\epsilon''(\tau) \, \dd \tau \geq \frac{b_\Lambda}{2\nu } (S_{1,\epsilon}-\sigma) \, . \nonumber
\end{eqnarray}
In turn, integrating once more  for all $S_{\epsilon,1} - s_1 \leq \xi < \sigma \leq S_{\epsilon,1}$ yields
\begin{eqnarray}
\Psi_\epsilon(\sigma) - \Psi_\epsilon(\xi) \geq  \frac{b_\Lambda}{2\nu} \left( S_{\epsilon,1} (\sigma-\xi) + (\sigma^2-\xi^2)/2 \right) \, . \nonumber
\end{eqnarray}
The backward time $\xi(\sigma)$ is such that $ \Psi_\epsilon(\sigma)-\Psi_\epsilon(\xi(\sigma)) =\epsilon$, which implies that we necessarily have
\begin{eqnarray}
\epsilon \geq  \frac{b_\Lambda}{2\nu} \left( S_{\epsilon,1} (\sigma-\xi(\sigma)) + (\sigma^2-\xi(\sigma)^2)/2 \right) \, . \nonumber
\end{eqnarray}
The above inequality provides us with a lower bound on admissible value for $\xi(\sigma)$.
Solving for $\xi(\sigma)$ reveals that the inequality holds with $\xi(\sigma) \leq \sigma$ if and only if
\begin{eqnarray}
\xi(\sigma) \leq S_{1,\epsilon} - \sqrt{\frac{4 \nu \epsilon}{ b_\Lambda}+ (S_{\epsilon,1}-\sigma)^2} \, .\nonumber
\end{eqnarray}
This shows that for small enough $\epsilon$
\begin{eqnarray}
\sigma > S_{1,\epsilon}- \sqrt{s_1^2-\frac{4 \nu \epsilon}{ b_\Lambda}} \quad  \Rightarrow \quad S_{1,\epsilon}-\xi(\sigma)<s_1 \, .\nonumber
\end{eqnarray}
In particular, if $\epsilon < 3 s_1^2  b_\Lambda/(16\nu)$, we have $S_{\epsilon,1}-s_1/2\leq \sigma \leq S_{\epsilon,1}$ implies that $S_{\epsilon,1}-s_1 < \xi(\sigma)<\sigma$.
Then the result follows $(1)$ from remembering that by \eqref{eq:startP}, $\Psi_\epsilon'$ is increasing on $[S_{\epsilon,1}-s_1,S_{\epsilon,1}]$ so that:
\begin{eqnarray}
S_{\epsilon,1} - \frac{s_1}{2} \leq \sigma \leq S_{\epsilon,1} \quad  \Rightarrow \quad  \xi_\epsilon'(\sigma) = \frac{\Psi_\epsilon'(\sigma)}{\Psi'(\xi_\epsilon(\sigma))} \leq 1 \, , \nonumber
\end{eqnarray}
and $(2)$ from observing that under Assumption \ref{mainAssump}, the blowup size is such that $\pi_1 \geq 1/2$, so that $\Psi_\epsilon'(\sigma)=0$ and $\Psi_\epsilon'(\xi_\epsilon(\sigma))=\Psi_\epsilon'(\xi_\epsilon(S_{\epsilon,1}))>0$ for all $S_{\epsilon,1} \leq \sigma \leq S_{\epsilon,1} + \lambda/2$
\end{proof}

\section{}\label{appB}

This appendix contains another lemma that is useful to control the $L_1$ norm of the density of surviving processes in the limit of vanishing refractory periods $\epsilon \to 0^+$.
It applies to integral terms bearing on active processes.

\begin{lemma}\label{lem:supbound}
For all $0<\sigma_a<\sigma_b<\infty$, there exists a finite upper bound
\begin{eqnarray}
M_{ab} = \sup_{x>0} \left( \int_0^\infty  \sup_{\sigma_a \leq \sigma \leq \sigma_b} \vert  \partial_\sigma \kappa(\sigma,y,x) \vert \, \dd y \right) < \infty \, .
\nonumber
\end{eqnarray}
\end{lemma}

\begin{proof}
Differentiating expression \eqref{eq:kappaDef} for $\kappa(\sigma,y,x)$  with respect to $\sigma$ yields:
\begin{eqnarray}
\vert  \partial_\sigma \kappa(\sigma,y,x) \vert 
&=&
 \frac{e^{-\frac{(y-x+\sigma)^2}{2 \sigma}}}{2 \sqrt{2 \pi \sigma^5}} \Big \vert  \sigma(1+\sigma)\left( 1-e^{-\frac{2 x y}{\sigma}}\right) + (y-x)^2-(y+x)^2 e^{-\frac{2 x y }{\sigma}} \Big \vert \, , \nonumber\\
 &\leq& 
 \frac{e^{-\frac{(y-x+\sigma)^2}{2 \sigma}}}{2 \sqrt{2 \pi \sigma^5}} \left(   \sigma(1+\sigma)+ \Big \vert (y-x)^2\left( 1-e^{-\frac{2 x y }{\sigma}}\right)-4xy e^{-\frac{2 x y }{\sigma}} \Big \vert \right) \, ,\nonumber \\
 &\leq& 
 \frac{e^{-\frac{(y-x+\sigma)^2}{2 \sigma}}}{2 \sqrt{2 \pi \sigma^5}} \left(  \sigma(1+\sigma)+  (y-x)^2  + 4xy e^{-\frac{2 x y }{\sigma}}  \right) \, , \nonumber \\
 &\leq& 
 \frac{e^{-\frac{(y-x)^2}{2 \sigma} -(y-x)-\frac{\sigma}{2}}}{2 \sqrt{2 \pi \sigma^5}} \left(  \sigma(1+\sigma)+  (y-x)^2  + 2\sigma/e  \right) \, ,  \nonumber
\end{eqnarray}
where the last inequality follows from the fact that $\Vert u e^{-u} \Vert_{0,\infty} \leq 1/e$.
From there, we have
\begin{eqnarray}
 \sup_{\pi_a \leq p \leq \pi_b} \big \vert  \partial_\sigma \kappa(\lambda p, y , x) \vert
 &\leq&
  \frac{e^{-\frac{(y-x)^2}{2 \sigma_b} -(y-x)-\frac{\sigma_a}{2}}}{2 \sqrt{2 \pi \sigma_a^5}} \left(  \sigma_b(1+2/e+\sigma_b)+  (y-x)^2   \right) \, ,  \nonumber\\
 &\leq&
 e^{-\frac{(y-x+\sigma_b)^2}{2 \sigma_b}}  \big(A_{ab}+B_{ab}(y-x)^2\big)\, , \nonumber
\end{eqnarray}
where we define the constants:
\begin{eqnarray}
A_{ab} =  \frac{e^{\frac{\sigma_b-\sigma_a}{2}}}{2 \sqrt{2 \pi \sigma_a^5}}   \sigma_b(1+2/e+\sigma_b)  \quad \mathrm{and} \quad
B_{ab} =  \frac{e^{\frac{\sigma_b-\sigma_a}{2}}}{2 \sqrt{2 \pi \sigma_a^5}} \, .  \nonumber
\end{eqnarray}
Integration with respect to $y$ then yields the finite upper bound:
\begin{eqnarray}
\int_0^\infty  \sup_{\sigma_a \leq \sigma \leq \sigma_b} && \!\!\! \vert  \partial_\sigma \kappa(\sigma,y,x) \vert \, \dd y \nonumber \\
&\leq&
M_{a,b} = A_{ab} \sqrt{\frac{\pi}{2}} \sigma_b + B_{ab}\left((1+\sigma_b) \sqrt{\frac{\pi}{2} \sigma_b^3}  + e^{\frac{-(x-\sigma_b)^2}{2\sigma_b}} \sigma_b (\sigma_b+x)\right)  \nonumber \, .
\end{eqnarray}
\end{proof}

\end{appendix}

\bibliographystyle{imsart-number}

\end{document}